\documentclass[10pt,twoside]{article}

\usepackage{amssymb}
\usepackage{amsbsy} 
\usepackage{amstext}
\usepackage{enumerate}
\usepackage{verbatim} %incluye el entorno comment
\usepackage{graphicx}
\usepackage{theorem}
\newtheorem{theor}{Theorem}[section]
\newtheorem{prop}[theor]{Proposition}
\newtheorem{coro}[theor]{Corollary}
\newtheorem{lemma}[theor]{Lemma}

\newtheorem{remark}[theor]{Remark}

\newenvironment{proof}
	{\par {\bf Proof:}}
 	{\hfill $\square$ \medskip}

\newcommand{\slim}{\mathop{\mbox{\rm s-lim}}}

\def\mf{\mathfrak}
\def\mc{\mathcal}
\def\ds{\displaystyle}
\def\be{\begin{equation}}
\def\ee{\end{equation}}

\def\CC{{\mathfrak C}}
\def\H{{\mathcal H}}
\def\M{{\mathfrak M}}

\def\N{\mathbb N}
\def\Z{\mathbb Z}
\def\R{\mathbb R}
\def\C{\mathbb C}
\def\D{\mathbb D}
\def\I{\mathbb I}
\def\J{\mathbb J}

\def\la{\lambda}
\def\om{\omega}
\def\<{\langle}
\def\>{\rangle}

\title{Univariate tight wavelet frames of minimal support}
\author{F. G\'omez-Cubillo$^1$, S. Villullas$^2$}

\begin{document}

\maketitle
\begin{center}
{\small $^1$Dpto de An\'alisis Matem\'atico, Instituto de Investigaci\'on en Matem\'aticas, Universidad de Valladolid, Facultad de Ciencias, 47011 Valladolid, Spain. fgcubill@am.uva.es.} 

{\small $^2$Dpto de Econom\'{\i}a, Universidad Carlos III de Madrid,
C/Madrid 126, 28903 Getafe (Madrid), Spain. svillull@uc3m.es.}
\end{center}

\begin{abstract}
Wavelet frames for $L^2({\mathbb R})$ can be characterized by means of  spectral techniques. This work uses spectral formulas to determine all the tight wavelet frames for $L^2(\R)$ with a fixed finite number of generators of minimal support. The method associates wavelet frames of this type with certain inner operator-valued functions in Hardy spaces. The cases with one and two generators are completely solved.

\smallskip
%\noindent
{\bf Keywords:} wavelet frames, Hardy spaces, inner functions

\smallskip
%\noindent
{\bf 2010 MSC:} 42C15, 30J05
\end{abstract}

%-------------------------------
\section{Introduction}\label{s1}

Let $T$ and $D$ be the translation and (dyadic) dilation unitary operators on $L^2(\R)$ defined by 
\be\label{tdo}
[Tf](x):=f(x-1)\,,\quad [ D f](x):=2^{1/2}\,f(2x)\,,\quad  (f\in L^2(\R)).
\ee 
Given a finite or countable subset $\Psi$ of $L^2(\R)$, the (dyadic) {\bf  wavelet system} of $L^2(\R)$ generated by $\Psi$ is of the form
\be\label{ws}
X=X_\Psi:=\big\{\psi_{k,j}:=D^kT^j\psi: \psi\in\Psi,\,k,j\in\Z\big\}.
\ee
$X_\Psi$ is called a {\bf wavelet frame} for $L^2(\R)$  if there exist constants $A,B>0$ such that
\be\label{c5.2}
A\,||f||^2\leq \sum_{\psi\in\Psi}\sum_{k,j\in\Z} |\<f,\psi_{k,j}\>|^2\leq B\,||f||^2,\quad (f\in L^2(\R)).
\ee
If, in addition, $A=B$, it is said that $X_\Psi$ is a {\bf tight wavelet frame}. 
The frame property (\ref{c5.2}) confers on $X_\Psi$ good properties for analysis and synthesis in $L^2(\R)$.

A first paper \cite{GV19-1} of this series characterizes the wavelet frames for $L^2({\mathbb R})$ by means of {\it spectral techniques} and presents the usual {\it extension principles} of the theory in terms of the periodized Fourier transform.
Since the introduction of the extension principles \cite{RS97b,RS97a,CHS02,DRS03}, the main part of the literature devoted to the construction of wavelet frames uses them looking for the corresponding {\it framelet filter banks} and paying attention to properties like vanishing moments, symmetry, number of generators,  support, etc.
In the univariate case see, e.g., \cite{RS97a,CH00,CHS02,CHS02-1,DRS03,SA04,HM05,Han13,Han14,CKK14,Han15}. The multivariate case in $L^2(\R^d)$ is discussed in, e.g., \cite{LS06,SAZ16,FJS16,HL17,HJSZ18}.
  
In this work we use the spectral formulas obtained in \cite{GV19-1} to calculate all the tight wavelet frames for $L^2(\R)$ with a fixed finite number (say $r$) of generators of minimal support.  Like in \cite{GS11b}, Hardy classes of vector-valued functions and operator-valued inner functions play a central role.
We solve explicitly the cases $r=1$ and $r=2$.
To our knowledge, there are no papers in the literature studying this problem.
In the context of extension principles, the tight framelet filter banks with $r=1$ are characterized in \cite[Theorem 7]{Han15}. For $r=2$, Theorem 4.2 in \cite{Han13} gives the tight framelet filter banks with complex symmetry and other partial results can be found in, e.g., \cite{RS97a,CH00,CHS02,DRS03,SA04}.
In particular, the case with B-splines as refinable functions has been extensively
studied; see, e.g., \cite[Section 4.4]{FJS16} for details.

Section \ref{sect2} below introduces the terminology and notation used in spectral methods necessary along the paper. We also recall the spectral  characterization of tight wavelet frames for $L^2(\R)$ given in \cite[Corollary 3.6]{GV19-1}.

Section \ref{sect4} shows how \cite[Corollary 3.6]{GV19-1} permits us to determine all the tight wavelet frames for $L^2(\R)$ with a fixed number of generators of minimal support. 
Like in \cite{GS11b}, Hardy classes \cite{RR85,R62,HALMOS61} play a central role here.
In particular, operator-valued functions called {\it rigid Taylor operator functions}  by Halmos \cite{HALMOS61} and {\it $M^+$-inner functions} by Rosenblum and Rovnyak \cite{RR85}. See section \ref{sect41} below for details.
Roughly speaking,
Halmos lemma \ref{lwr}, Rovnyak lemma \ref{lRovnyak} and proposition \ref{prophh} imply the following result: 
\begin{quote}
{\it Let $X_\Psi$ be a wavelet system of the form (\ref{ws}), with cardinal of $\Psi$ finite, say $r$, and such that the support of each $\psi\in\Psi$ is included in the interval $[0,1]$.
Then, $X_\Psi$ is a tight wavelet frame for $L^2(\R)$ if and only if $\Psi$ is associated with an $M^+$-inner $(r\times r)$-matrix function $A^+(\om)$ satisfying certain properties.
} 
\end{quote}
This result comes from a particular choice of orthonormal bases 
$\big\{L_{i}^{(n)}\big\}$ and $\big\{K_{s,j}^{(m)}\big\}$ in the spectral method, the Haar orthonormal bases given in the appendix, and the corresponding  distribution of indices in the set of equations (\ref{ses1}) --see proposition \ref{prophv}, in particular, table \ref{table1}--. 

We discuss the cases $r=1$ and $r=2$ in sections \ref{sect42} and \ref{sect43}, respectively.
For $r=1$ the solution is given in corollary \ref{coro17}: 
\begin{quote}
{\it The only function $\psi\in L^2[0,1]$ such that the wavelet system $X_\Psi$ of the form (\ref{ws}) generated by $\Psi=\{\psi\}$ is a tight frame for $L^2(\R)$, with frame bound $B$, is proportional to the Haar wavelet:
$$
\psi=\beta\,[\chi_{[0,1/2)}-\chi_{[1/2,1)}],
$$
where $\beta\in\C$ and $|\beta|^2=B$. It is associated with the constant $M^+$-inner scalar function
$$
A^+:\partial \D\to\C:\om\mapsto \beta/|\beta|\,.
$$
}
\end{quote}
This result invalidates theorem 5 in \cite{GS11b}. 
See remark \ref{rm18} for more details.

For $r=2$, i.e., $\Psi=\{\psi_1,\psi_2\}\subset L^2[0,1]$, the $M^+$-inner $(2\times 2)$-matrix functions $A^+(\om)$ of interest appear in proposition \ref{pclc2a} and the final solution is given in propositions \ref{prop22} and \ref{prop23}:
\begin{quote}
{\it  
There are five types of families of $M^+$-inner ($2\times 2$)-matrix functions
$$
A^+(\om)= \left(\begin{array}{cc} {\mf a}^{(0)}_1(\om) & {\mf a}^{(1)}_1(\om)\\ {\mf a}^{(0)}_2(\om)& {\mf a}^{(1)}_2(\om)\end{array}\right)=
\left(\begin{array}{cc} {\mf a}^{(0)}(\om) & {\mf a}^{(1)}(\om)\end{array}\right)\,,\quad\text{ where } {\mf a}^{(0)},{\mf a}^{(1)}\in H^+_{\C^2}\,,
$$
leading to tight wavelet frames $X_\Psi$ of the form (\ref{ws}) generated by $\Psi$. They are as follows: 
\begin{itemize}
\item[Type 1.]
Given $u_0\in\C^2$, such that $||u_0||_{\C^2}=1$,
$$
\left\{\begin{array}{l}
\ds {\mf a}^{(0)}(\om)=u_0\,,
\\
{\mf a}^{(1)}(\om)=0\,.
\end{array}\right.
$$
\item[Type 2.]
Given an orthonormal basis $\{u_0,u_1\}$ of $\C^2$, 
$$
\left\{\begin{array}{l}
\ds {\mf a}^{(0)}(\om)=\om\,u_0\,,
\\
{\mf a}^{(1)}(\om)=u_1\,.
\end{array}\right.
$$
\item[Type 3.]
Given an orthonormal basis $\{u_0,u_1\}$ of $\C^2$, $0<\rho<1$ and $\theta\in\R$,
$$
\left\{\begin{array}{l}
\ds {\mf a}^{(0)}(\om)=\rho u_0+\om(1-\rho^2)^{1/2}u_1\,,
\\
\ds {\mf a}^{(1)}(\om)=e^{i\theta}[(1-\rho^2)^{1/2}u_0-\om\,\rho u_1]\,.
\end{array}\right.
$$
\item[Type 4.]
Given an orthonormal basis $\{u_0,u_1\}$ of $\C^2$, $0<\rho<1$ and $\theta\in\R$, 
$$
\left\{\begin{array}{l}
\ds {\mf a}^{(0)}(\om)=\big(\rho+(1-\rho^2)e^{i\theta}\sum_{k=1}^\infty\om^k\,\big(-\rho\,e^{i\theta}\big)^{k-1}\big)\,u_0\,,
\\
\ds {\mf a}^{(1)}(\om)=u_1\,.
\end{array}\right.
$$
\item[Type 5.]
Given $0<|\rho_0|<1$, choose three unitary vectors $u_0$, $u_1$ and $v$ in $\C^2$ such that (\ref{c2n0c3n0-5}) is satisfied.
Then, $|\rho_1|$ and $|\tau_0|$ are given by (\ref{c2n0c3n0-11}). Once the free arguments for $\rho_0$, $\rho_1$ and $\tau_0$ have been selected, say $\theta_{\rho_0}$, $\theta_{\rho_1}$ and $\theta_{\tau_0}$, the value of $r$ is determined by (\ref{c2n0c3n0-8}) and the value of $\tau_1$ is given by (\ref{c2n0c3n0-11}). Then,
$$
\left\{\begin{array}{l}
\ds {\mf a}^{(0)}(\om)=\rho_0u_0+\rho_1v\sum_{k=1}^\infty\om^k\,r^{k-1}\,,
\\
\ds {\mf a}^{(1)}(\om)=\tau_0u_1+\tau_1v\sum_{k=1}^\infty\om^k\,r^{k-1}\,.
\end{array}\right.
$$
\end{itemize}
}
\end{quote}

In order to obtain $\psi_1$ and $\psi_2$ from the function $A^+(\om)$ one must proceed in the following way. 
Since $\text{supp}\,\psi_j\subseteq [0,1]$, ($j=1,2$), their expansions (\ref{fdobl2}) read
$$
\psi_j=\sum_{i\in\N\cup\{0\}} [\hat\psi_j]_i^{(0)}L_i^{(0)}\,,\quad  (j=1,2)\,.
$$
Let us write 
$$
\Psi_i:=\left(\begin{array}{c} [\hat\psi_1]_i^{(0)} \cr [\hat\psi_2]_i^{(0)}\end{array}\right)\in\C^2\,,\quad (i\in\N\cup\{0\})\,.
$$
Then, for a frame bound $B>0$,
$$
\begin{array}{l}
\ds \sum_{k=0}^\infty \om^k\,\Psi_{2^k}=B_0\,{\mf a}^{(0)}(\om)\,,
\\[2ex]
\ds \sum_{k=0}^\infty \om^{k}\,\Psi_{2^{p+k+1}+l}=C_l\,{\mf a}^{(1)}(\om)\,,\quad (l\geq1,\,l=2^p+\sum_{t=0}^{p-1}l_t2^t)\,,
\end{array}
$$
where $B_0\in\C$ is such that $|B_0|^2=B$ and 
the sequence $\{C_l\}_{l\in\N}\subset\C$ is given by (\ref{c1x})--(\ref{c21x}) and (\ref{clg}).

Clearly, for type 1 functions $A^+(\om)$, both functions $\psi_1$ and $\psi_2$ are proportional to the Haar wavelet. For types 2--5 functions $A^+(\om)$, some examples of real functions $\psi_1$ and $\psi_2$ are shown in figures \ref{fig1}--\ref{fig4}.
As it can be seen in figures \ref{fig2} and \ref{fig3},  type 3 functions $A^+(\om)$ lead to the reflected version of functions $\psi_1$ and $\psi_2$ obtained from  type 4 functions $A^+(\om)$ with the same parameters.

Given an $M^+$-inner ($2\times 2$)-matrix function $A^+(\om)$ of type 2--5, also $B^{+}(\om)=A^{+}(\om) \cdot U$ is an $M^+$-inner ($2\times 2$)-matrix function of type 2--5, for every constant unitary ($2\times 2$)-matrix $U$. This relationship establishes connections between $M^+$-inner ($2\times 2$)-matrix functions of types 2 and 3 on the one hand and $M^+$-inner ($2\times 2$)-matrix functions of types 4 and 5 on the other hand. See proposition \ref{prop24} and comments that follow it. 
Thus, starting from a simple tight wavelet frame of type 2, one can obtain a frame of type 3 by means of a unitary matrix $U$, to get a frame of type 4 by reflection and, finally, to reach a frame of type 5 using again a unitary matrix $U'$.

%------
%\noindent
\begin{figure}
\footnotesize%\tiny
\begin{center}
\begin{tabular}{cccc}
\includegraphics[width=0.23\linewidth]{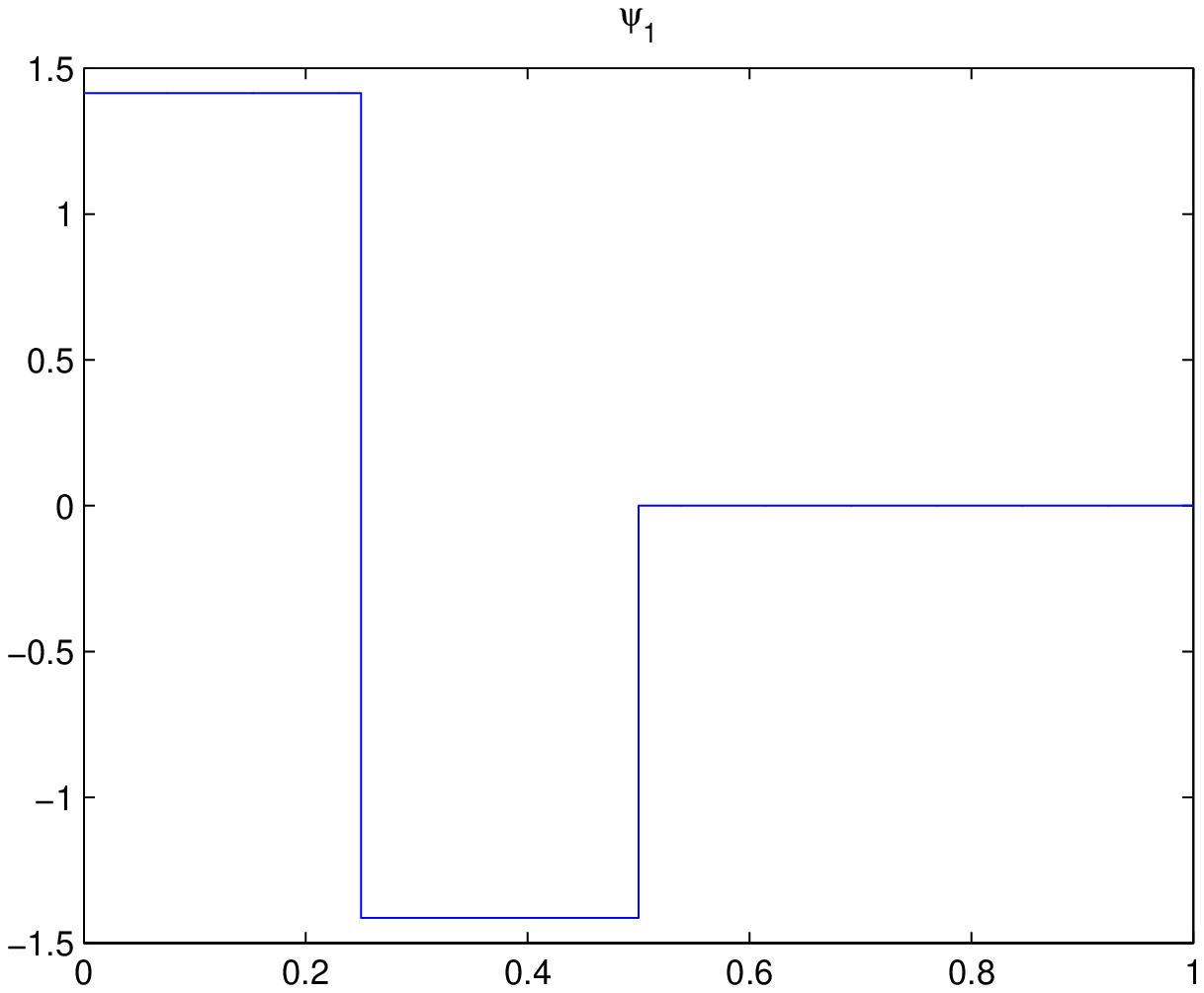} &
\includegraphics[width=0.23\linewidth]{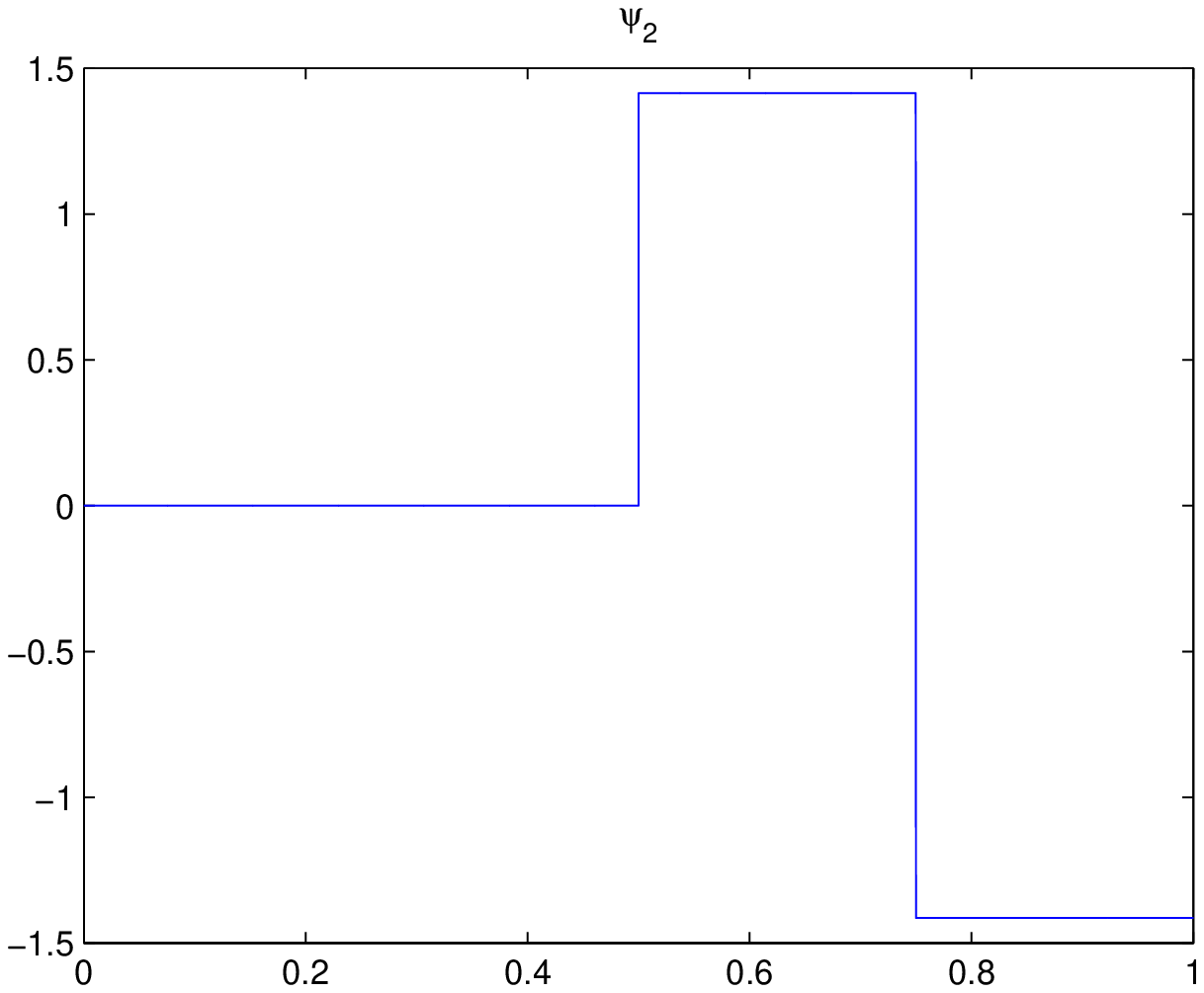} &
\includegraphics[width=0.23\linewidth]{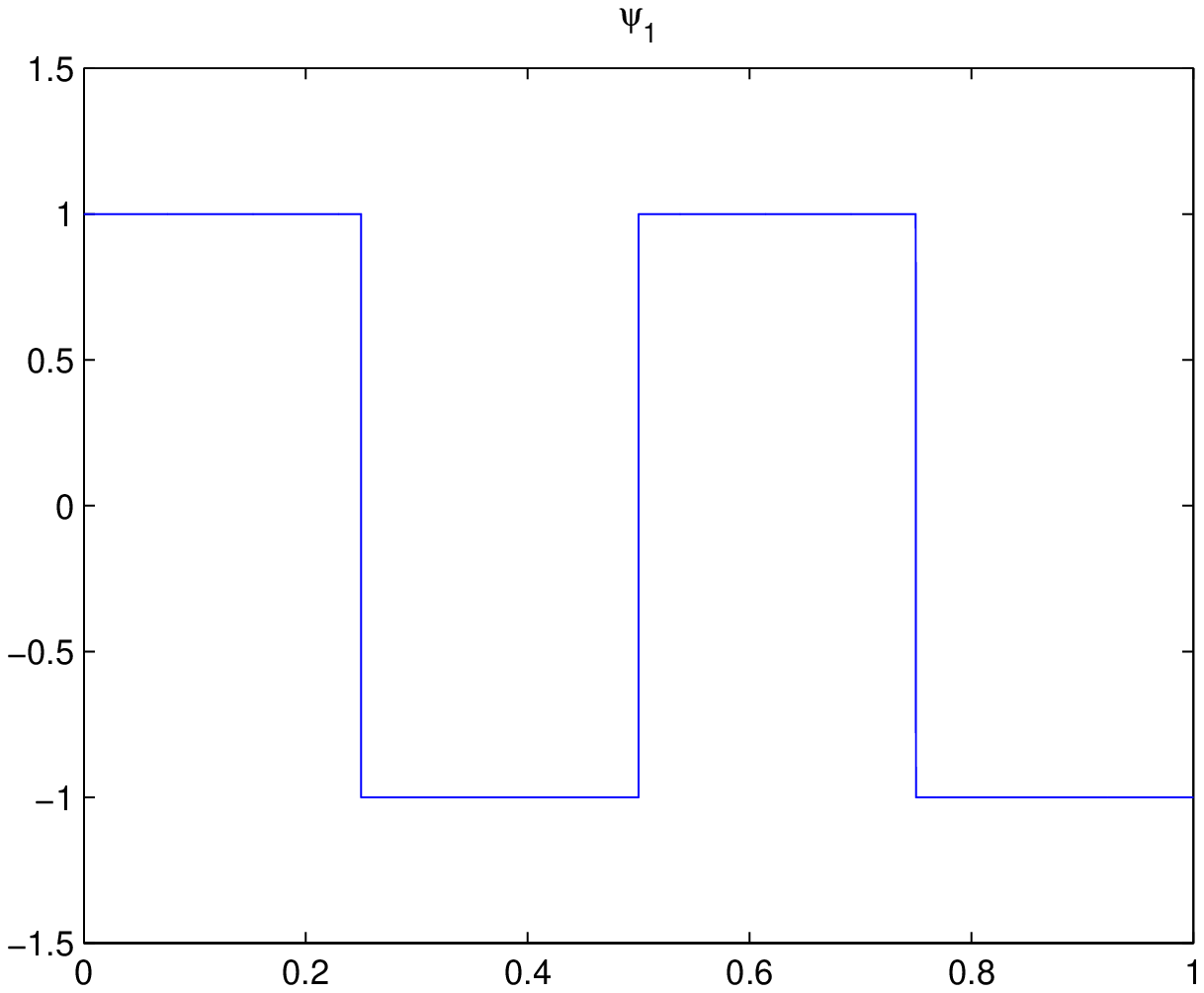} &
\includegraphics[width=0.23\linewidth]{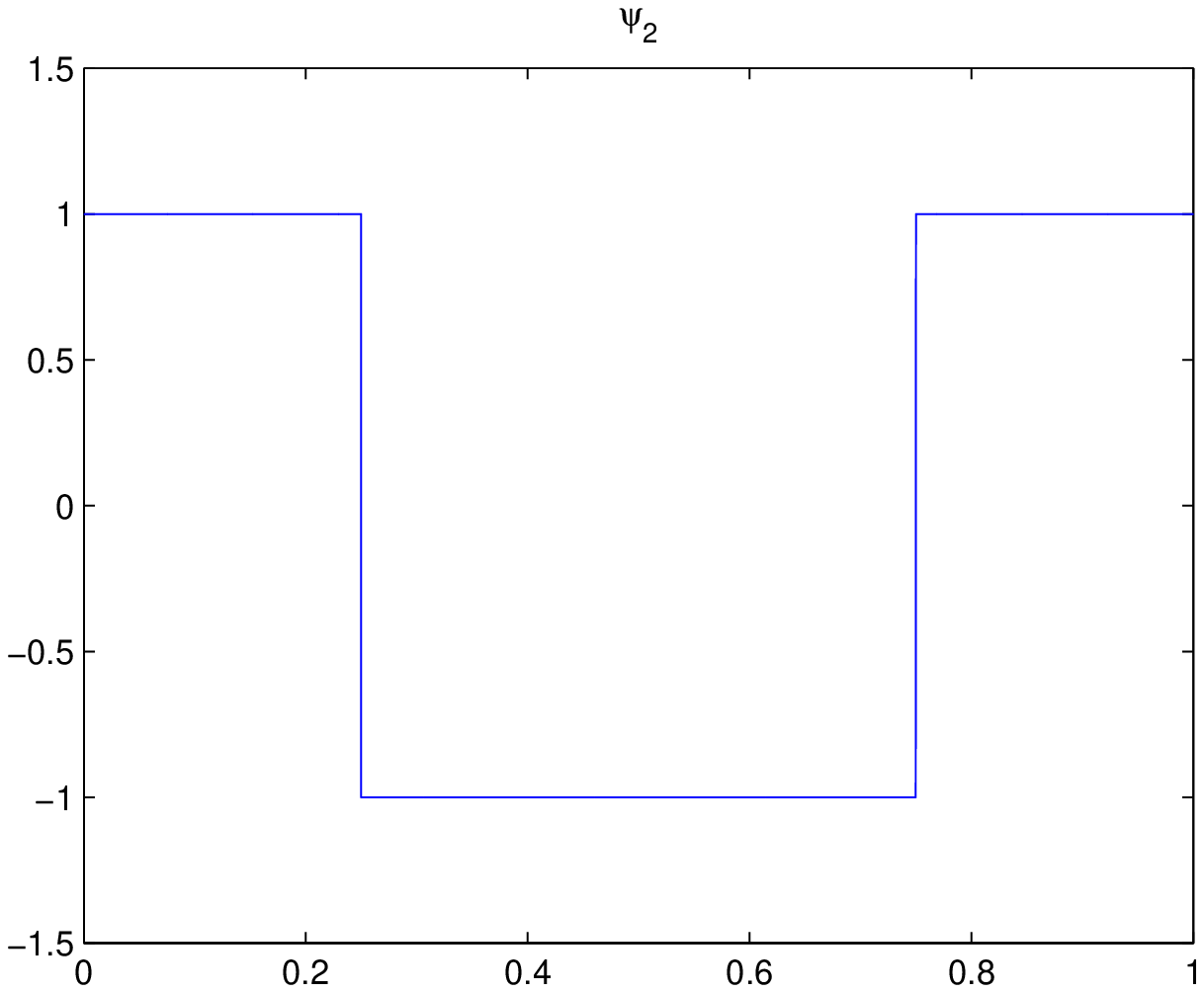} 
\\ 
\multicolumn{2}{c}{(a) Parameters: $u_0= (1,0), \, u_1=(0,1)$,}
&
\multicolumn{2}{c}{(b) Parameters: $u_0= (1,1)/\sqrt{2}, \, u_1= (1,-1)/\sqrt{2}$,}
\\ 
\multicolumn{2}{c}{$B_{0} = 1$ and $\text{arg}(C_{1}) = 0$.}
&
\multicolumn{2}{c}{$B_{0} = 1$ and $\text{arg}(C_{1}) = 0$.}
\end{tabular} 
\end{center}
\caption{Case $r=2$. Examples of real functions $\psi_1$ and $\psi_2$ of type 2.  \label{fig1}}
\end{figure}

\begin{figure}
\footnotesize%\tiny
\begin{center}
\begin{tabular}{cccc}
\includegraphics[width=0.23\linewidth]{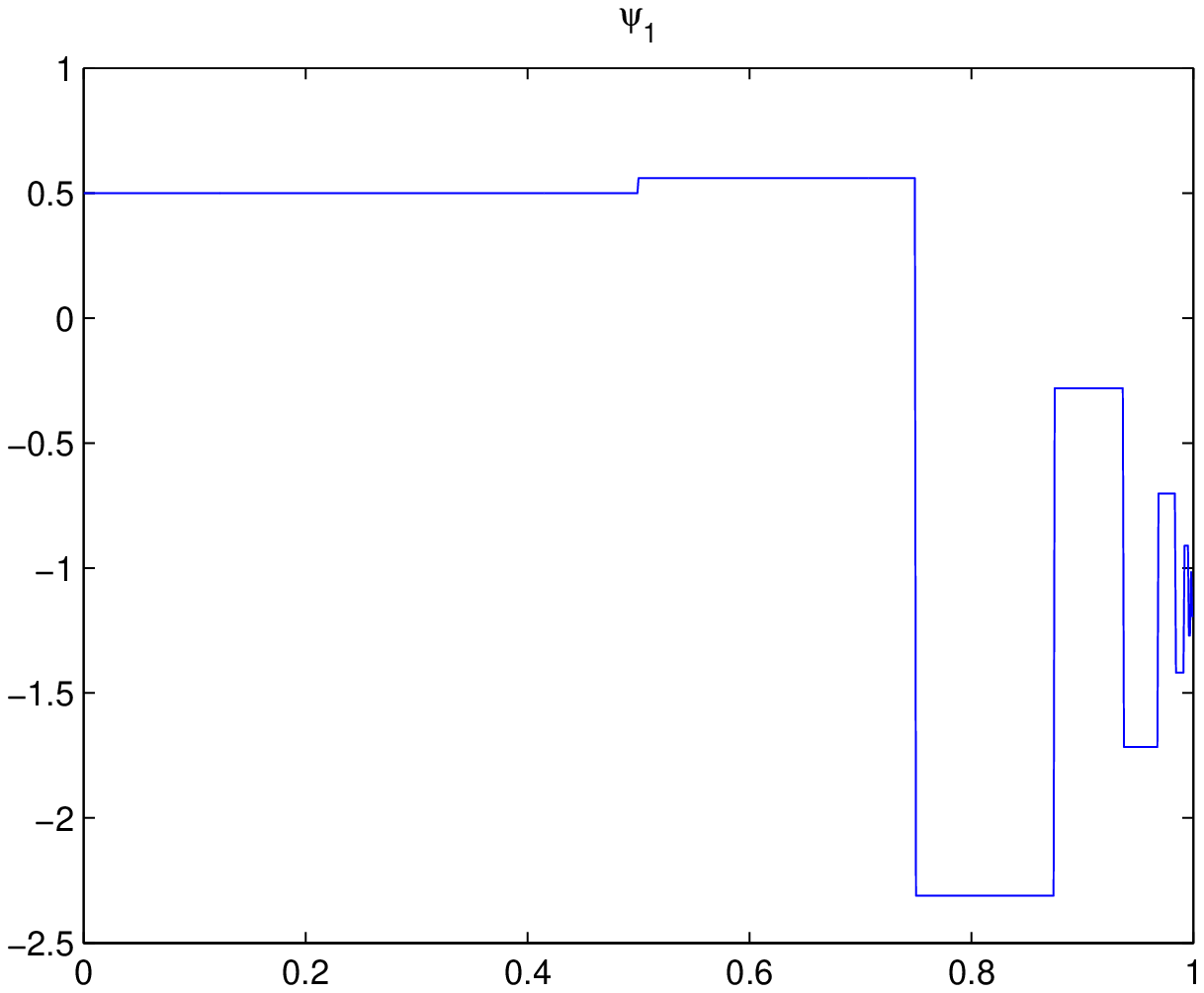} &
\includegraphics[width=0.23\linewidth]{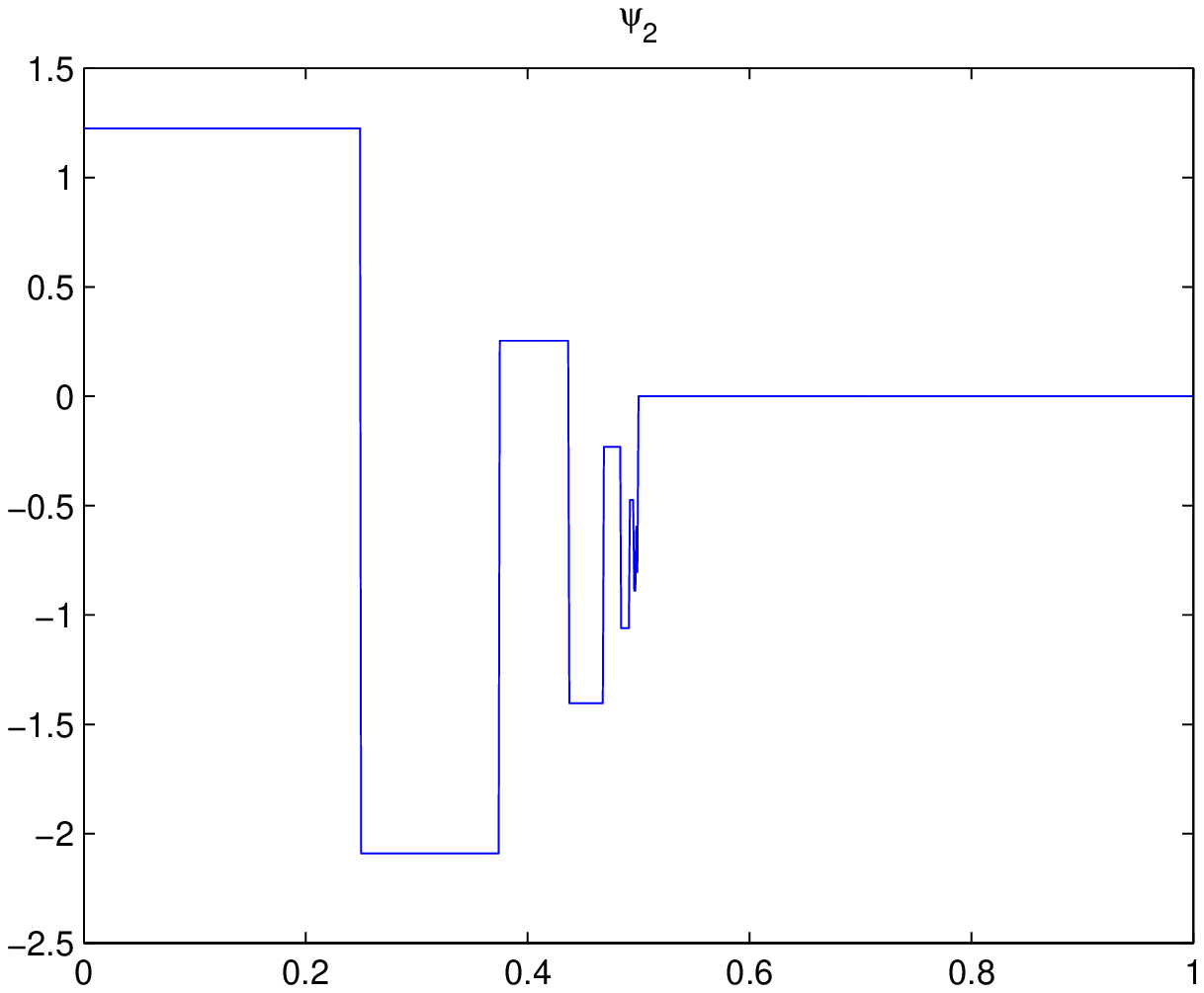} &
\includegraphics[width=0.23\linewidth]{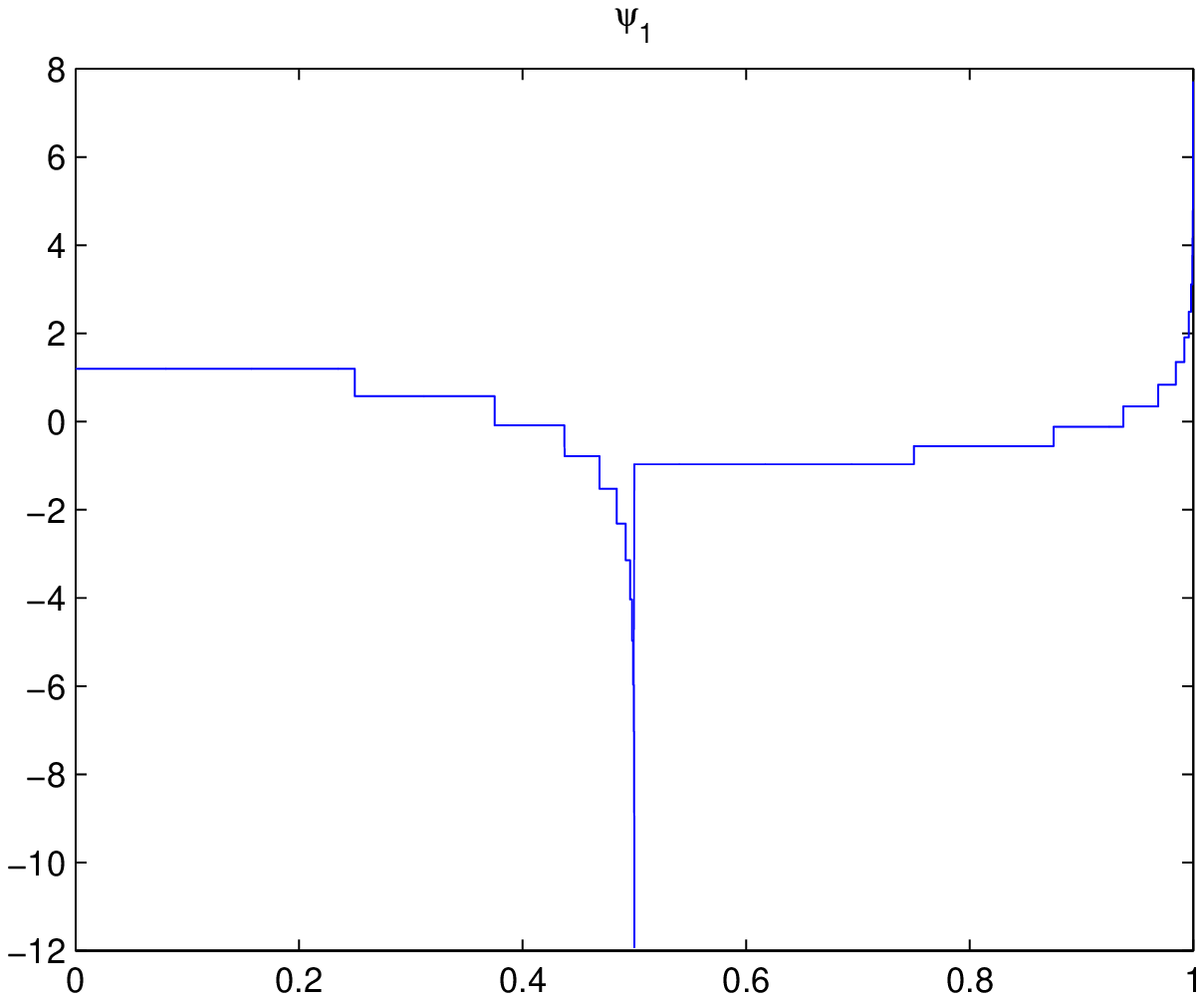} &
\includegraphics[width=0.23\linewidth]{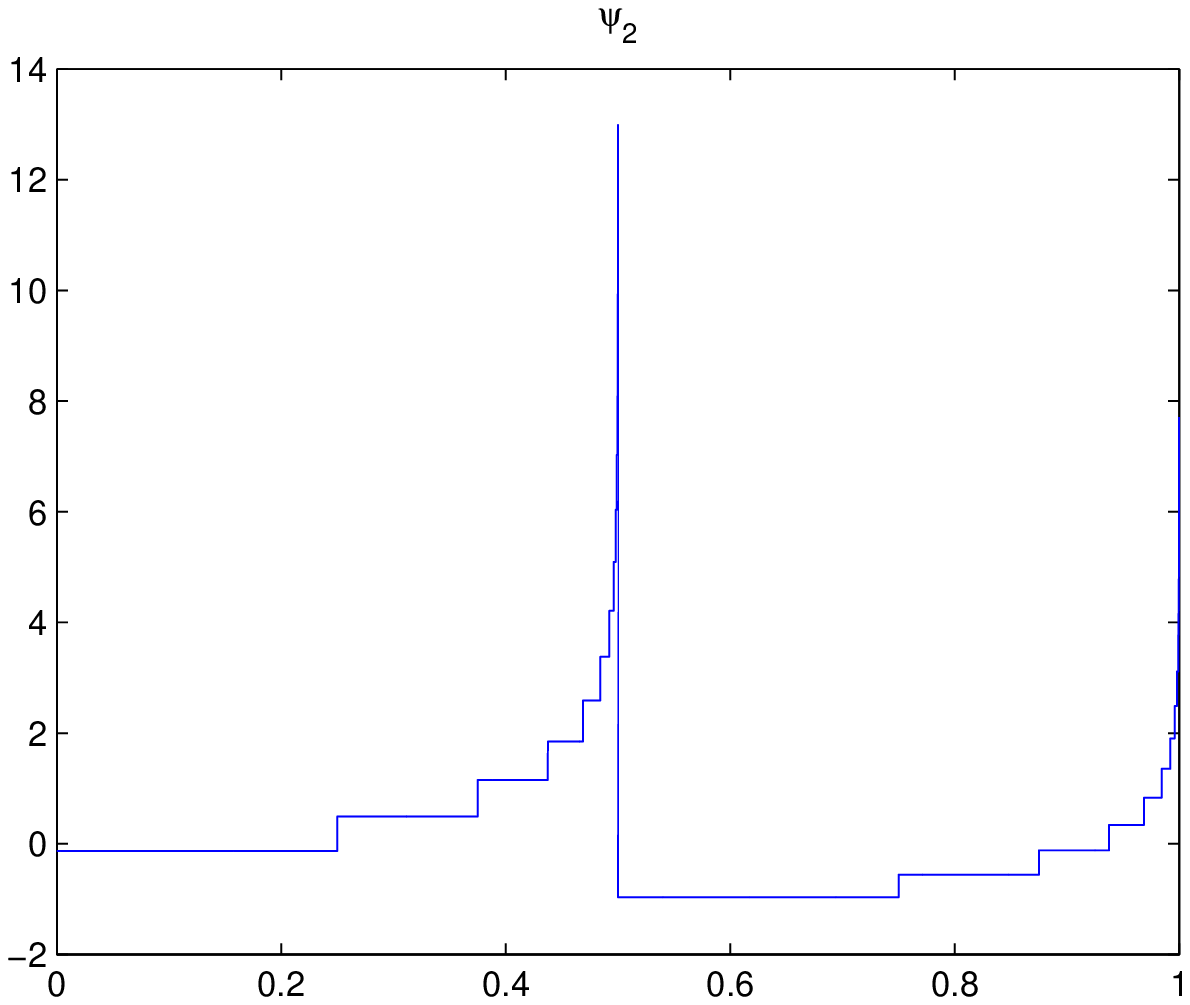} 
\\ 
\multicolumn{2}{c}{(a) Parameters: ${\mf u}_{0} = (1,0), \, {\mf u}_{1} = (0,1)$,} &
\multicolumn{2}{c}{(b) Parameters: ${\mf u}_{0} = (1,1)/\sqrt{2}, \, {\mf u}_{1} = (1,-1)/\sqrt{2}$,}
\\ 
\multicolumn{2}{c}{$\rho = 1/2$, $B_{0} = 1$, $\text{arg}(C_{1}) = 0$ and $\theta = 0$.} &
\multicolumn{2}{c}{$\rho = 3/4$, $B_{0} = 1$, $\text{arg}(C_{1}) = 0$ and $\theta = \pi$.}
\end{tabular} 
\end{center}
\caption{Case $r=2$. Examples of real functions $\psi_1$ and $\psi_2$ of type 3.  \label{fig2}}
\end{figure}

\begin{figure}
\footnotesize%\tiny
\begin{center}
\begin{tabular}{cccc}
\includegraphics[width=0.23\linewidth]{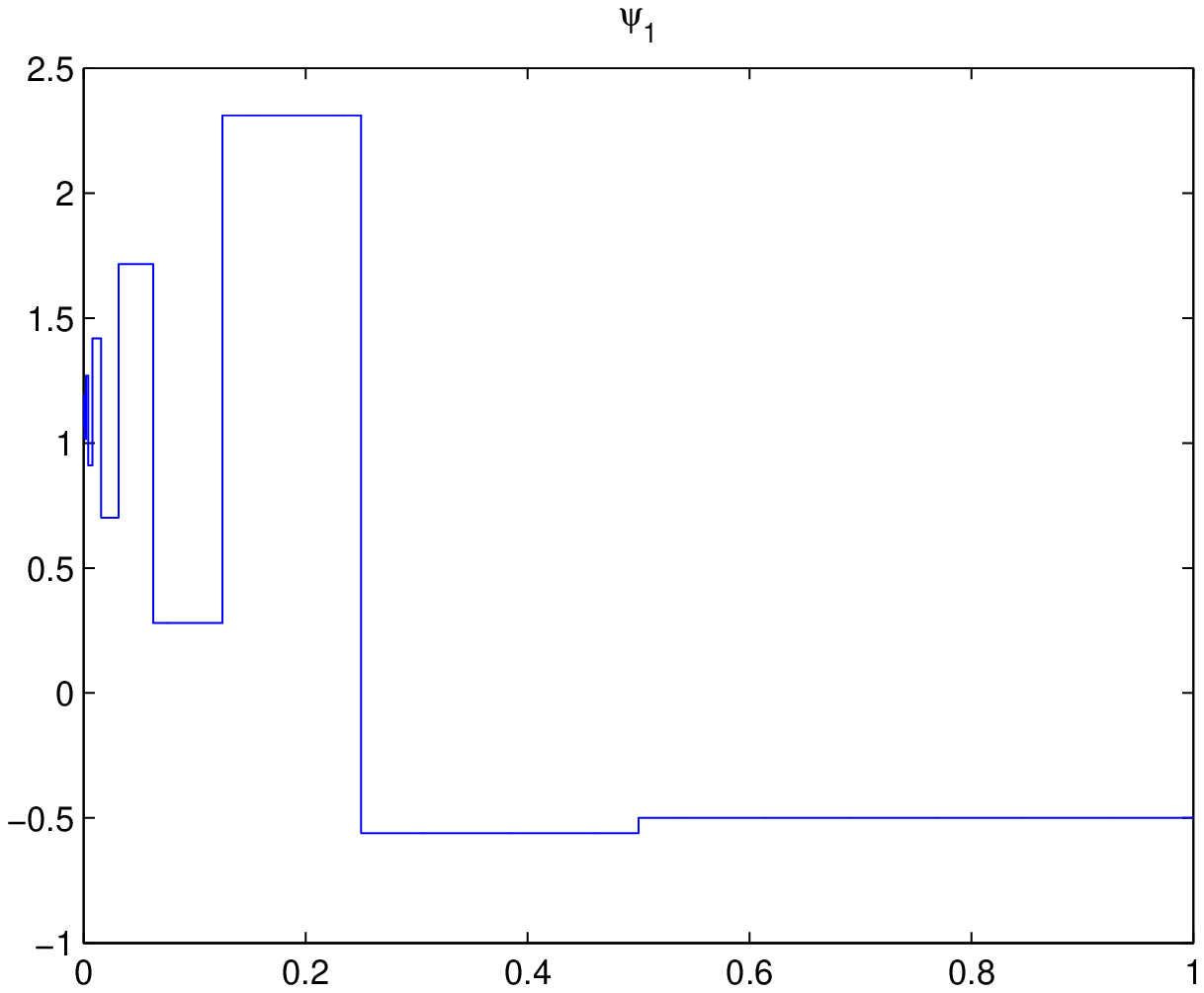} &
\includegraphics[width=0.23\linewidth]{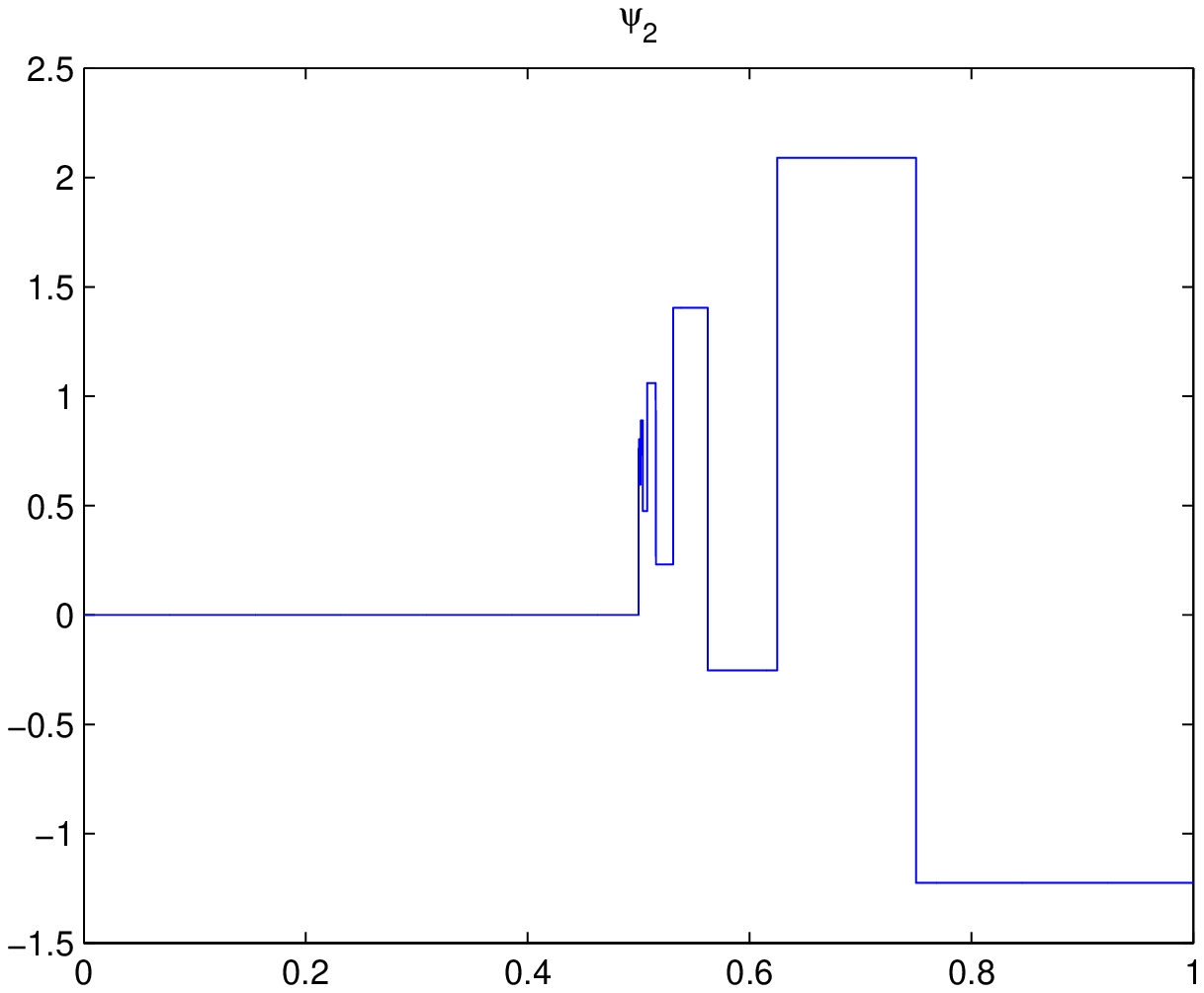} &
\includegraphics[width=0.23\linewidth]{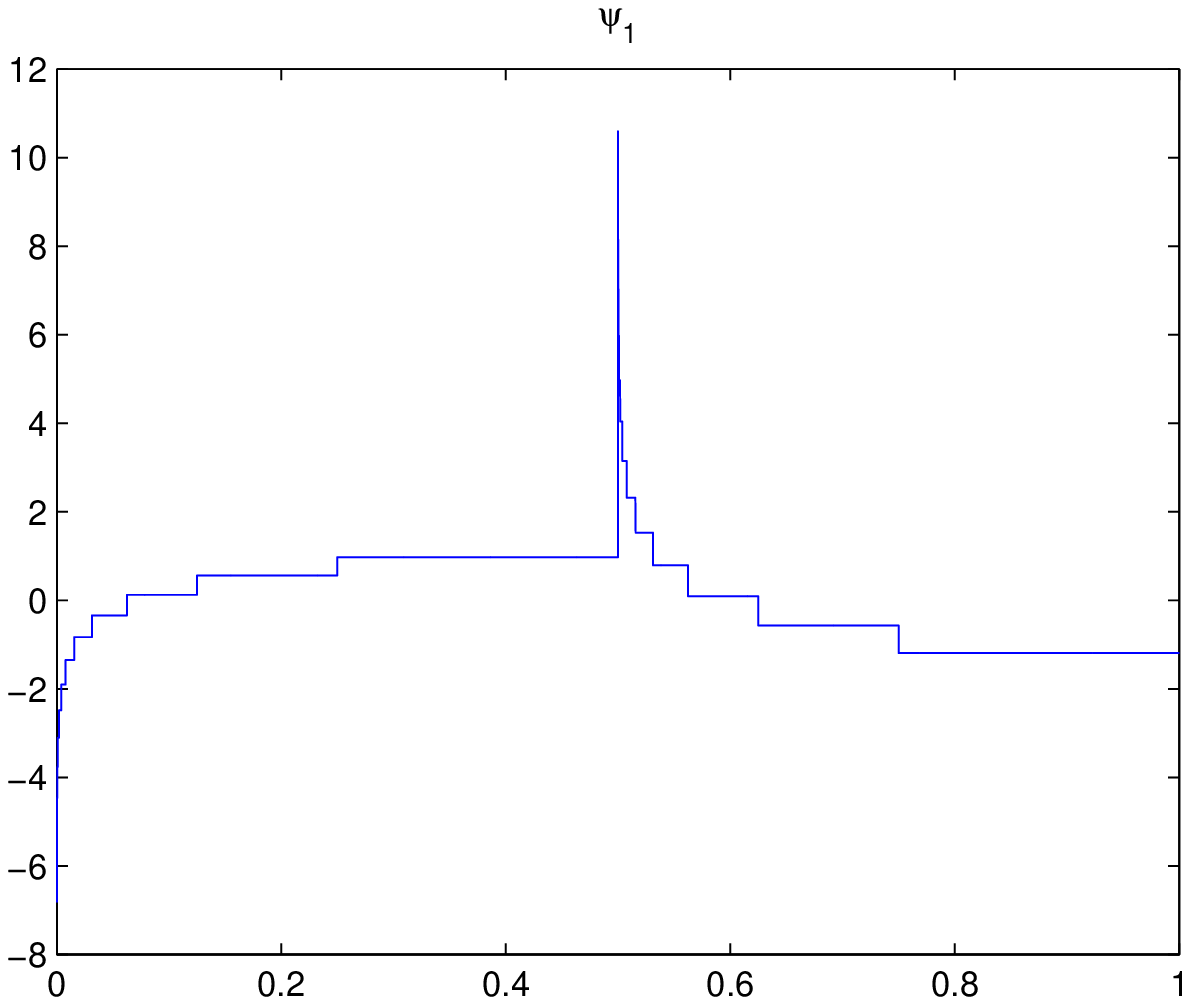} &
\includegraphics[width=0.23\linewidth]{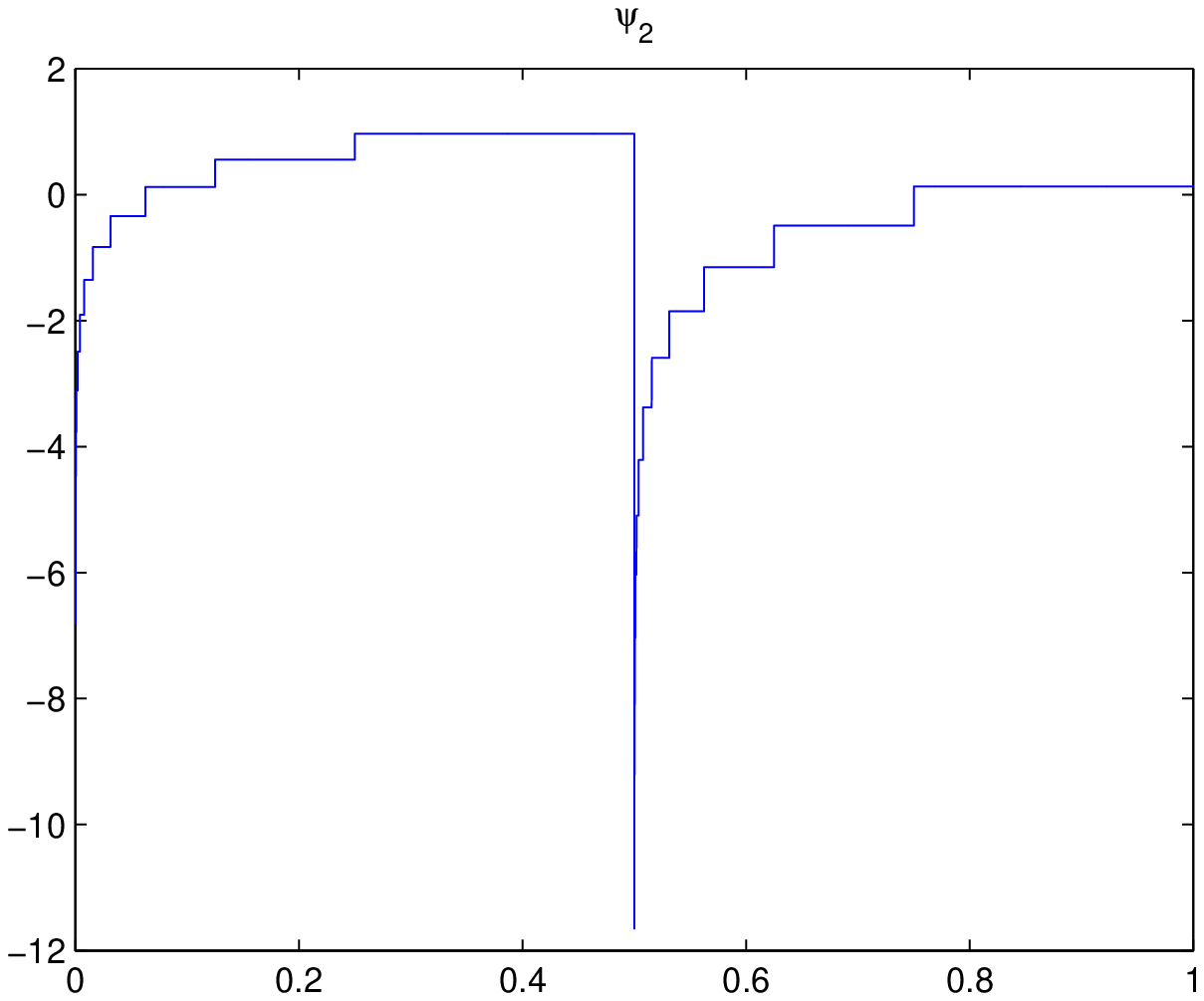} 
\\ 
\multicolumn{2}{c}{(a) Parameters: ${\mf u}_{0} = (1,0), \, {\mf u}_{1} = (0,1)$,} &
\multicolumn{2}{c}{(b) Parameters: ${\mf u}_{0} = (1,1)/\sqrt{2}, \, {\mf u}_{1} = (1,-1)/\sqrt{2}$,}
\\ 
\multicolumn{2}{c}{$\rho = 1/2$, $B_{0} = 1$, $\text{arg}(C_{1}) = 0$ and $\theta = 0$.} &
\multicolumn{2}{c}{$\rho = 3/4$, $B_{0} = 1$, $\text{arg}(C_{1}) = 0$ and $\theta = \pi$.}
\end{tabular} 
\end{center}
\caption{Case $r=2$. Examples of real functions $\psi_1$ and $\psi_2$ of type 4.  \label{fig3}}
\end{figure}

\begin{figure}
\footnotesize%\tiny
\begin{center}
\begin{tabular}{cccc}
\includegraphics[width=0.23\linewidth]{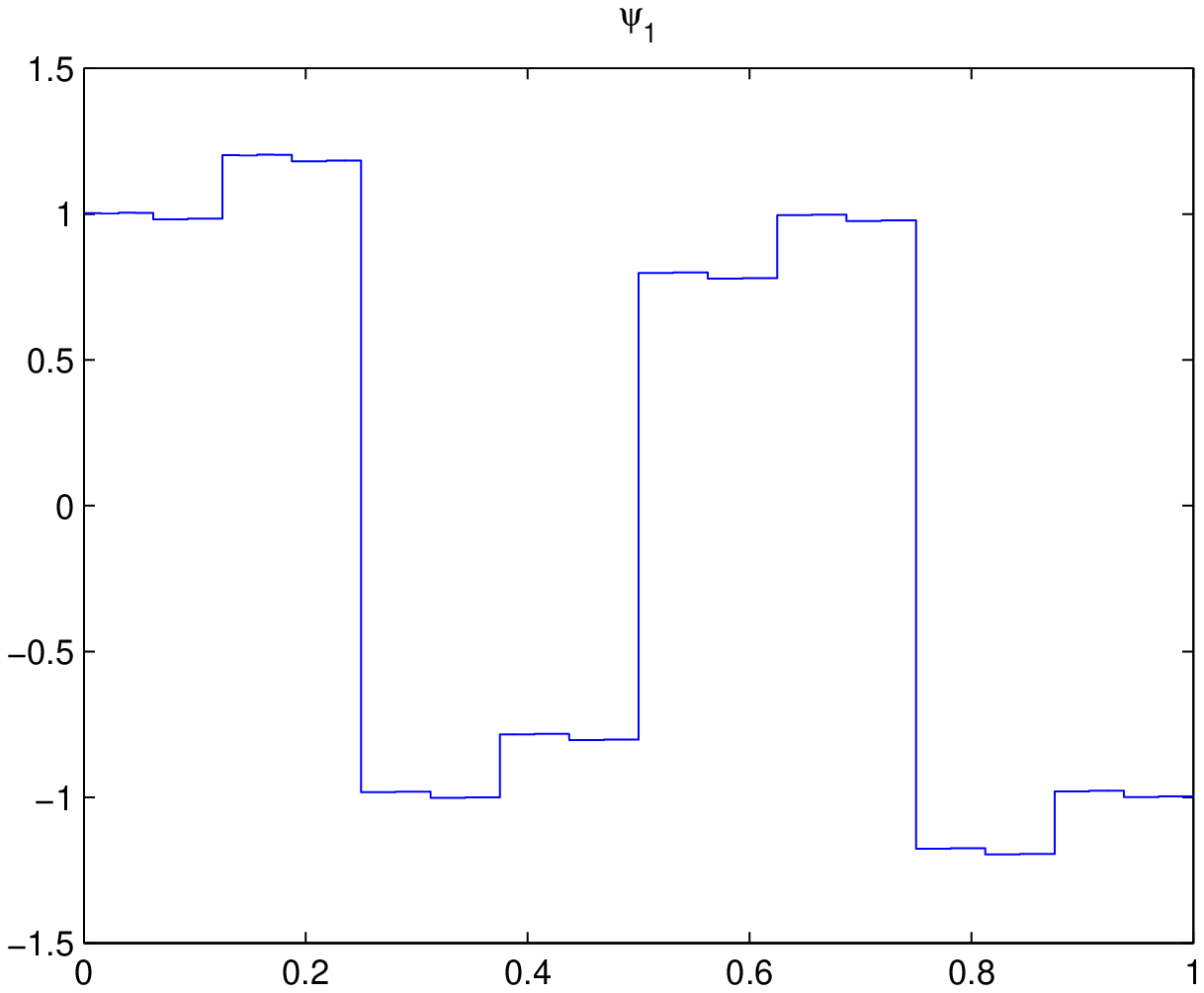} &
\includegraphics[width=0.23\linewidth]{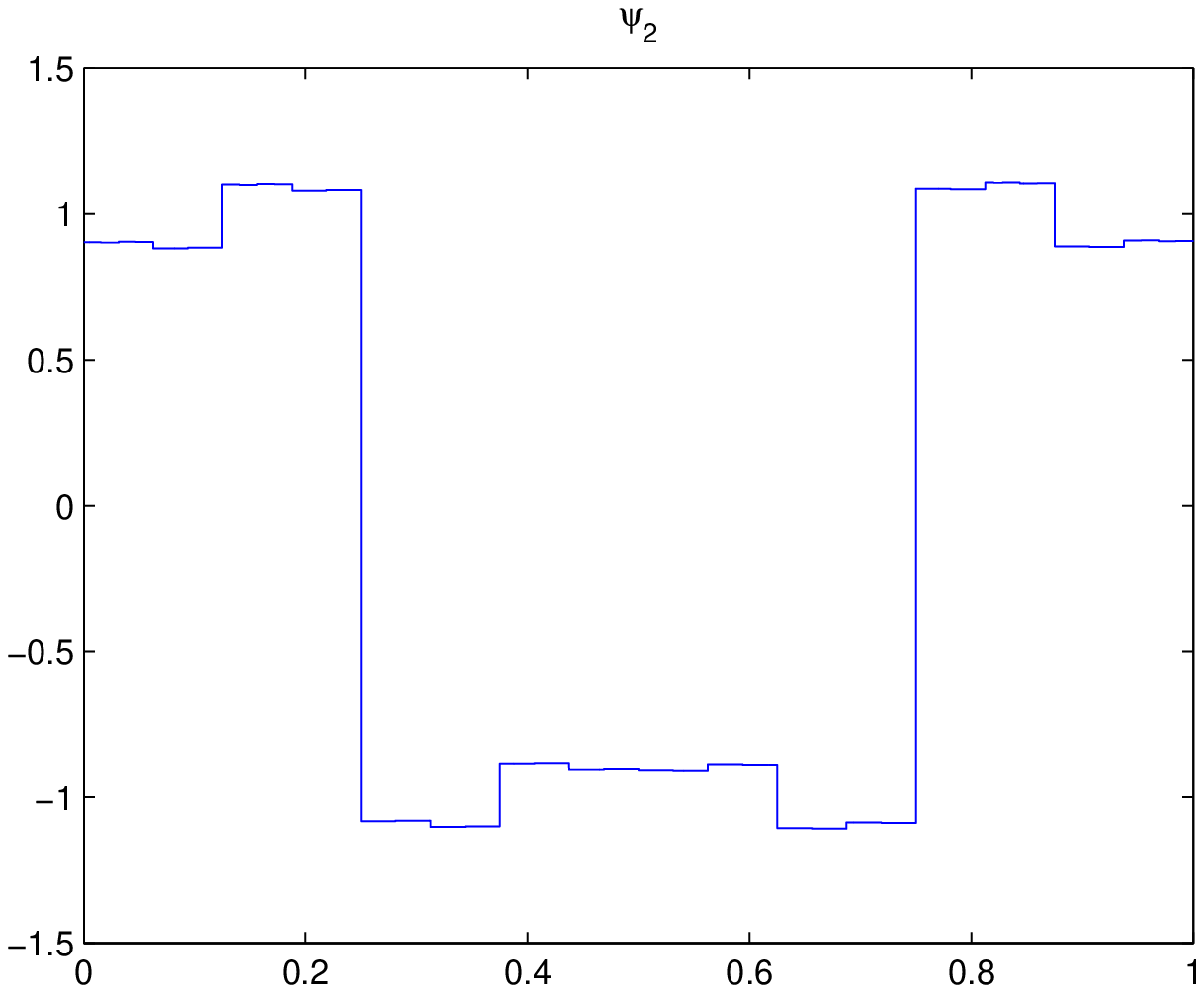} &
\includegraphics[width=0.23\linewidth]{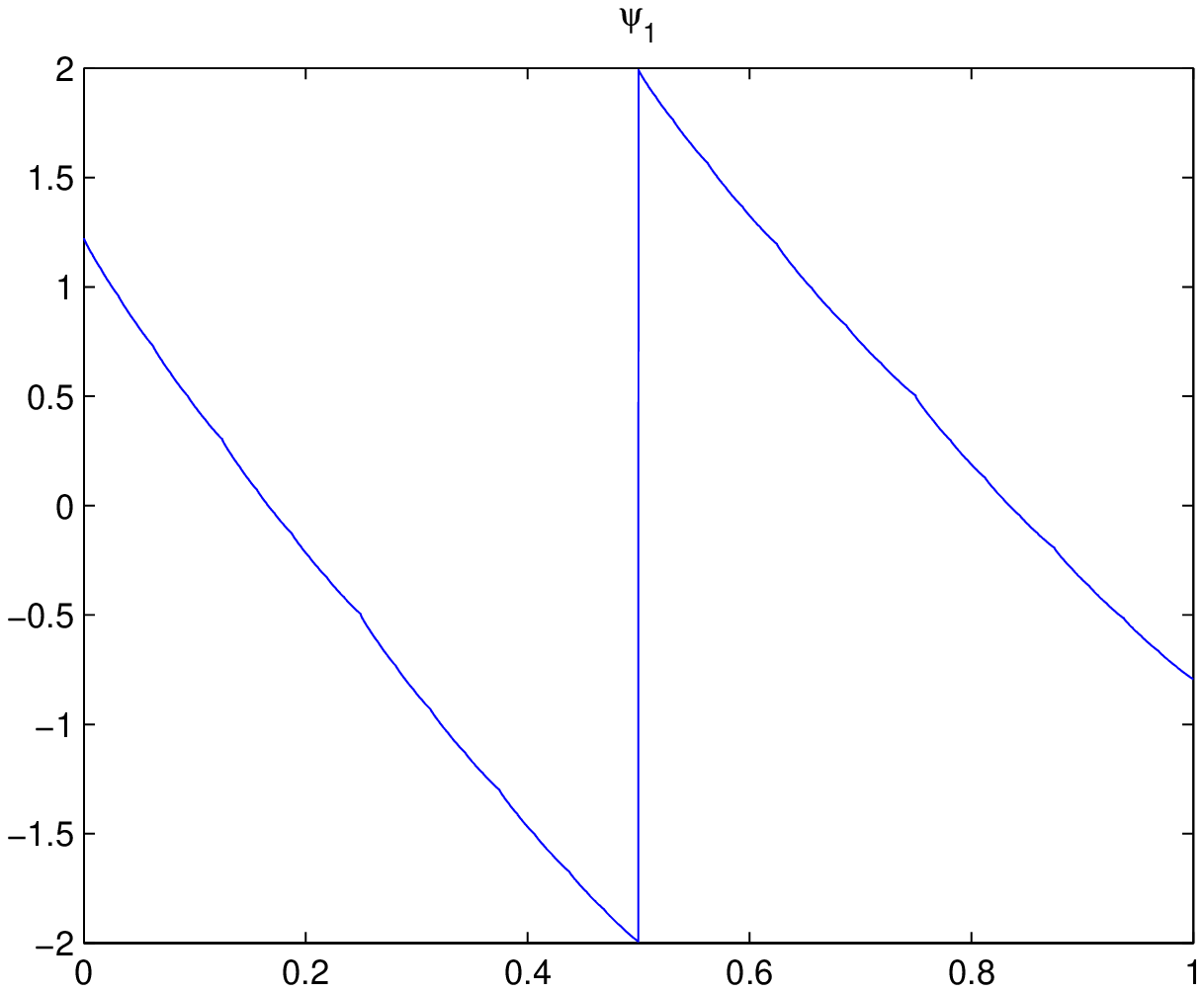} &
\includegraphics[width=0.23\linewidth]{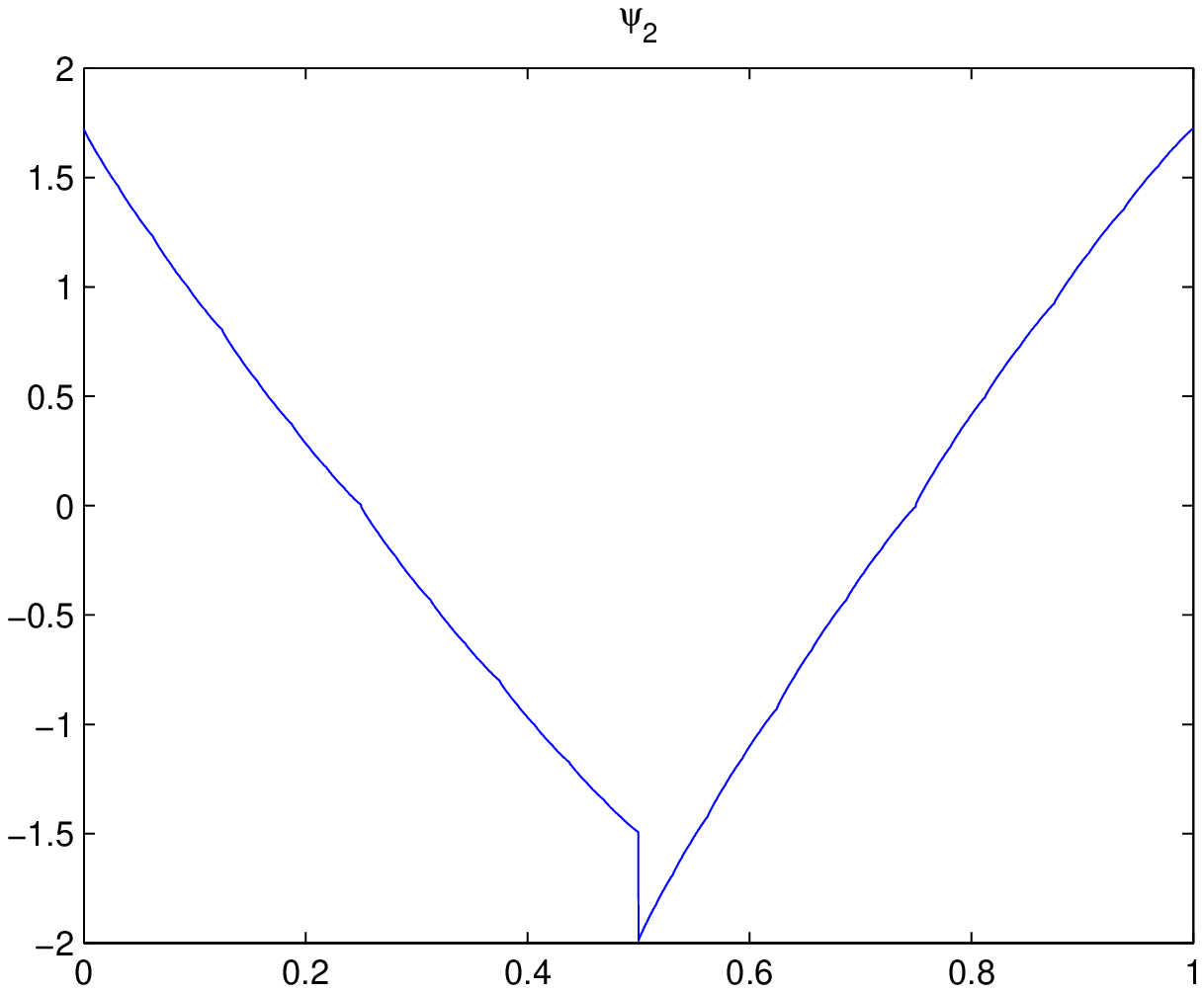} 
\\ 
\multicolumn{2}{c}{(a) Parameters: $\rho_0=0.1$, $u_{0} = (1,0)$, $v = (1,1)/\sqrt{2}$,} &
\multicolumn{2}{c}{(d) Parameters: $\rho_0=-1/2$, $u_{0} = (1,0)$, $v = (1,1)/\sqrt{2}$,}
\\ 
\multicolumn{2}{c}{$\text{arg}(\rho_{1})=\text{arg}(\tau_{0})= 0$, $B_{0} = 1$ and $\text{arg}(C_{1}) = 0$.} &
\multicolumn{2}{c}{$\text{arg}(\rho_{1})=\text{arg}(\tau_{0})= 0$, $B_{0} = 1$ and $\text{arg}(C_{1}) = 0$.}
\\[5ex]
\includegraphics[width=0.23\linewidth]{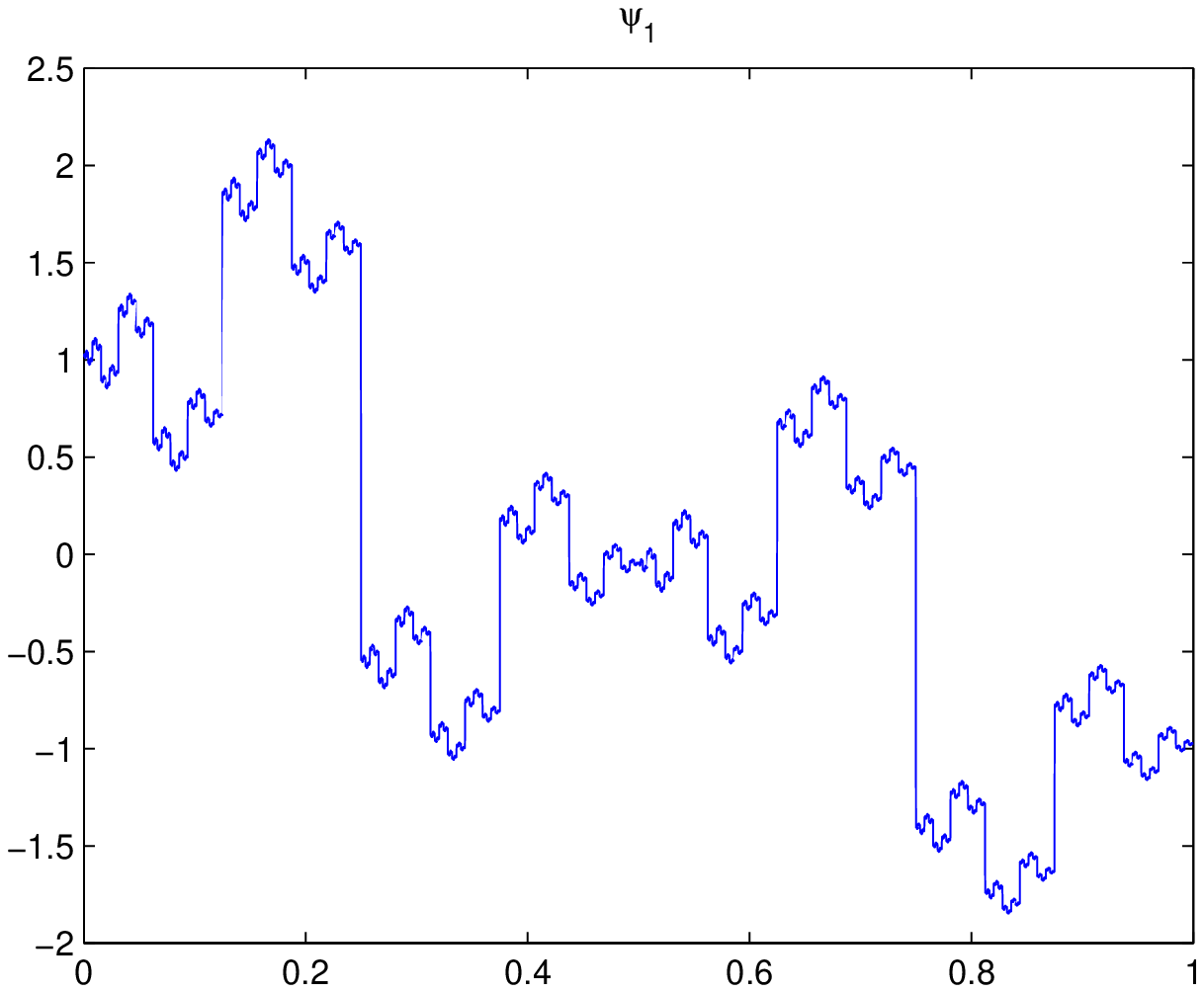} &
\includegraphics[width=0.23\linewidth]{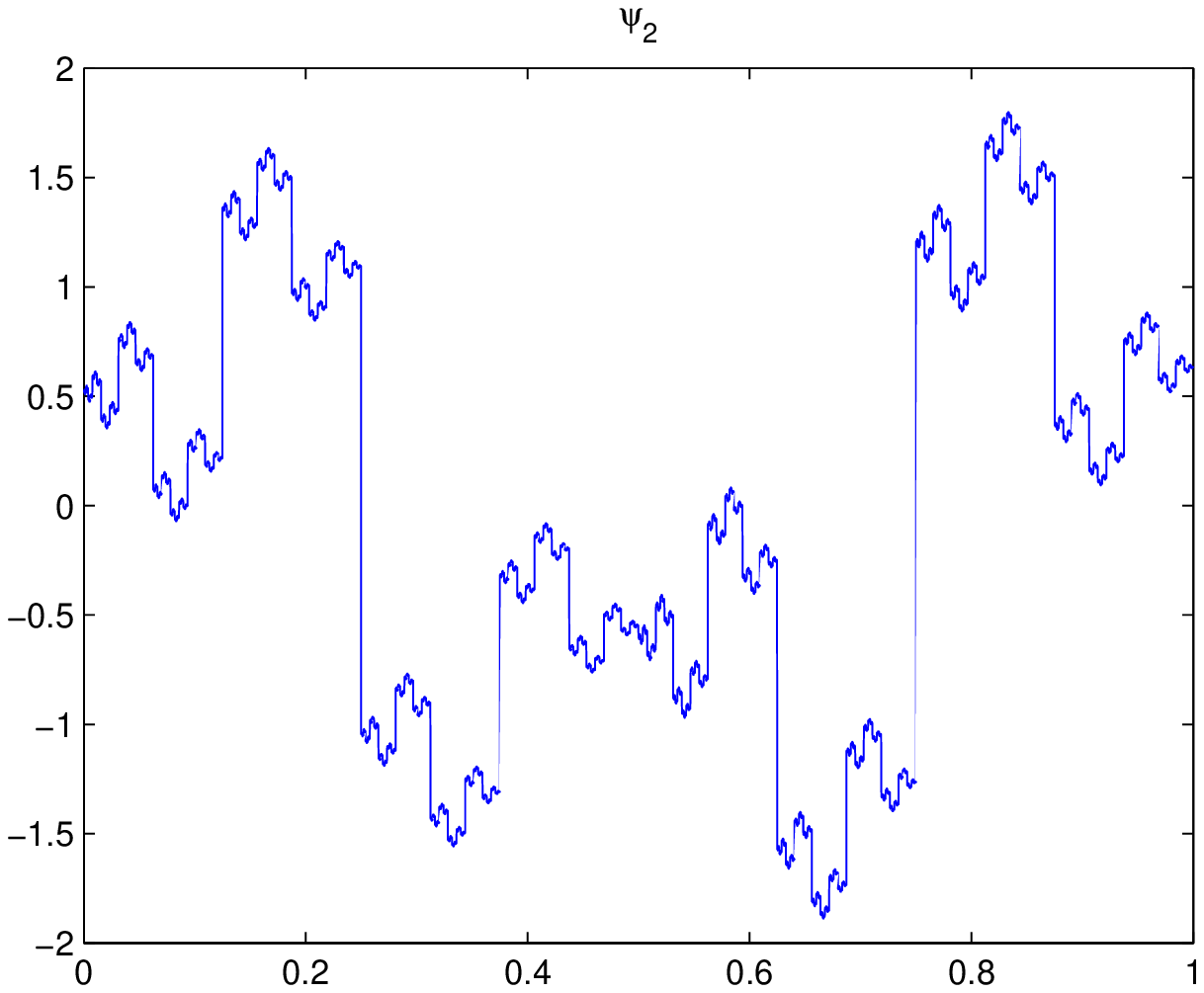} &
\includegraphics[width=0.23\linewidth]{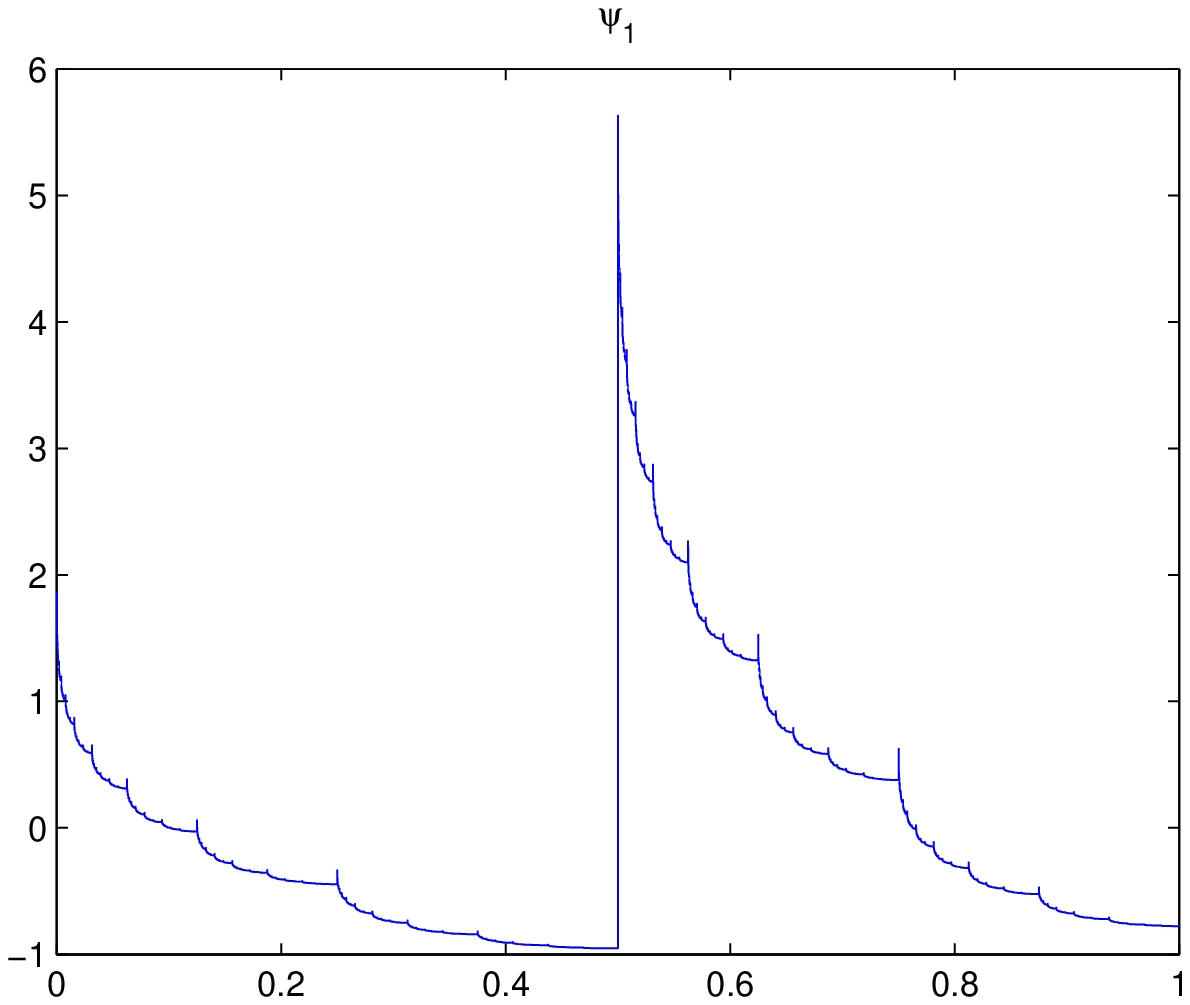} &
\includegraphics[width=0.23\linewidth]{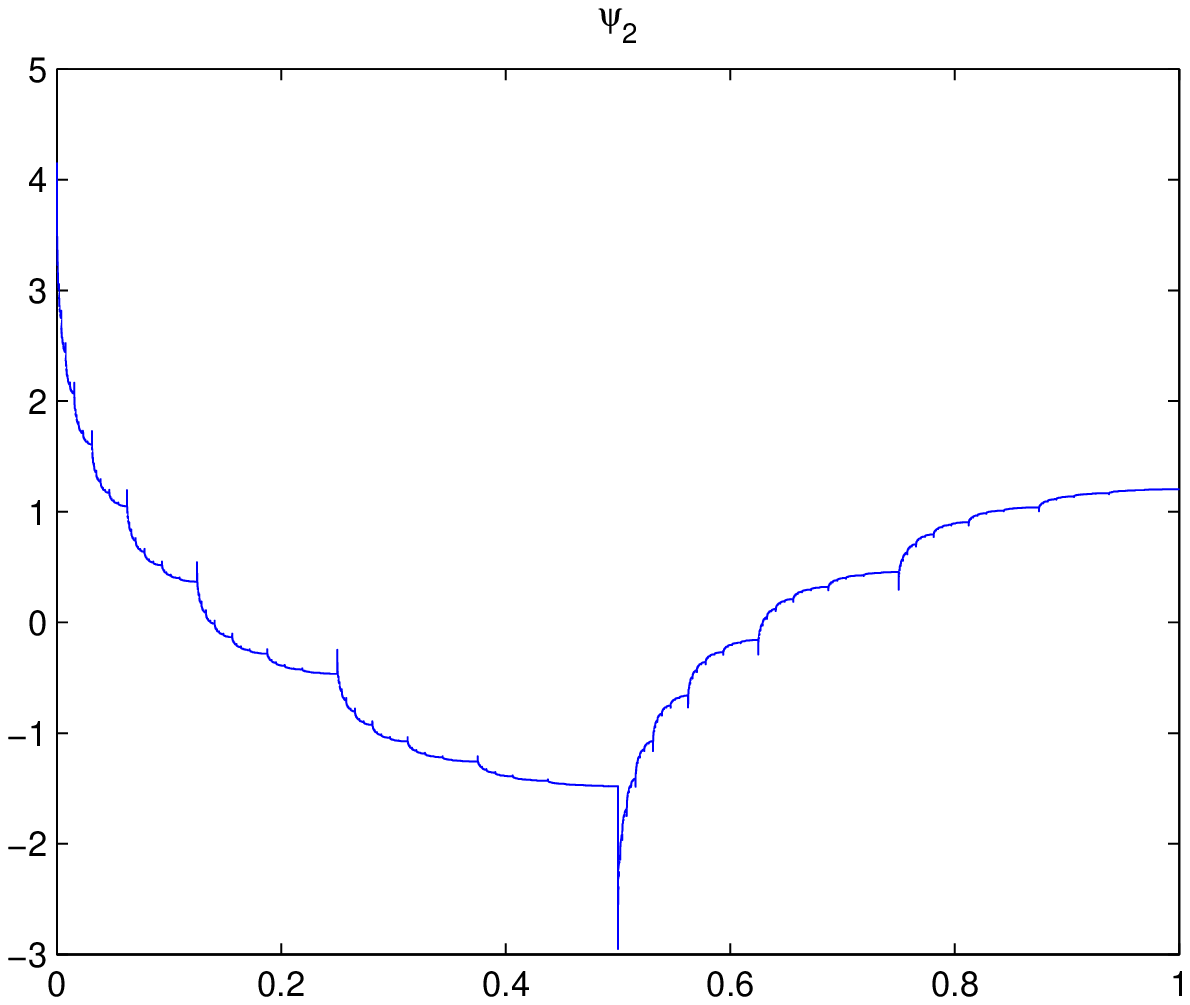} 
\\ 
\multicolumn{2}{c}{(b) Parameters: $\rho_0=1/2$, $u_{0} = (1,0)$, $v = (1,1)/\sqrt{2}$,} &
\multicolumn{2}{c}{(e) Parameters: $\rho_0=-0.6$, $u_{0} = (1,1)/\sqrt{2}$, $v = (1,2)/\sqrt{5}$,}
\\ 
\multicolumn{2}{c}{$\text{arg}(\rho_{1})=\text{arg}(\tau_{0})= 0$, $B_{0} = 1$ and $\text{arg}(C_{1}) = 0$.} &
\multicolumn{2}{c}{$\text{arg}(\rho_{1})=\text{arg}(\tau_{0})= 0$, $B_{0} = 1$ and $\text{arg}(C_{1}) = 0$.}
\\[5ex]
\includegraphics[width=0.23\linewidth]{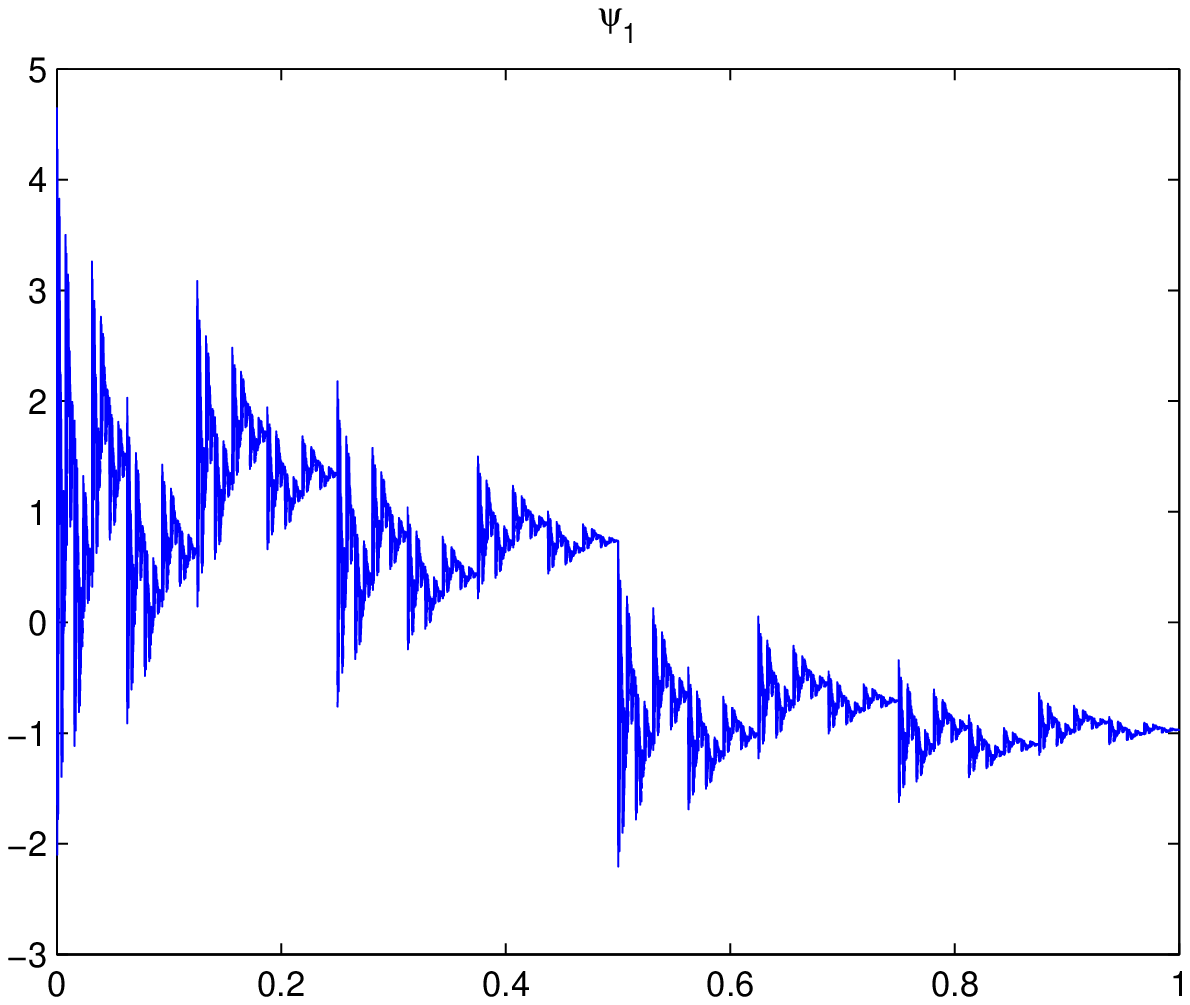} &
\includegraphics[width=0.23\linewidth]{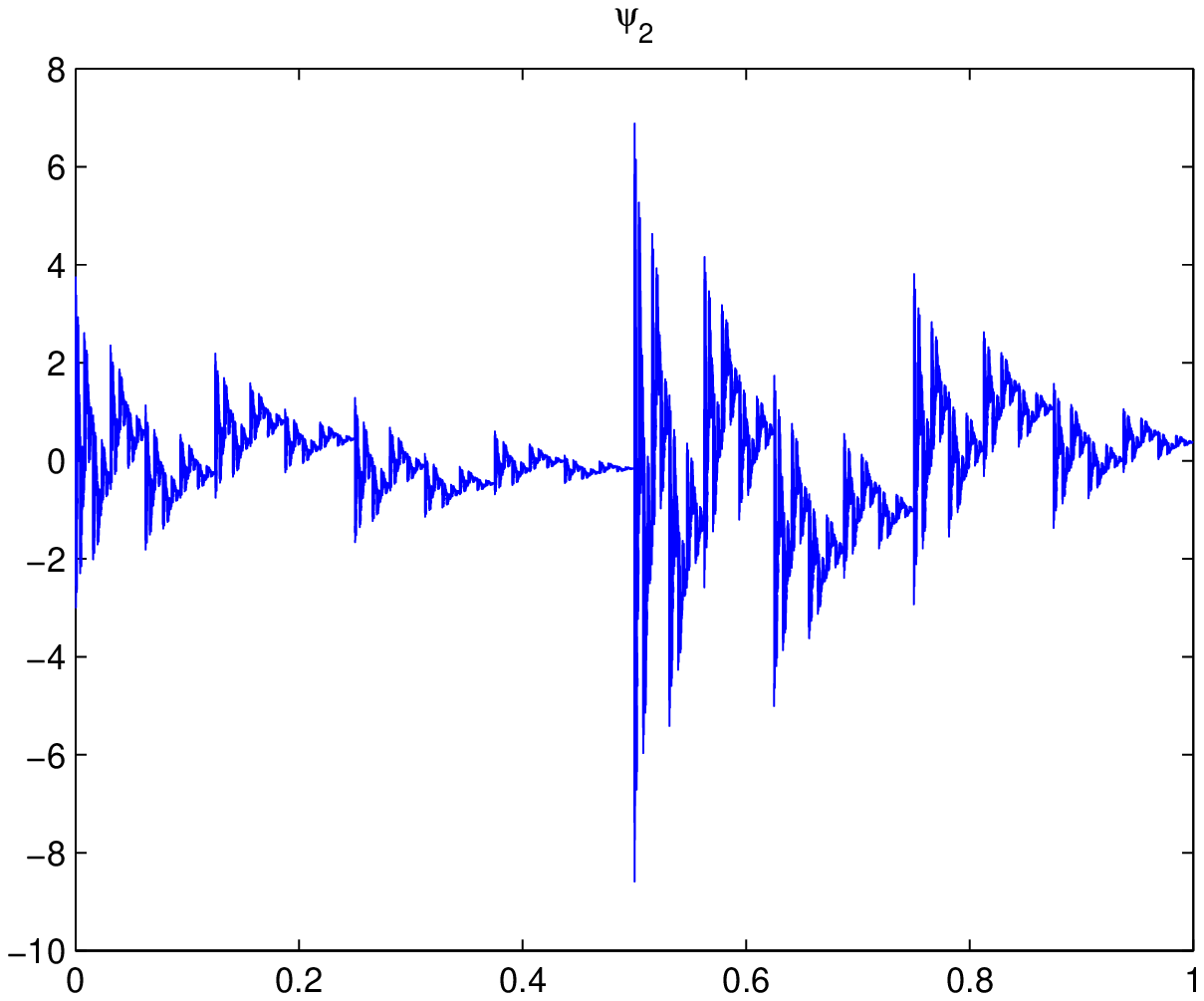} &
\includegraphics[width=0.23\linewidth]{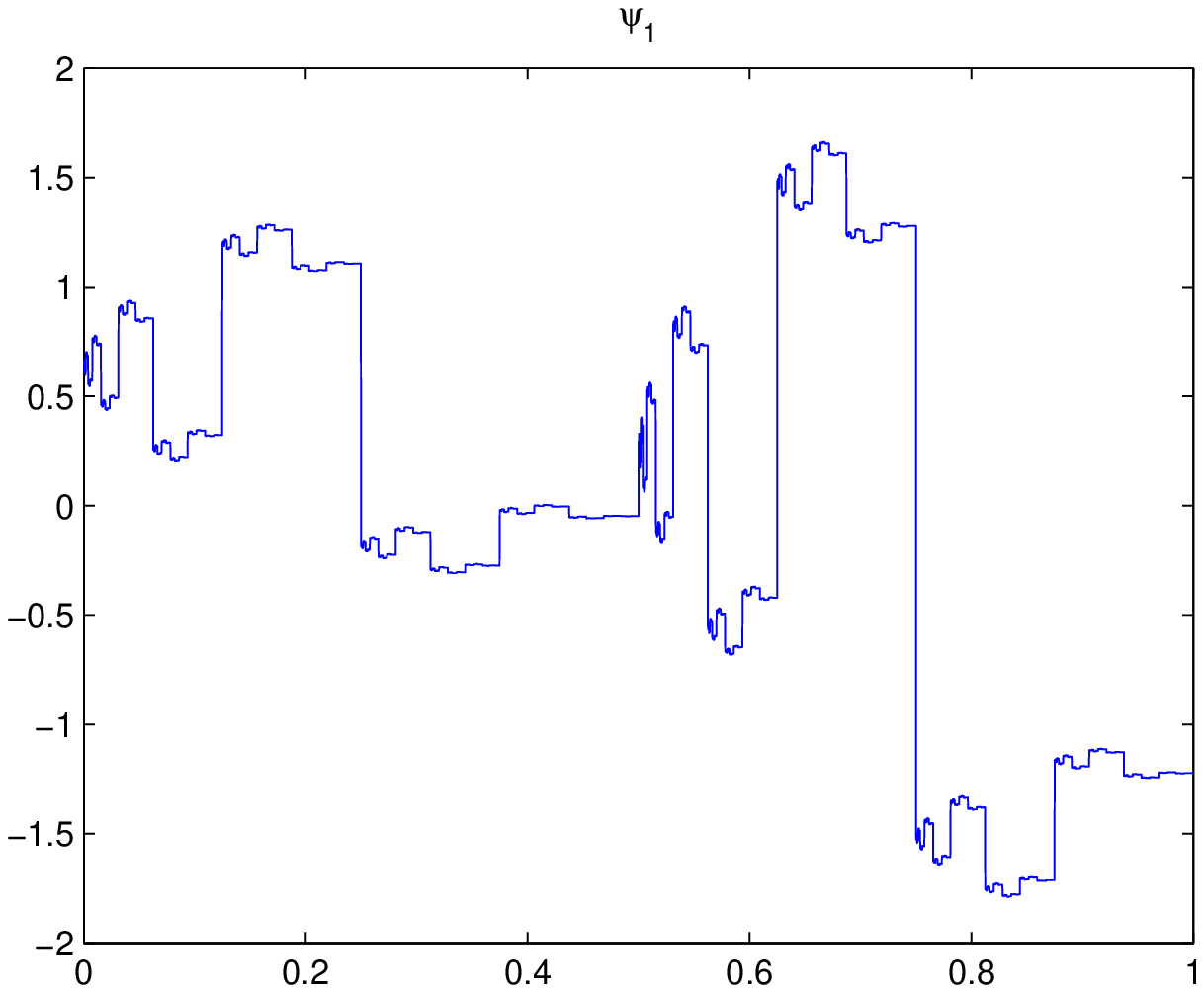} &
\includegraphics[width=0.23\linewidth]{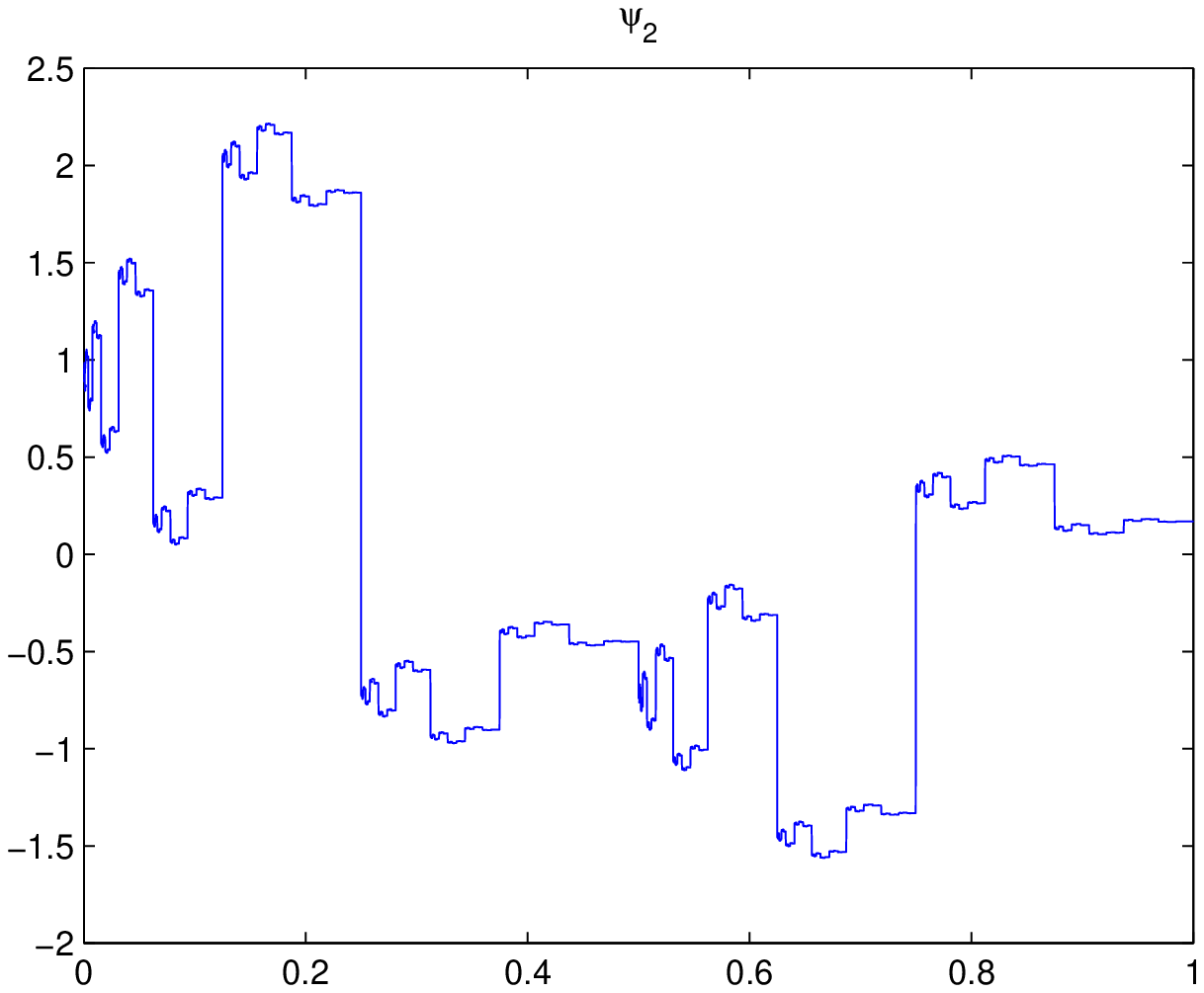} 
\\ 
\multicolumn{2}{c}{(c) Parameters: $\rho_0=0.9$, $u_{0} = (1,0)$, $v = (1,1)/\sqrt{2}$,} &
\multicolumn{2}{c}{(f) Parameters: $\rho_0=1/2$, $u_{0} = (1,1)/\sqrt{2}$, $v = (1,2)/\sqrt{5}$,}
\\ 
\multicolumn{2}{c}{$\text{arg}(\rho_{1})=\text{arg}(\tau_{0})= 0$, $B_{0} = 1$ and $\text{arg}(C_{1}) = 0$.} &
\multicolumn{2}{c}{$\text{arg}(\rho_{1})=\text{arg}(\tau_{0})= 0$, $B_{0} = 1$ and $\text{arg}(C_{1}) = 0$.}
\end{tabular} 
\end{center}
\caption{Case $r=2$. Examples of real functions $\psi_1$ and $\psi_2$ of type 5.  \label{fig4}}
\end{figure}

%-----------

%--------------------------
\section{Preliminaries}\label{sect2}

The spectral techniques for univariate wavelet frames developed in \cite{GV19-1} are based on suitable spectral representations of the translation and dilation operators $T$ and $D$ given in (\ref{tdo}). These representations are built in terms of an orthonormal basis (shortly, ONB) $\{L_{i}^{(0)}(x)\}_{i\in\I}$  of $L^2[0,1)$ and ONBs $\{K_{\pm,j}^{(0)}(x)\}_{j\in\J}$ of $L^2[\pm 1,\pm 2)$, where $\I$, $\J$ are denumerable sets of indices (usually, $\N$, $\N\cup\{0\}$ or $\Z$). Obviously, the families
\be\label{lt}
\big\{L_{i}^{(n)}(x):=[T^n L_{i}^{(0)}](x)=L_{i}^{(0)}(x-n)\big\}_{i\in\I,n\in\Z},
\ee
\be\label{kd}
\big\{K_{s,j}^{(m)}(x):=[D^m K_{s,j}^{(0)}](x)=2^{m/2}K_{s,j}^{(0)}(2^mx)\big\}_{j\in\J,m\in\Z,s=\pm}
\ee
are ONBs of $L^2(\R)$ and, for each $f\in L^2(\R)$, one has (in $L^2$-sense)
\be\label{fdobl2}
f=\sum_{i,n} \hat f^{(n)}_{i} L_i^{(n)},\text{ with }\hat f^{(n)}_{i}:=\<f,L_i^{(n)}\>_{L^2(\R)}\,,
\ee
\be\label{fdobl1}
f=\sum_{s,j,m} \tilde f^{(m)}_{s,j} K_{s,j}^{(m)},\text{ with }\tilde f^{(m)}_{s,j}:=\<f,K_{s,j}^{(m)}\>_{L^2(\R)}\,.
\ee
The change of representation between both expansions (\ref{fdobl2}) and (\ref{fdobl1}) is governed by a matrix $\big(\alpha_{i,n}^{s,j,m}\big)$, where
\be\label{chm}
\alpha_{i,n}^{s,j,m}:=\< L_i^{(n)},K_{s,j}^{(m)}\>_{L^2(\R)}\,.
\ee
In what follows, fixed ONBs $\{L_{i}^{(n)}(x)\}_{i\in\I,n\in\Z}$ and  $\{K_{s,j}^{(m)}(x)\}_{j\in\J,m\in\Z,s=\pm}$ of $L^2(\R)$ as above, for each $f\in L^2(\R)$ we shall write
$$
f=\big\{\hat f_{i}^{(n)}\big\}=\big\{\tilde f_{s,j}^{(m)}\big\}
$$
and we shall also use the notation
$$
\hat{f}_{i}(\om):=\sum_{n\in\Z} \om^{n}\,\hat f_{i}^{(n)}\quad \text{and}\quad
\tilde{f}_{s,j}(\om):=\sum_{m\in\Z} \om^{m}\,\tilde f_{s,j}^{(m)}\,,\quad (f\in L^2(\R);i\in\I;s=\pm,j\in\J)\,.
$$

Corollary 3.6 in \cite{GV19-1} gives a useful description of tight wavelet frames for $L^2(\R)$: 

\medskip\noindent
{\it 
A wavelet system $X_\Psi$ of the form (\ref{ws}) and such that 
\be\label{aePsi}
\sup_{\psi\in\Psi} ||\psi||_{L^2(\R)}=M<\infty
\ee 
is a {tight frame} for $L^2(\R)$, with frame bound $B$, if and only if
\be\label{sesg}
\sum_{\stackrel{i,n}{i',n'}}
\big({\sum_{k,j\in\Z}}^u\, \overline{\alpha_{i,n+j}^{s,l,k}}\,\alpha_{i',n'+j}^{s',l',k+\sigma}\big)\big({\sum_{\psi\in\Psi}}^u\, \overline{\hat\psi_{i}^{(n)}}\,\hat\psi_{i'}^{(n')}\big)=
B\,\delta_{s,s'}\delta_{l-l'}\delta_\sigma\,,\quad
(s,s'=\pm,\,l,l'\in\J,\,\sigma\in\Z)\,.
\ee
}
%\medskip

\noindent
Here, $\delta$ denotes the Dirac $\delta$-function, the superindex 'u' added to the sum symbols $\sum$ reflects the unconditional convergence of the series (see \cite[Lemma 3.4]{GV19-1} and comments that follow it), and the components $\big\{\hat\psi_i^{(n)}\big\}$ of each $\psi\in\Psi$ and the $\alpha_{i,n}^{s,j,m}$'s are related with the spectral representations of section \ref{sect2} (see equations (\ref{fdobl2}) and  (\ref{chm})).

\section{Tight wavelet frames of minimal support}\label{sect4}

In this section, Corollary 3.6 in \cite{GV19-1} is used  to determine all the tight wavelet frames for $L^2(\R)$ of the form (\ref{ws}), with cardinal of $\Psi$ finite
and such that the support of each $\psi\in\Psi$, $\text{supp}\,\psi$, is included in the interval $[0,1]$.
Note that, since the cardinal of $\Psi$ is finite, condition (\ref{aePsi}) is trivially satisfied.

Due to the structure of the ONB $\big\{L_i^{(j)}\big\}$, defined by (\ref{lt}),
the fact that $\text{supp}\,\psi\subseteq [0,1]$, ($\psi\in\Psi$), implies that their expansions (\ref{fdobl2}) read
$$
\psi=\sum_{i\in\I} \hat\psi_i^{(0)}L_i^{(0)}\,,\quad  (\psi\in\Psi)\,,
$$
since $\hat\psi_i^{(n)}=0$ for $n\neq 0$.
Thus,  for non-zero summands in the left hand side of (\ref{sesg}) it must be $n=n'=0$, so that 
\be\label{uutf}
\sum_{\stackrel{i,n}{i',n'}}
\big({\sum_{k,j\in\Z}}^u\, \overline{\alpha_{i,n+j}^{s,l,k}}\,\alpha_{i',n'+j}^{s',l',k+\sigma}\big)\big({\sum_{\psi\in\Psi}}^u\, \overline{\hat\psi_{i}^{(n)}}\,\hat\psi_{i'}^{(n')}\big)=
\sum_{{i,i'}}
\big({\sum_{k,j\in\Z}}^u\, \overline{\alpha_{i,j}^{s,l,k}}\,\alpha_{i',j}^{s',l',k+\sigma}\big)\big({\sum_{\psi\in\Psi}}^u\, \overline{\hat\psi_{i}^{(0)}}\,\hat\psi_{i'}^{(0)}\big)\,.
\ee
Now, being finite the cardinal of $\Psi$, according to \cite[Lemma 3.4]{GV19-1} and the comments that follow it, the unconditional sums in (\ref{uutf}) may be calculated, for example, in the following way:
$$
\sum_{{i,i'}}
\big({\sum_{k,j\in\Z}}^u\, \overline{\alpha_{i,j}^{s,l,k}}\,\alpha_{i',j}^{s',l',k+\sigma}\big)\big({\sum_{\psi\in\Psi}}^u\, \overline{\hat\psi_{i}^{(0)}}\,\hat\psi_{i'}^{(0)}\big)=
\lim_{a\to\infty}\sum_{i,i'}\big[\sum_{k=-a}^a\sum_{j=-2^a}^{2^a}\overline{\alpha_{i,j}^{s,l,k}} \alpha_{i',j}^{s',l',k+\sigma}\big]
\big[\sum_{\psi\in\Psi} \overline{\hat\psi_{i}^{(0)}}\hat\psi_{i'}^{(0)}\big]\,.
$$

Then, in this particular case, Corollary 3.6 in \cite{GV19-1} can be rewritten as follows:

\begin{prop}\label{pp11}
Let $X$ be a wavelet system in $L^2(\R)$ of the form (\ref{ws}), where $\Psi$ has finite cardinal and $\text{supp}\,\psi\subseteq [0,1]$ for every $\psi\in\Psi$. 
Then, $X$ is a {tight frame} for $L^2(\R)$, with frame bound $B$, if and only if
\be\label{ses1}
\lim_{a\to\infty}\sum_{i,i'\in\I}\big[\sum_{k=-a}^a\sum_{j=-2^a}^{2^a}\overline{\alpha_{i,j}^{s,l,k}} \alpha_{i',j}^{s',l',k+\sigma}\big]
\big[\sum_{\psi\in\Psi} \overline{\hat\psi_{i}^{(0)}}\hat\psi_{i'}^{(0)}\big]
=B\,\delta_{s,s'}\delta_{l-l'}\delta_\sigma,\quad
(s,s'=\pm,\, l,l'\in\J,\,\sigma\in\Z).
\ee
\end{prop}

From now on, we consider the Haar orthonormal bases $\big\{L_i^{(j)}\big\}$ and $\big\{K_{s,l}^{(k)}\big\}$ and the corresponding matrix $\big(\alpha_{i,j}^{s,l,k}\big)$ given in the appendix. This choice leads to the following result, which we write in vectorial form:

\begin{prop}\label{prophv}
Let $\big\{L_i^{(0)}\big\}_{i\in\N\cup\{0\}}$ be the Haar orthonormal basis  of $L^2[0,1]$ given in (\ref{hbb0}).
Let $r\in\N$ and 
$$
\Psi=\{\psi_1,\psi_2,\ldots,\psi_r\}\subset L^2[0,1]\,,
$$ 
where $\psi_j=\sum_i [\hat\psi_j]_i^{(0)}L_i^{(0)}$, ($j=1,\ldots,r$). Let us put
$$
\Psi_i:=\left(\begin{array}{c} [\hat\psi_1]_i^{(0)} \cr [\hat\psi_2]_i^{(0)}\cr 
\vdots \cr [\hat\psi_r]_i^{(0)}\end{array}\right)\in\C^r,\quad (i\in\N\cup\{0\}).
$$
Then, the wavelet system $X$ of the form (\ref{ws}) generated by $\Psi$
is a tight frame for $L^2(\R)$, with frame bound $B$, if and only if the following conditions are satisfied:
\begin{enumerate}
\item
$\Psi_0=(0,0,\ldots,0)\in\C^r$.
\item
\be\label{v00z}
\sum_{k=\sup\{1,1-\sigma\}}^\infty \<\Psi_{2^{k-1+\sigma}},\Psi_{2^{k-1}}\>_{\C^r}=B\,\delta_\sigma,\quad (\sigma\in\Z).
\ee
\item
For $l\geq 1$, $l=2^p+\sum_{t=0}^{p-1}l_t2^t$, ($p\geq0$),
\be\label{v++0}
\sum_{k=-p}^{0} ||\Psi_{2^{p+k}+\sum_{t=0}^{p+k-1}l_t2^t}||^2_{\C^r}+
\sum_{k=1}^\infty ||\Psi_{2^{p+k}+l}||^2_{\C^r}=B,
\ee
\be\label{v+++}
\sum_{k=\sup\{1,1-\sigma\}}^\infty \<\Psi_{2^{p+k+\sigma}+l},\Psi_{2^{p+k}+l}\>_{\C^r}=0,\quad (\sigma\in\Z\backslash\{0\}),
\ee
\be\label{v0+z}
\sum_{k=\sup\{1,1-\sigma\}}^\infty \<\Psi_{2^{p+k+\sigma}+l},\Psi_{2^{k-1}}\>_{\C^r}=0,\quad (\sigma\in\Z).
\ee
\item
For $l,l'\geq 1$, $l\neq l'$, $l=2^p+\sum_{t=0}^{p-1}l_t2^t$, $l'=2^{p'}+\sum_{t=0}^{p'-1}l'_t2^t$, ($p,p'\geq 0$),  
\be\label{v+(+)0}
\begin{array}{c}
\ds\sum_{k=\sup\{-p,-p'\}}^{0} \delta_{(\sum_{t=0}^{-k}l_{p+k+t}2^t)-(\sum_{t=0}^{-k}l'_{p'+k+t}2^t)}\,\<\Psi_{2^{p'+k}+\sum_{t=0}^{p'+k-1}l'_t2^t},\Psi_{2^{p+k}+\sum_{t=0}^{p+k-1}l_t2^t}\>_{\C^r}+
\\
\ds +\sum_{k=1}^\infty \<\Psi_{2^{p'+k}+l'},\Psi_{2^{p+k}+l}\>_{\C^r}=0,
\end{array}
\ee
\be\label{v+(+)z}
\sum_{k=\sup\{1,1-\sigma\}}^\infty \<\Psi_{2^{p'+k+\sigma}+l'},\Psi_{2^{p+k}+l}\>_{\C^r}=0,\quad (\sigma\in\Z\backslash\{0\}).
\ee
\end{enumerate}
\end{prop}

\begin{proof}
In order to avoid additional indices along the proof, we work with generic $\psi\in\Psi$ and not with $\psi_1,\ldots,\psi_r$. 
Consider the Haar orthonormal bases $\big\{L_i^{(j)}\big\}$ and $\big\{K_{s,l}^{(k)}\big\}$ and the corresponding matrix $\big(\alpha_{i,j}^{s,l,k}\big)$ given in the appendix.
For $s=s'=+$, $l=l'=0$ and $\sigma=0$ in (\ref{ses1}) one obtains
\begin{eqnarray*}
\ds\lim_{a\to\infty}\sum_{i,i'}\big[\sum_{k=-a}^a\sum_{j=-2^a}^{2^a}\overline{\alpha_{i,j}^{s,l,k}} \alpha_{i',j}^{s',l',k+\sigma}\big]
\big[\sum_\psi \overline{\hat\psi_{i}^{(0)}}\hat\psi_{i'}^{(0)}\big]=\\
\ds=\lim_{a\to\infty}\big(a+1+\sum_{k=1}^a 2^{-k}\big)\big[\sum_\psi |\hat\psi_{0}^{(0)}|^2\big]+\sum_{r=0}^\infty \big[\sum_\psi |\hat\psi_{2^r}^{(0)}|^2\big].
\end{eqnarray*}
The last expression can be equal to $B$ if and only if $\hat\psi_{0}^{(0)}=\int \psi=0$ for every $\psi\in\Psi$, which is condition 1 in the statement, and $\sum_{r=0}^\infty \big[\sum_\psi |\hat\psi_{2^r}^{(0)}|^2\big]=B$, the condition (\ref{v00z}) for $\sigma=0$ in the statement.
Assuming then that $\hat\psi_{0}^{(0)}=0$ for every $\psi\in\Psi$, straightforward calculations with the $\alpha_{i,j}^{s,l,k}$'s lead to the fact that the conditions in (\ref{ses1}), for $s=s'=+$, are related with table \ref{table1}.

\begin{table}
\small
\begin{center}
\begin{tabular}{c|c|c|c|c|c|}
$s=s'=+$ & $2^0$ & $2^1$ & $2^2$ & $2^3$ & $\cdots$
\\ \hline
\begin{tabular}{c}$\mathbf{l,l'=0}$ \\ (0) \end{tabular} 
& \begin{tabular}{l} $k,k+\sigma=1$ \\ $j=0$ \\ $i,i'=2^0$ \end{tabular} 
& \begin{tabular}{l} $k,k+\sigma=2$ \\ $j=0$ \\ $i,i'=2^1$ \end{tabular} 
& \begin{tabular}{l} $k,k+\sigma=3$ \\ $j=0$ \\ $i,i'=2^2$ \end{tabular} 
& \begin{tabular}{l} $k,k+\sigma=4$ \\ $j=0$ \\ $i,i'=2^3$ \end{tabular} 
& $\cdots$
\\ \hline
\begin{tabular}{c}$\mathbf{l,l'=1}$ \\ (1) \end{tabular} 
& \begin{tabular}{l} $\mathbf{k,k+\sigma=0}$ \\ $\mathbf{j=1}$ \\ $\mathbf{i,i'=2^0}$ \end{tabular} 
& \begin{tabular}{l} $k,k+\sigma=1$ \\ $j=0$ \\ $i,i'=2^1+1$ \end{tabular} 
& \begin{tabular}{l} $k,k+\sigma=2$ \\ $j=0$ \\ $i,i'=2^2+1$ \end{tabular} 
& \begin{tabular}{l} $k,k+\sigma=3$ \\ $j=0$ \\ $i,i'=2^3+1$ \end{tabular} 
& $\cdots$
\\ \hline
\begin{tabular}{c}$\mathbf{l,l'=2}$ \\ (10) \end{tabular} 
& \begin{tabular}{l} $k,k+\sigma=-1$ \\ $j=2$ \\ $i,i'=2^0$ \end{tabular} 
& \begin{tabular}{l} $\mathbf{k,k+\sigma=0}$ \\ $\mathbf{j=1}$ \\ $\mathbf{i,i'=2^1}$ \end{tabular} 
& \begin{tabular}{l} $k,k+\sigma=1$ \\ $j=0$ \\ $i,i'=2^2+2$ \end{tabular} 
& \begin{tabular}{l} $k,k+\sigma=2$ \\ $j=0$ \\ $i,i'=2^3+2$ \end{tabular} 
& $\cdots$
\\ \hline
\begin{tabular}{c}$\mathbf{l,l'=3}$ \\ (11) \end{tabular} 
& \begin{tabular}{l} $k,k+\sigma=-1$ \\ $j=3$ \\ $i,i'=2^0$ \end{tabular} 
& \begin{tabular}{l} $\mathbf{k,k+\sigma=0}$ \\ $\mathbf{j=1}$ \\ $\mathbf{i,i'=2^1+1}$ \end{tabular} 
& \begin{tabular}{l} $k,k+\sigma=1$ \\ $j=0$ \\ $i,i'=2^2+3$ \end{tabular} 
& \begin{tabular}{l} $k,k+\sigma=2$ \\ $j=0$ \\ $i,i'=2^3+3$ \end{tabular} 
& $\cdots$
\\ \hline
\begin{tabular}{c}$\mathbf{l,l'=4}$ \\ (100) \end{tabular} 
& \begin{tabular}{l} $k,k+\sigma=-2$ \\ $j=4$ \\ $i,i'=2^0$ \end{tabular} 
& \begin{tabular}{l} $k,k+\sigma=-1$ \\ $j=2$ \\ $i,i'=2^1$ \end{tabular} 
& \begin{tabular}{l} $\mathbf{k,k+\sigma=0}$ \\ $\mathbf{j=1}$ \\ $\mathbf{i,i'=2^2}$ \end{tabular} 
& \begin{tabular}{l} $k,k+\sigma=1$ \\ $j=0$ \\ $i,i'=2^3+4$ \end{tabular} 
& $\cdots$
\\ \hline
\begin{tabular}{c}$\mathbf{l,l'=5}$ \\ (101) \end{tabular} 
& \begin{tabular}{l} $k,k+\sigma=-2$ \\ $j=5$ \\ $i,i'=2^0$ \end{tabular} 
& \begin{tabular}{l} $k,k+\sigma=-1$ \\ $j=2$ \\ $i,i'=2^1+1$ \end{tabular} 
& \begin{tabular}{l} $\mathbf{k,k+\sigma=0}$ \\ $\mathbf{j=1}$ \\ $\mathbf{i,i'=2^2+1}$ \end{tabular} 
& \begin{tabular}{l} $k,k+\sigma=1$ \\ $j=0$ \\ $i,i'=2^3+5$ \end{tabular} 
& $\cdots$
\\ \hline
\begin{tabular}{c}$\mathbf{l,l'=6}$ \\ (110) \end{tabular} 
& \begin{tabular}{l} $k,k+\sigma=-2$ \\ $j=6$ \\ $i,i'=2^0$ \end{tabular} 
& \begin{tabular}{l} $k,k+\sigma=-1$ \\ $j=3$ \\ $i,i'=2^1$ \end{tabular} 
& \begin{tabular}{l} $\mathbf{k,k+\sigma=0}$ \\ $\mathbf{j=1}$ \\ $\mathbf{i,i'=2^2+2}$ \end{tabular} 
& \begin{tabular}{l} $k,k+\sigma=1$ \\ $j=0$ \\ $i,i'=2^3+6$ \end{tabular} 
& $\cdots$
\\ \hline
\begin{tabular}{c}$\mathbf{l,l'=7}$ \\ (111) \end{tabular} 
& \begin{tabular}{l} $k,k+\sigma=-2$ \\ $j=7$ \\ $i,i'=2^0$ \end{tabular} 
& \begin{tabular}{l} $k,k+\sigma=-1$ \\ $j=3$ \\ $i,i'=2^1+1$ \end{tabular} 
& \begin{tabular}{l} $\mathbf{k,k+\sigma=0}$ \\ $\mathbf{j=1}$ \\ $\mathbf{i,i'=2^2+3}$ \end{tabular} 
& \begin{tabular}{l} $k,k+\sigma=1$ \\ $j=0$ \\ $i,i'=2^3+7$ \end{tabular} 
& $\cdots$
\\ \hline
$\vdots$ & $\vdots$ & $\vdots$ & $\vdots$ & $\vdots$ & $\ddots$
\\ \hline
\end{tabular}
\end{center}

\begin{center}
\begin{tabular}{c|c|c|c|}
$s=s'=+$ & $2^r,\,(0\leq r<p)$ & $2^p$ & $2^r,\,(r>p)$  
\\ \hline
\begin{tabular}{c}$\mathbf{l,l'>0}$ \\ ($l,l'=2^p+\sum_{t=0}^{p-1}l_t2^t$) \end{tabular} 
& \begin{tabular}{l} $k,k+\sigma=r-p$ \\ $j=\sum_{t=0}^{p-r}l_{r+t}2^t$ \\ $i,i'=2^r+\sum_{t=0}^{r-1}l_t2^t$ \end{tabular} 
& \begin{tabular}{l} $\mathbf{k,k+\sigma=0}$ \\ $\mathbf{j=1}$ \\ $\mathbf{i,i'=l,l'}$ \end{tabular} 
& \begin{tabular}{l} $k,k+\sigma=r-p$ \\ $j=0$ \\ $i,i'=2^r+l$ \end{tabular} 
\\ \hline
\end{tabular}
\end{center}
\caption{\label{table1} Distribution of indices for $s=s'=+$ in the set of equations (\ref{ses1}) using the Haar orthonormal bases $\big\{L_i^{(j)}\big\}$ and $\big\{K_{s,l}^{(k)}\big\}$ and the corresponding matrix $\big(\alpha_{i,j}^{s,l,k}\big)$ given in the appendix. See the proof of proposition \ref{prophv} for details.}
\end{table}

To obtain the conditions in (\ref{ses1}), for $s=s'=+$, table \ref{table1} is used as follows: Choose the values of $l$ and $l'$, and consider the corresponding files in the table. Choose the value of $\sigma$. In each entry of the table, the indices $k$ and $i$ are associated with $l$, and the indices $k+\sigma$ and $i'$ are associated with $l'$. One must pair the columns with the same $k$ and $j$ in both files, multiply $\overline{\hat\psi_i^{(0)}}$ by $\hat\psi_{i'}^{(0)}$ for the indices $i$, $i'$ selected in the paired columns, sum over $\psi\in\Psi$ and, finally, sum over all the paired columns.
For example, for $l=l'=0$ and $\sigma=0$ we arrive to the already known equation
$$%\be\label{000}
\sum_{k=1}^\infty \big[\sum_\psi |\hat\psi_{2^{k-1}}^{(0)}|^2\big]=B,
$$%\ee
the condition (\ref{v00z}) for $\sigma=0$ in the statement. For $l=l'=0$ and $\sigma>0$ we get (note that in this case $j=0$ in every column)
\be\label{00+}
\sum_{k=1}^\infty \big[\sum_\psi \overline{\hat\psi_{2^{k-1}}^{(0)}}\,\hat\psi_{2^{k-1+\sigma}}^{(0)}\big]=0.
\ee
For $l=l'=0$ and $\sigma<0$, the resultant condition 
$\sum_{k=1-\sigma}^\infty \big[\sum_\psi \overline{\hat\psi_{2^{k-1}}^{(0)}}\,\hat\psi_{2^{k-1+\sigma}}^{(0)}\big]=0$ coincides with (\ref{00+}) for $-\sigma$. Both of them correspond to the condition (\ref{v00z}) for $\sigma\neq 0$ in the statement.
For $l=l'=2^p+\sum_{t=0}^{p-1}l_t2^t$, ($p\geq0$),  and $\sigma=0$,
$$%\be\label{++0}
\sum_{k=-p}^{-1} \big[\sum_\psi |\hat\psi_{2^{p+k}+\sum_{t=0}^{p+k-1}l_t2^t}^{(0)}|^2\big]+
\big[\sum_\psi |\hat\psi_{l}^{(0)}|^2\big]+
\sum_{k=1}^\infty \big[\sum_\psi |\hat\psi_{2^{p+k}+l}^{(0)}|^2\big]=B,
$$%\ee
which is the condition (\ref{v++0}) in the statement.
For $l=l'=2^p+\sum_{t=0}^{p-1}l_t2^t$, ($p\geq0$),  and $\sigma>0$ (note that in this case there are different $j$'s in the first columns of table \ref{table1}),
\be\label{+++}
\sum_{k=1}^\infty \big[\sum_\psi \overline{\hat\psi_{2^{p+k}+l}^{(0)}}\,\hat\psi_{2^{p+k+\sigma}+l}^{(0)}\big]=0.
\ee
For $l=l'=2^p+\sum_{t=0}^{p-1}l_t2^t$, ($p\geq0$),  and $\sigma<0$, the condition $\sum_{k=1-\sigma}^\infty \big[\sum_\psi \overline{\hat\psi_{2^{p+k}+l}^{(0)}}\,\hat\psi_{2^{p+k+\sigma}+l}^{(0)}\big]=0$ coincides with (\ref{+++}). Both of them correspond to condition (\ref{v+++}) in the statement.
For $l=0$, $l'=2^p+\sum_{t=0}^{p-1}l'_t2^t$, ($p\geq0$),  and $\sigma\in\Z$,
$$%\be\label{0+z}
\sum_{k=\sup\{1,1-\sigma\}}^\infty \big[\sum_\psi \overline{\hat\psi_{2^{k-1}}^{(0)}}\,\hat\psi_{2^{p+k+\sigma}+l'}^{(0)}\big]=0,
$$%\ee
the condition (\ref{v0+z}) in the statement.
For $l=2^p+\sum_{t=0}^{p-1}l_t2^t$, $l'=2^{p'}+\sum_{t=0}^{p'-1}l'_t2^t$, ($p,p'\geq0$),  and $\sigma=0$,
$$%\be\label{+(+)0}
\begin{array}{c}
\ds\sum_{k=\sup\{-p,-p'\}}^{-1} \delta_{(\sum_{t=0}^{-k}l_{p+k+t}2^t)-(\sum_{t=0}^{-k}l'_{p'+k+t}2^t)}\,\big[\sum_\psi \overline{\hat\psi_{2^{p+k}+\sum_{t=0}^{p+k-1}l_t2^t}^{(0)}}\,\hat\psi_{2^{p'+k}+\sum_{t=0}^{p'+k-1}l'_t2^t}^{(0)}\big]+
\\
\ds +\big[\sum_\psi\overline{\hat\psi_{l}^{(0)}}\,\hat\psi_{l'}^{(0)}\big]+\sum_{k=1}^\infty \big[\sum_\psi \overline{\hat\psi_{2^{p+k}+l}^{(0)}}\,\hat\psi_{2^{p'+k}+l'}^{(0)}\big]=0,
\end{array} 
$$%\ee
which is the condition (\ref{v+(+)0}) in the statement.
For $l=2^p+\sum_{t=0}^{p-1}l_t2^t$, $l'=2^{p'}+\sum_{t=0}^{p'-1}l'_t2^t$, ($p,p'\geq0$),  and $\sigma\in\Z\backslash\{0\}$,
$$%\be\label{+(+)z}
\sum_{k=\sup\{1,1-\sigma\}}^\infty \big[\sum_\psi \overline{\hat\psi_{2^{p+k}+l}^{(0)}}\,\hat\psi_{2^{p'+k+\sigma}+l'}^{(0)}\big]=0,
$$%\ee
the condition (\ref{v+(+)z}) in the statement.
For $s=s'=-$, the conditions derived from the set of equations (\ref{ses1}) are equivalent to the former ones for $s=s'=+$. For $s=+$ and $s'=-$, or $s=-$ and $s'=+$, the set of equations (\ref{ses1}) leads to trivial conditions.
\end{proof}

\subsection{Hardy functions}\label{sect41}

It is not easy to handle the set of conditions for the vectors $\Psi_i\in\C^r$ in proposition \ref{prophv}. A better way to tackle this set of conditions consists in writing them in terms of Hardy functions in $H^+(\partial\D,\C^r)$ we next define.
In such approach, inner matrix functions and results by Halmos (lemma \ref{lwr} below) and Rovnyak (lemma \ref{lRovnyak}) play a central role.

Recall that ${\D}$ denotes the open unit disc of the complex plane $\C$ and $\partial {\D}$ its boundary, the unit circle.
Let $\H$ be a separable Hilbert space and denote by ${\mc L}(\H)$ the space of bounded operators on $\H$ (in the sequel we shall only need to consider Hilbert spaces of finite dimension, $\H=\C^r$, for which ${\mc L}(\H)$ can be identified with the space of complex $(r\times r)$-matrices). We denote by $H^2({\D};\H)$ the {\it Hardy class} of functions
$$
\tilde {\mf h}(\lambda)=\sum_{k=0}^\infty \lambda^k h_k,\qquad (\lambda\in {\D},\,h_k\in\H),
$$
with values in $\H$, such that $\sum ||h_k||^2_\H<\infty$.
For each function $\tilde {\mf h}\in H^2({\D};\H)$ the non-tangential limit in strong sense 
$$
\slim_{\lambda\to\omega} \tilde {\mf h}(\lambda)=\sum_{k=0}^\infty \omega^k h_k=: {\mf h}(\omega)
$$
exist for almost all $\omega\in {\partial {\D}}$. The functions $\tilde {\mf h}(\lambda)$ and $ {\mf h}(\omega)$ determine each other (they are connected by Poisson formula), so that we can identify $H^2({\D};\H)$ with a subspace of $L^2(\partial {\D};\H)$, say $H^+(\partial {\D};\H)$, thus providing $H^2({\D};\H)$ with the Hilbert space structure of $H^+(\partial {\D};\H)$ and embedding it in $L^2(\partial {\D};\H)$ as a subspace (the space $L^2(\partial {\D};\H)$ has been defined in section \ref{sect2}).

From now on, the operator ``multiplication by $\om$" on $H^+(\partial {\D};\H)$ shall be denoted by $M^+$, that is, 
\be\label{fshift}
[M^+{\mf h}](\om):=\om\cdot {\mf h}(\om)\,,\quad ({\mf h}\in H^+(\partial {\D};\H),\, \om\in\partial {\D})\,.
\ee
The operator $M^+$ is an isometry from $H^+(\partial {\D};\H)$ into $H^+(\partial {\D};\H)$.
A subspace $\M\subseteq H^+(\partial {\D};\H)$ is called a {\it wandering subspace for $M^+$} if $[M^+]^m\M\perp [M^+]^n\M$ whenever $m$ and $n$ are distinct non-negative integers.

Consider the subspace $\CC$ of $H^+(\partial {\D};\H)$ consisting of all constant functions, i.e., the functions ${\mf h}:\partial\D\to\H$ such that there exists a vector $h\in\H$ with ${\mf h}(\om)=h$ for a.e. $\om\in\partial {\D}$.
A weakly measurable\footnote{$A^+$ {\it weakly measurable} means the scalar product $(A^+(\omega)h,g)_\H$ is a Borel measurable scalar function on $\partial\D$ for each $h,g\in\H$.} operator-valued function 
$$
A^+:\partial {\D}\to{\mc L}(\H):\om\mapsto A(\om)
$$
is called\footnote{The name {\it rigid Taylor operator function} is introduced by Halmos \cite{HALMOS61}. The name {\it $M^+$-inner function} is used by Rosenblum and Rovnyak \cite{RR85} and co-workers.
}
 a {\it $M^+$-inner function} or {\it rigid Taylor operator function} if $A^+$ maps $\CC$ into $H^+(\partial {\D};\H)$ and $A^+(\omega)$ is for a.e. $\omega\in \partial {\D}$ a partial isometry\footnote{An operator $B\in{\mathcal L}(\H)$ is a {\it partial isometry} when there is a (closed) subspace $\M$ of $\H$ such that $||Bu||=||u||$ for $u\in\M$ and $Bv=0$ for $v\in\M^\perp$. In such case $\M$ is called the {\it initial space} of $B$.} on $\H$ with the same initial space.

According to Halmos \cite[lemma 5]{HALMOS61}, wandering subspaces for $M^+$ and $M^+$-inner functions (or rigid Taylor operator functions) are related as follows:

\begin{lemma}[Halmos]\label{lwr} 
A subspace $\M$ of $H^+(\partial {\D};\H)$ is a wandering subspace for $M^+$ if and only if there exists a $M^+$-inner function $A^+$ such that $\M=A^+\,\CC$.
The subspace $\M$ uniquely determines $A^+$ to within a constant partially isometric factor on the right.
\end{lemma}

Another fundamental result in what follows is due to Rovnyak \cite[lemma 5]{R62}:

\begin{lemma}[Rovnyak]\label{lRovnyak} 
If $\H$ has finite dimension $r$, there is no orthonormal set
${\mf h}_1,\ldots,{\mf h}_{r+1}$ in $H^+(\partial {\D};\H)$ containing $r+1$ elements and such that $\om^m{\mf h}_i(\om)$ is orthogonal to $\om^n{\mf h}_j(\om)$ whenever $m\neq n$.
\end{lemma}

In terms of Hardy functions proposition \ref{prophv} reads as follows:

\begin{prop}\label{prophh}
Under the conditions of proposition \ref{prophv}, for $l\geq0$ consider the Hardy function ${\mf h}_l\in H^+(\partial\D,\C^r)$ defined by
$$
\begin{array}{l}
\ds {\mf h}_0(\om):= \sum_{k=0}^\infty \om^k\,\Psi_{2^k}\,,
\\[2ex]
\ds {\mf h}_l(\om):= \sum_{k=0}^\infty \om^{k}\,\Psi_{2^{p+k+1}+l},\quad (l\geq1,\,l=2^p+\sum_{t=0}^{p-1}l_t2^t)\,.
\end{array}
$$
Then, the wavelet system $X$ of the form (\ref{ws}) generated by $\Psi$ is a tight frame, with frame bound $B$, if and only if the following conditions are satisfied: 
\be\label{h0}
\Psi_0=0\in\C^r,
\ee
\be\label{h1}
\<\om^m{\mf h}_{l'}(\om),\om^n{\mf h}_{l}(\om)\>_{H^+(\partial\D,\C^r)}=\delta_{m-n}\beta_{l,l'},\quad (l,l',m,n\in\N\cup\{0\}),
\ee
where
$$
\beta_{l,l'}:=\left\{\begin{array}{ll}
\ds B, & \text{ if }l,l'=0,
\\[1ex]
\ds 0, & \text{ if }l=0,l'\geq1,
\\[1ex]
\ds B-\sum_{k=-p}^{0} ||\Psi_{2^{p+k}+\sum_{t=0}^{p+k-1}l_t2^t}||^2_{\C^r}, & \text{ if }l=l'\geq 1,\, l=2^p+\sum_{t=0}^{p-1}l_t2^t.
\\[1ex]
\ds \begin{array}{l}\ds-\sum_{k=\sup\{-p,-p'\}}^{0} \delta_{(\sum_{t=0}^{-k}l_{p+k+t}2^t)-(\sum_{t=0}^{-k}l'_{p'+k+t}2^t)}\times\\ \ds \times\<\Psi_{2^{p'+k}+\sum_{t=0}^{p'+k-1}l'_t2^t},\Psi_{2^{p+k}+\sum_{t=0}^{p+k-1}l_t2^t}\>_{\C^r},\end{array}
&\ds \begin{array}{r}\ds \text{ if }l,l'\geq 1,\,l\neq l',\,l=2^p+\sum_{t=0}^{p-1}l_t2^t,\\ \ds l'=2^{p'}+\sum_{t=0}^{p'-1}l'_t2^t.\end{array}
\end{array}\right.
$$
\end{prop}

\begin{proof}
Condition 1 of Proposition \ref{prophv} coincides with (\ref{h0}), and conditions 2--4 in Proposition \ref{prophv}, i.e., equations (\ref{v00z})--(\ref{v+(+)z}), are equivalent to condition (\ref{h1}) of the statement. In detail, conditions (\ref{v00z}), (\ref{v+++}), (\ref{v0+z}) and (\ref{v+(+)z}), all of them for $\sigma\in\Z\backslash\{0\}$, are in correspondence with condition (\ref{h1}) for $m\neq n$; condition (\ref{v00z}) for $\sigma=0$ corresponds to the first line of the definition of $\beta_{l,l'}$ in condition (\ref{h1}); condition (\ref{v0+z}) for $\sigma=0$ corresponds to the second line of the definition of $\beta_{l,l'}$ in condition (\ref{h1}); condition (\ref{v++0}) corresponds to the third line of the definition of $\beta_{l,l'}$ in condition (\ref{h1}); finally, condition (\ref{v+(+)0}) corresponds to the fourth line of the definition of $\beta_{l,l'}$ in condition (\ref{h1}).
\end{proof}

%-------------------------------------------------------
\subsection{Case $r=1$: $\Psi=\{\psi\}\subset L^2[0,1]$}\label{sect42}

For $r=1$, proposition \ref{prophh} leads to the following:

\begin{coro}\label{coro17}
The only function $\psi\in L^2[0,1]$ such that the wavelet system $X$ of the form (\ref{ws}) generated by $\Psi=\{\psi\}$ is a tight frame for $L^2(\R)$, with frame bound $B$, is proportional to the Haar wavelet:
$$
\psi=\beta\,[\chi_{[0,1/2)}-\chi_{[1/2,1)}],
$$
where $\beta\in\C$ and $|\beta|^2=B$.
\end{coro}

\begin{proof}
Condition (\ref{h1}) in proposition \ref{prophh}, with $l=l'\geq 0$ and $m\neq n$, implies that each ${\mf h}_l$ is a scalar ($M^+$-)inner function in $H^+(\partial\D,\C)$, unless $\beta(l,l)=0$. Since $\beta(0,0)=B>0$, ${\mf h}_0$ is a scalar inner function in $H^+(\partial\D,\C)$. Condition (\ref{h1}) again, now with $m=n=0$, $l=0$ and $l'\geq1$, assures that ${\mf h}_0$ is orthogonal to every ${\mf h}_{l'}$, ($l'\geq1$). Then, according to Rovniak's lemma \ref{lRovnyak}, one has ${\mf h}_{l'}=0$ for every $l'\geq1$. 
From condition (\ref{h1}) once more,  with $m=n=0$ and $l=l'=1$, one gets $||\Psi_1||^2=|\hat\psi_1^{(0)}|^2=B$, so that  $||\Psi_l||^2=|\hat\psi_l^{(0)}|^2=0$ for $l\neq 1$.
\end{proof}

\begin{remark}\label{rm18}\rm
Corollary \ref{coro17} implies, in particular, that the only orthonormal wavelet $\psi\in L^2(\R)$ with $\text{supp}\,\psi\subseteq [0,1]$ is the Haar wavelet. This fact contradicts theorem 5 in \cite{GS11b}. The main problem in \cite{GS11b} is to check the completeness condition (ii) of corollary 3 there, and the sufficient condition given in item (2) of proposition 4 there, $\tilde\psi_{+,1}^{(0)}\neq 0$, fails to be right. According to corollary \ref{coro17} here, the completeness condition (ii) of corollary 3 in \cite{GS11b} is satisfied if and only if $\tilde\psi_{+,1}^{(0)}=1$. Thus, for the functions $\psi$ given in theorem 5 of \cite{GS11b}, save the Haar wavelet, the family $\{\psi_{m,n}:=D^mT^n\psi:m,n\in\Z\}$ is an orthonormal system of $L^2(\R)$, but it is not complete. 
\end{remark}

%-------------------------------------------------------
\subsection{Case $r=2$: $\Psi=\{\psi_1,\psi_2\}\subset L^2[0,1]$}\label{sect43}

For $r=2$, condition (\ref{h1}) in proposition \ref{prophh} implies, among other things, that the closed subspace generated by any subfamily of the set $\{{\mf h}_l\}_{l\in\N\cup\{0\}}$ is a wandering subspace for $M^+$ in $H^+_{\C^2}=H^+(\partial\D,\C^2)$.
According to Halmos's lemma \ref{lwr}, such a wandering subspace has at most dimension $2$ and is of the form $A^+\,\CC$,  for some rigid Taylor ($M^+$-inner) operator-valued function $A^+$, where $\CC$ denotes the subspace of constant functions. Thus, $A^+(\omega):\C^2\to\C^2$ is for a.e. $\omega\in \partial {\D}$ a partial isometry with the same initial subspace.
For non-zero $A^+$, such initial subspace, say $C\subset\C^2$, can have dimension $1$ or $2$.

Consider, in particular, the closed wandering subspace for $M^+$ generated by $\{{\mf h}_0,{\mf h}_1\}$ in $H^+_{\C^2}$ and the corresponding rigid Taylor ($M^+$-inner) operator-valued function $A^+$ with initial subspace $C\subset\C^2$. Condition (\ref{h1}) implies, in particular, that 
\be\label{ch1}
||{\mf h}_0||^2_{H^+_{\C^2}}=B\neq 0;\quad {\mf h}_0\perp {\mf h}_l,\quad (l\geq 1),
\ee
\be\label{ch2}
||{\mf h}_1||^2_{H^+_{\C^2}}=B-||\Psi_1||^2_{\C^2};\quad \<{\mf h}_l,{\mf h}_1\>_{H^+_{\C^2}}=-\<\Psi_l,\Psi_1\>_{\C^2},\quad (l\geq 2).
\ee
Then, if $\text{dim}\,C=1$, the only option is that ${\mf h}_0\neq 0$ and ${\mf h}_1=0$, so that $B=||\Psi_1||^2_{\C^2}$; in such case, conditions (\ref{h0}) and (\ref{h1}) are satisfied if and only if $\Psi_1\neq 0$ and $\Psi_l=0$ for every $l\neq1$ in $\N\cup\{0\}$. That is, if $\text{dim}\,C=1$, in $\Psi=\{\psi_1,\psi_2\}$ both functions $\psi_1$ and $\psi_2$ are proportional to the Haar wavelet.
Thus, the only possible non-trivial situation requires that $\text{dim}\,C=2$.
For $\text{dim}\,C=2$, the rigid Taylor ($M^+$-inner) operator-valued function $A^+$ can be written as a ($2\times 2$) matrix inner function 
$$
A^+(\om)=\left(\begin{array}{cc} {\mf a}^{(0)}_1(\om) & {\mf a}^{(1)}_1(\om)\\ {\mf a}^{(0)}_2(\om)& {\mf a}^{(1)}_2(\om)\end{array}\right),
$$
whose entries are functions belonging to the scalar Hardy space $H^+_{\C}=H^+(\partial\D,\C)$ and such that $A^+(\om)$ is unitary for a.e. $\om\in\partial\D$. In other words, the columns of $A^+$, say
$$
{\mf a}^{(0)}(\om):=\left(\begin{array}{c} {\mf a}^{(0)}_1(\om) \\ {\mf a}^{(0)}_2(\om)\end{array}\right),\quad 
{\mf a}^{(1)}(\om):=\left(\begin{array}{c} {\mf a}^{(1)}_1(\om) \\ {\mf a}^{(1)}_2(\om)\end{array}\right),
$$
are elements of $H^+_{\C^2}=H^+(\partial\D,\C^2)$ satisfying
\be\label{a0a1}
\<{\mf a}^{(i)},\om^m{\mf a}^{(j)}\>_{H^+_{\C^2}}=\delta_m\delta_{i-j},\quad  (m\in\Z),\, (i,j=0,1).
\ee
Since the closed subspace generated by $\{{\mf h}_0,{\mf h}_1\}$ coincides with $A^+\CC$, the functions ${\mf h}_0$ and ${\mf h}_1$ can be taken to be proportional to these vectors: ${\mf h}_0=B_0\,{\mf a}^{(0)}$ and  ${\mf h}_1=C_1\,{\mf a}^{(1)}$ for certain non-null constants $B_0,C_1\in\C$. Moreover, according to (\ref{ch1}), (\ref{ch2}) and Rovniak's lemma \ref{lRovnyak},
$$
{\mf h}_0(\om)=B_0\,{\mf a}^{(0)}(\om);\quad 
{\mf h}_l(\om)=C_l\,{\mf a}^{(1)}(\om),\quad (l\geq 1),
$$
where the constants $B_0,C_l\in\C$ must satisfy
$$
|B_0|^2=B\neq0,\quad |C_1|^2=B-||\Psi_1||^2_{\C^2}\neq0,
$$
$$
C_l\,\overline{C_1}=\<{\mf h}_l,{\mf h}_1\>_{H^+_{\C^2}}=-\<\Psi_l,\Psi_1\>_{\C^2},\quad (l\geq 2).
$$

Consider the Taylor-Fourier series
$$
{\mf a}^{(0)}(\om)=\sum_{k=0}^\infty \om^k\,a^{(0)}_k,\quad
{\mf a}^{(1)}(\om)=\sum_{k=0}^\infty \om^k\,a^{(1)}_k,
$$
where $a^{(0)}_k,a^{(1)}_k\in\C^2$, ($k\in\N\cup\{0\}$), and 
$$
{\mf h}_0(\om)=\sum_{k=0}^\infty \om^k\,\Psi_{2^k}=\sum_{k=0}^\infty \om^k\,
\left(\begin{array}{c} [\hat\psi_1]_{2^k}^{(0)} \\[1ex] [\hat\psi_2]_{2^k}^{(0)}\end{array}\right)={B_0}\sum_{k=0}^\infty \om^k\,a^{(0)}_k,
$$
$$
{\mf h}_1(\om)=\sum_{k=0}^\infty \om^k\,\Psi_{2^{k+1}+1}=\sum_{k=0}^\infty \om^k\,
\left(\begin{array}{c} [\hat\psi_1]_{2^{k+1}+1}^{(0)} \\[1ex] [\hat\psi_2]_{2^{k+1}+1}^{(0)}\end{array}\right)={C_1}\sum_{k=0}^\infty \om^k\,a^{(1)}_k.
$$

In these terms, 
\be\label{c1}
\quad|C_1|^2 = B-||\Psi_1||_{\C^2}^2)=B(1-||a^{(0)}_0||_{\C^2}^2)\neq0,
\ee
\be\label{c2}
C_{2^k} = -\frac{1}{\overline{C_1}}\,\<\Psi_{2^k},\Psi_1\>_{\C^2}=
-\frac{B}{\overline{C_1}}\,\<a^{(0)}_k,a^{(0)}_0\>_{\C^2},\quad (k\in\N),
\ee
\be\label{c21}
C_{2^k+1}=-\frac{1}{\overline{C_1}}\,\<\Psi_{2^k+1},\Psi_1\>_{\C^2}=
-\frac{C_1\overline{B_0}}{\overline{C_1}}\,\<a^{(1)}_{k-1},a^{(0)}_0\>_{\C^2},\quad (k\in\N).
\ee
Any other $C_l$, ($l\geq2$), can be expressed in terms of the $C_{2^k}$ and $C_{2^k+1}$:

\begin{lemma}\label{lemacs0}
For $l\geq2$, $l\neq 2^k,\,2^k+1$, with $l=2^p+2^{p_1}+2^{p_2}+\cdots+2^{p_s}$, where $p>p_1>p_2>\cdots>p_s\geq0$, one has
\be\label{clg}
C_l=\left\{\begin{array}{ll}
C_1^{-(s-1)}C_{2^{p-p_1}+1}C_{2^{p_1-p_2}+1}\cdots C_{2^{p_{s-2}-p_{s-1}}+1}C_{2^{p_{s-1}}+1}, & \text{ if }p_s=0\,\,(\text{if $l$ is odd}),
\\[2ex]
C_1^{-s}C_{2^{p-p_1}+1}C_{2^{p_1-p_2}+1}\cdots C_{2^{p_{s-1}-p_{s}}+1}C_{2^{p_{s}}}, & \text{ if }p_s>0\,\,(\text{if $l$ is even}),
\end{array}\right.
\ee
\end{lemma}

\begin{proof}
For $l\geq2$, $l\neq 2^k,\,2^k+1$, so that $l=2^p+l_1$ with $l_1\neq0,1$, and $l_1=2^{p_1}+l_2$, since ${\mf h}_{l_1}(\om):= \sum_{k=0}^\infty \om^{k}\,\Psi_{2^{p_1+k+1}+l_1}$, the vector $\Psi_l=\Psi_{2^p+l_1}$ is the $k$-coefficient in the Taylor series of ${\mf h}_{l_1}$, where  $k=p-p_1-1$. Thus, being ${\mf h}_{l_1}=C_{l_1}{\mf a}^{(1)}$,
\be\label{clcl1}
C_l=-\frac{1}{\overline{C_1}}\,\<\Psi_l,\Psi_1\>_{\C^2}=
-\frac{1}{\overline{C_1}}\,\<C_{l_1}a^{(1)}_{p-p_1-1},B_0a^{(0)}_0\>_{\C^2}=\frac{C_{2^{p-p_1}+1}}{C_1}\,C_{l_1}.
\ee
The same argument for $l_1,l_2,\ldots,l_{p_{s-2}}$ leads to
$$
C_l=\frac{ C_{2^{p-p_1}+1}C_{2^{p_1-p_2}+1}\cdots C_{2^{p_{s-2}-p_{s-1}}+1}}{C_1^{s-1}}\,C_{2^{p_{s-1}}+2^{p_s}}.
$$
In the final step, if $p_s=0$, then $C_{2^{p_{s-1}}+2^{p_s}}=C_{2^{p_{s-1}}+1}$; on the other hand, if $p_s>0$, as before,
$$
C_{2^{p_{s-1}}+2^{p_s}}=\frac{C_{2^{p_{s-1}-p_s}+1}}{C_1}\,C_{2^{p_s}}.
$$
\end{proof}

In general, since ${\mf h}_l=C_l\,{\mf a}^{(1)}$ for $l\geq1$,
$$
||{\mf h}_{l}||_{H^+_{\C^2}}=|C_l|^2\,||{\mf a}^{(1)}||_{H^+_{\C^2}}=|C_l|^2.
$$
On the other hand, equation (\ref{h1}) with $m=n$ and $l=l'\geq 1$, $l=2^p+\sum_{t=0}^{p-1}l_t2^t$, implies
$$
||{\mf h}_{l}||_{H^+_{\C^2}}=B-\sum_{k=-p}^{0} ||\Psi_{2^{p+k}+\sum_{t=0}^{p+k-1}l_t2^t}||^2_{\C^2}.
$$
In particular, for $l=2^p$ with $p\geq0$,
\be\label{cm2}
|C_{2^p}|^2=||{\mf h}_{2^p}||_{H^+_{\C^2}}=B-\sum_{k=0}^p ||\Psi_{2^k}||^2_{\C^2}=B\big(1-\sum_{k=0}^p ||a^{(0)}_k||^2_{\C^2}\big),\quad (p\geq0).
\ee
And, for $l=2^p+1$ with $p>0$,
\be\label{cm21}
\begin{array}{rl}
\ds |C_{2^p+1}|^2=||{\mf h}_{2^p+1}||_{H^+_{\C^2}}&
\ds =B-||\Psi_{1}||^2_{\C^2}-\sum_{k=1}^p ||\Psi_{2^k+1}||^2_{\C^2}=
\\[1ex]
&\ds =B\big(1-||a^{(0)}_0||^2_{\C^2}\big)\big(1-\sum_{k=0}^{p-1} ||a^{(1)}_k||^2_{\C^2}\big)\,,\quad (p>0)\,.
\end{array}
\ee

In a similar way, from equation (\ref{h1}), now with $m=n$ and $l\neq l'$, $l=2^p$ or $l=2^p+1$, $l'=2^{p'}$ or $l'=2^{p'}+1$, one gets
\be\label{c2c2}
C_{2^{p'}}\overline{C_{2^{p}}}=
\<{\mf h}_{2^{p'}},{\mf h}_{2^{p}}\>_{H^+_{\C^2}}=
-\sum_{r=0}^{p} \<\Psi_{2^{p'-r}},\Psi_{2^{p-r}}\>_{\C^2}=
-B\sum_{r=0}^{p} \<a^{(0)}_{p'-r},a^{(0)}_{p-r}\>_{\C^2},
\quad (0\leq p<p'),
\ee
\be\label{c21c21}
\begin{array}{rl}
\ds C_{2^{p'}+1}\overline{C_{2^{p}+1}} &
\ds =\<{\mf h}_{2^{p'}+1},{\mf h}_{2^{p}+1}\>_{H^+_{\C^2}}=
-\sum_{r=0}^{p-1} \<\Psi_{2^{p'-r}+1},\Psi_{2^{p-r}+1}\>=
\\[2ex]
&\ds =-B\big(1-||a^{(0)}_0||^2_{\C^2}\big)\sum_{r=0}^{p-1} \<a^{(1)}_{p'-r-1},a^{(1)}_{p-r-1}\>_{\C^2},
\quad (0<p<p'),
\end{array}
\ee
\be\label{c21c2}
\begin{array}{rl}
\ds C_{2^{p'}+1}\overline{C_{2^{p}}}&
\ds =\<{\mf h}_{2^{p'}+1},{\mf h}_{2^{p}}\>_{H^+_{\C^2}}=
\\
&\ds =\left\{\begin{array}{ll}
\ds -\sum_{r=0}^{p} \<\Psi_{2^{p'-r}+1},\Psi_{2^{p-r}}\>_{\C^2}=
-C_1\overline{B_0}\sum_{r=0}^{p} \<a^{(1)}_{p'-r-1},a^{(0)}_{p-r}\>_{\C^2},
 & \ds \text{ if } 0\leq p<p',
\\[2ex]
\ds -\sum_{r=0}^{p'-1} \<\Psi_{2^{p'-r}+1},\Psi_{2^{p-r}}\>_{\C^2}= 
-C_1\overline{B_0}\sum_{r=0}^{p'-1}\<a^{(1)}_{p'-r-1},a^{(0)}_{p-r}\>_{\C^2},
& \ds \text{ if } 0<p'\leq p,
\end{array}\right.
\end{array}
\ee

Let us note that (\ref{c1}) coincides with (\ref{cm2}) for $p'=0$, (\ref{c2}) is (\ref{c2c2}) for $p=0$ and $p'=k\in\N$, and (\ref{c21}) is (\ref{c21c2}) for $p=0$ and $p'=k\in\N$.

Conditions (\ref{cm2})--(\ref{c21c2}) are part of the conditions (\ref{h1}) in proposition \ref{prophh}, but they are sufficient for the family of Hardy functions $\{{\mf h}_l\}_{l\in\N\cup\{0\}}\subset H^+_{\C^2}$ to satisfy the complete set of conditions (\ref{h1}):

\begin{prop}\label{pclc2}
Let
$$
{\mf a}^{(0)}(\om)=\sum_{k=0}^\infty \om^k\,a^{(0)}_k,\quad
{\mf a}^{(1)}(\om)=\sum_{k=0}^\infty \om^k\,a^{(1)}_k
$$
be a pair of functions in $H^+_{\C^2}=H^+(\partial\D,\C^2)$ satisfying (\ref{a0a1}) and $||a^{(0)}_0||_{\C^2}<1$.
Let $B_0\in\C$ such that 
\be\label{b0}
|B_0|^2=B
\ee
and let $\{C_l\}_{l\in\N}\subset\C$ be a sequence verifying (\ref{cm2})--(\ref{c21c2}).
For $l\in\N\cup\{0\}$, define the Hardy function ${\mf h}_l\in H^+_{\C^2}$ by 
$$
\begin{array}{l}
\ds {\mf h}_0(\om)= \sum_{k=0}^\infty \om^k\,\Psi_{2^k}:=B_0\,{\mf a}^{(0)}(\om)= B_0\,\sum_{k=0}^\infty \om^k\,a^{(0)}_k\,,
\\[2ex]
\ds {\mf h}_l(\om)= \sum_{k=0}^\infty \om^{k}\,\Psi_{2^{p+k+1}+l}:=C_l\,{\mf a}^{(1)}(\om)=C_l\,\sum_{k=0}^\infty \om^k\,a^{(1)}_k,\quad (l\geq1,\,l=2^p+\sum_{t=0}^{p-1}l_t2^t)\,.

\end{array}
$$
Then the family $\{{\mf h}_l\}_{l\in\N\cup\{0\}}\subset H^+_{\C^2}$ satisfies the complete set of conditions (\ref{h1}).
\end{prop}

\begin{proof}
(\ref{a0a1}) for $m\neq0$ implies (\ref{h1}) for $m\neq n$. 
(\ref{a0a1}) for $m=0$ and $i=j=0$, together with (\ref{b0}), coincide with (\ref{h1}) for $m=n$ and $l=l'=0$. 
(\ref{a0a1}) for $m=0$, $i=0$ and $j=1$ leads to (\ref{h1}) for $m=n$, $l=0$ and $l'\geq1$.
(\ref{cm2})--(\ref{c21c2}) are just (\ref{h1}) for $m=n$ and $l,l'\geq1$, $l=2^p$ or $l=2^p+1$, $l'=2^{p'}$ or $l'=2^{p'}+1$.

For $m=n$ and $l=l'\geq2$, $l=2^p+\sum_{t=0}^{p-1}l_t2^t$, $l\neq 2^k,\,2^k+1$, so that $l=2^p+l_1$  with $l_1\neq0,1$, condition (\ref{h1}) is
$$
\begin{array}{rl}
\ds |C_{l}|^2=\<{\mf h}_{l},{\mf h}_{l}\>_{H^+_{\C^2}} &
\ds =B-\sum_{k=-p}^0 ||\Psi_{2^{p+k}+\sum_{t=0}^{p+k-1}l_t2^t}||^2_{\C^r}=
\\[2ex]
& \ds =|C_{l_1}|^2-\sum_{k=-(p-p_1-1)}^0 ||\Psi_{2^{p+k}+l_1}||^2_{\C^r}=
\\[2ex]
& \ds =|C_{l_1}|^2\,\big(1-\sum_{r=0}^{p-p_1-1} ||a^{(1)}_r||^2_{\C^r}\big)=|C_{l_1}|^2\,\frac{|C_{2^{p-p_1}+1}|^2}{B(1-||a^{(0)}_0||^2_{\C^2})},
\end{array}
$$
which is true, due to (\ref{clcl1}) together with (\ref{cm2}) for $p=1$.

For $m=n$ and $l,l'\geq2$, $l\neq l'$, $l,l'\neq 2^k,\,2^k+1$, so that $l=2^p+l_1$, $l'=2^{p'}+l'_1$ with $l_1,l'_1\neq0,1$, and $l_1=2^{p_1}+l_2$, $l'_1=2^{p'_1}+l'_2$, condition (\ref{h1}) reads
$$
\begin{array}{rl}
\ds C_{l'}\overline{C_l}=\<{\mf h}_{l'},{\mf h}_{l}\>_{H^+_{\C^2}} &
\ds =\delta_{(p-p_1)-(p'-p'_1)}\<{\mf h}_{l'_1},{\mf h}_{l_1}\>_{H^+_{\C^2}}-
\sum_{k=\sup\{p_1-p+1,p'_1-p'+1\}}^{0}\<\Psi_{2^{p'+k}+l'_1},\Psi_{2^{p+k}+l_1}\>_{\C^2}=
\\[2ex]
& \ds =C_{l'_1}\overline{C_{l_1}}\big(\delta_{(p-p_1)-(p'-p'_1)}-\sum_{r=0}^{\inf\{p-p_1-1,p'-p'_1-1\}}\<a^{(1)}_{p'-r-p'_1-1},a^{(1)}_{p-r-p_1-1}\>_{\C^2}\big),
\end{array}
$$
which, by (\ref{clcl1}), coincides with (\ref{cm21}) when $p-p_1=p'-p_1'$ and coincides with (\ref{c21c21}) when $p-p_1\neq p'-p_1'$.

And in a similar way for the two remaining cases: for $m=n$, $l=2^p$ or $l=2^p+1$, and $l'\neq 2^k,\,2^k+1$.
\end{proof}

Proposition \ref{pclc2} says that everything can be written in terms of the $M^+$-inner matrix $A^+$ or, equivalently, in terms of the pair of functions ${\mf a}^{(0)}$ and ${\mf a}^{(1)}$ of $H^+_{\C^2}=H^+(\partial\D,\C^2)$:

\begin{prop}\label{pclc2a}
Let
$$
{\mf a}^{(0)}(\om)=\sum_{k=0}^\infty \om^k\,a^{(0)}_k\quad\text{ and }\quad
{\mf a}^{(1)}(\om)=\sum_{k=0}^\infty \om^k\,a^{(1)}_k
$$
be a pair of functions in $H^+_{\C^2}$ satisfying (\ref{a0a1}), $||a^{(0)}_0||_{\C^2}<1$, and such that
\be\label{cm2x}
\frac{\big|\<a^{(0)}_p,a^{(0)}_0\>_{\C^2}\big|^2}{1-||a^{(0)}_0||^2_{\C^2}}=1-\sum_{k=0}^p ||a^{(0)}_k||^2_{\C^2},\quad (p>0),
\ee
\be\label{cm21x}
\frac{\big|\<a^{(1)}_p,a^{(0)}_0\>_{\C^2}\big|^2}{1-||a^{(0)}_0||^2_{\C^2}}=
1-\sum_{k=0}^{p} ||a^{(1)}_k||^2_{\C^2},\quad (p\geq0),
\ee
\be\label{c2c2x}
\frac{\<a^{(0)}_{p'},a^{(0)}_0\>_{\C^2}\<a^{(0)}_{0},a^{(0)}_p\>_{\C^2}}{1-||a^{(0)}_0||^2_{\C^2}}=
-\sum_{r=0}^{p} \<a^{(0)}_{p'-r},a^{(0)}_{p-r}\>_{\C^2},
\quad (0<p<p'),
\ee
\be\label{c21c21x}
\frac{\<a^{(1)}_{p'},a^{(0)}_0\>_{\C^2}\<a^{(0)}_{0},a^{(1)}_p\>_{\C^2}}{1-||a^{(0)}_0||^2_{\C^2}}= 
-\sum_{r=0}^{p} \<a^{(1)}_{p'-r},a^{(1)}_{p-r}\>_{\C^2},
\quad (0\leq p<p'),
\ee
\be\label{c21c2x}
\frac{\<a^{(1)}_{p'},a^{(0)}_0\>_{\C^2}\<a^{(0)}_{0},a^{(0)}_p\>_{\C^2}}{1-||a^{(0)}_0||^2_{\C^2}}=
\left\{\begin{array}{ll}
\ds -\sum_{r=0}^{p} \<a^{(1)}_{p'-r},a^{(0)}_{p-r}\>_{\C^2},
 & \ds \text{ if } 0< p\leq p',
\\[2ex]
\ds -\sum_{r=0}^{p'}\<a^{(1)}_{p'-r},a^{(0)}_{p-r}\>_{\C^2},
& \ds \text{ if } 0\leq p'< p.
\end{array}\right.
\ee
Let $B_0\in\C$ such that 
\be\label{b0x}
|B_0|^2=B
\ee
and let $\{C_l\}_{l\in\N}\subset\C$ be a sequence of scalars, where 
\be\label{c1x}
\quad|C_1|^2 =B(1-||a^{(0)}_0||_{\C^2}^2)\neq0,
\ee
\be\label{c2x}
C_{2^k} = -\frac{B}{\overline{C_1}}\,\<a^{(0)}_k,a^{(0)}_0\>_{\C^2},\quad (k\in\N),
\ee
\be\label{c21x}
C_{2^k+1}=-\frac{C_1\overline{B_0}}{\overline{C_1}}\,\<a^{(1)}_{k-1},a^{(0)}_0\>_{\C^2},\quad (k\in\N),
\ee
and the rest of $C_l$'s are given by (\ref{clg}).
For $l\in\N\cup\{0\}$, define the Hardy function ${\mf h}_l\in H^+_{\C^2}$ by 
$$
\begin{array}{l}
\ds {\mf h}_0(\om)= \sum_{k=0}^\infty \om^k\,\Psi_{2^k}:=B_0\,{\mf a}^{(0)}(\om)= B_0\,\sum_{k=0}^\infty \om^k\,a^{(0)}_k\,,
\\[2ex]
\ds {\mf h}_l(\om)= \sum_{k=0}^\infty \om^{k}\,\Psi_{2^{p+k+1}+l}:=C_l\,{\mf a}^{(1)}(\om)=C_l\,\sum_{k=0}^\infty \om^k\,a^{(1)}_k,\quad (l\geq1,\,l=2^p+\sum_{t=0}^{p-1}l_t2^t)\,.
\end{array}
$$
Then the family $\{{\mf h}_l\}_{l\in\N\cup\{0\}}\subset H^+_{\C^2}$ satisfies the complete set of conditions (\ref{h1}).
\end{prop}

\begin{proof}
Conditions (\ref{cm2x})--(\ref{c21c2x}) are just conditions (\ref{cm2})--(\ref{c21c2}) where the $C$'s are eliminated using (\ref{c1})--(\ref{c21}), except condition (\ref{cm2}) for $p=0$, condition (\ref{c2c2}) for $p=0$ and condition (\ref{c21c2}) for $p=0$. These three excluded conditions of proposition \ref{pclc2} are the added relations (\ref{c1x})--(\ref{c21x}) in proposition \ref{pclc2a} to define the sequence $\{C_l\}_{l\in\N}$ (which coincide with (\ref{c1})--(\ref{c21})).    
\end{proof}

Taking into account that $1=||{\mf a}^{(j)}||_{H^+_{\C^2}}=\sum_{k=0}^\infty ||a^{(j)}_k||^2_{\C^2}$, ($j=0,1$), (see (\ref{a0a1})), from (\ref{cm2}) and (\ref{cm21}) one deduces that, given $p\geq1$,
\be\label{cm2-0}
C_{2^p}=0\quad \Leftrightarrow\quad C_{2^k}=0\text{ for all }k\geq p\quad \Leftrightarrow\quad a^{(0)}_{k}=0\text{ for all }k> p,
\ee
\be\label{cm21-0}
C_{2^p+1}=0\quad \Leftrightarrow\quad C_{2^k+1}=0\text{ for all }k\geq p\quad \Leftrightarrow\quad a^{(1)}_{k}=0\text{ for all }k\geq p.
\ee

Now we are ready to obtain all the non-trivial families $\{{\mf h}_l\}_{l\in\N\cup\{0\}}\subset H^+_{\C^2}$ that satisfy the complete set of conditions (\ref{h1}) or, equivalently, all the $M^+$-inner ($2\times 2$)-matrix functions
$A^+(\om)$ verifying the conditions (\ref{cm2x})--(\ref{c21c2x}).
For the sake of clarity, we collect the results in propositions \ref{prop22} and \ref{prop23} below.

When $a^{(0)}_0=0$, relations (\ref{c2x}) and (\ref{c21x}) imply that  $C_{2^k}=C_{2^k+1}=0$ for every $k\geq1$, so that, by (\ref{cm2-0}) and (\ref{cm21-0}),  $a^{(0)}_{k}=0$ for every $k>1$, $a^{(1)}_{k}=0$ for every $k>0$, and, by (\ref{clg}), $C_l=0$ for every $l\geq2$. In this case, $\{a^{(0)}_1,a^{(1)}_0\}$ is an orthonormal basis of $\C^2$ (by (\ref{c21x}) for $k=1$, since $C_3=0$) and $|B_0|^2=|C_1|^2=B$. This leads to the families of {\it type 2} in propositions \ref{prop22} and \ref{prop23} below, where $u_0=a^{(0)}_1$ and $u_1=a^{(1)}_0$. 
It is trivial to check that this type of families of Hardy functions satisfies conditions (\ref{h0}) and (\ref{h1}) of proposition \ref{prophh}.

In what follows, we assume that $0<||a^{(0)}_0||<1$.

If $C_2=0$, 
\be\label{c20-1}
{\mf h}_0(\om)=\Psi_1+\om\,\Psi_{2}=B_0(a^{(0)}_0+\om\,a^{(0)}_1),
\quad ||a^{(0)}_1||^2_{\C^2}=1-||a^{(0)}_0||^2_{\C^2},
\quad \text{[from (\ref{cm2}) with $p=1$]};
\ee
\be\label{c20-2}
\<a^{(0)}_1,a^{(0)}_0\>_{\C^2}=0,
\quad \text{[by (\ref{c2}) with $k=1$ or (\ref{c2c2}) with $p=0$, $p'=1$]};
\ee
\be\label{c20-22}
\<a^{(1)}_0,a^{(0)}_1\>_{\C^2}=0,
\quad \text{[by (\ref{c21c2}) with $1=p=p'$]};
\ee
\be\label{c20-3}
\<a^{(1)}_{p'},a^{(0)}_1\>_{\C^2}=-\<a^{(1)}_{p'-1},a^{(0)}_0\>_{\C^2},\quad (p'>0),
\quad \text{[by (\ref{c21c2}) with $1=p<p'$]}.
\ee

If $C_3=0$, 
\be\label{c30-1}
{\mf h}_1(\om)=\Psi_3=C_1\,a^{(1)}_0,\quad ||a^{(1)}_0||^2_{\C^2}=1,
\quad \text{[from (\ref{cm21}) with $p=1$]}; 
\ee
\be\label{c30-2}
\<a^{(1)}_0,a^{(0)}_{p}\>_{\C^2}=0\quad (p\geq0),
\quad \text{[by (\ref{c21}) with $k=1$ and (\ref{c21c2}) with $1=p'\leq p$]}.
\ee

Thus, if $C_2=C_3=0$, by (\ref{c20-2}) and (\ref{c30-2}), the three non-null vectors $a^{(0)}_0,a^{(0)}_1,a^{(1)}_0$ should be orthogonal to each other in $\C^2$, which is not possible.

If $C_3\neq0$, 
\be\label{c3n0-1}
\<a^{(1)}_0,a^{(0)}_0\>_{\C^2}\neq0,
\quad \text{[by (\ref{c21}) with $k=1$]},
\ee
and equating the expressions for $C_{2^{p'}+1}\overline{C_{3}}$ obtained using (\ref{c21}) and (\ref{c21c21}) with $1=p<p'$, i.e., by (\ref{c21c21x}) with $1=p<p'$,
\be\label{c3n0-2}
\frac{\<a^{(1)}_{p'},a^{(0)}_0\>_{\C^2}\,\<a^{(0)}_0,a^{(1)}_{0}\>_{\C^2}}{1-||a^{(0)}_0||^2_{\C^2}}=-\<a^{(1)}_{p'},a^{(1)}_0\>_{\C^2},
\quad (p'>0). 
\ee
Due to (\ref{c3n0-1}), $\frac{\<a^{(1)}_0,a^{(0)}_{0}\>_{\C^2}}{1-||a^{(0)}_0||^2_{\C^2}}a^{(0)}_{0}+a^{(1)}_0\neq0$, and (\ref{c3n0-2}) is equivalent to
\be\label{c3n0-3}
a^{(1)}_p\perp \frac{\<a^{(1)}_0,a^{(0)}_{0}\>_{\C^2}}{1-||a^{(0)}_0||^2_{\C^2}}a^{(0)}_{0}+a^{(1)}_0\neq0,
\quad (p>0). 
\ee

Now, if $C_2=0$ and $C_3\neq0$, by (\ref{c20-2}), (\ref{c20-22}) and (\ref{c3n0-1}),
\be\label{c20c3n0-1}
a^{(1)}_0=\la_0\,a^{(0)}_0,\quad (0\neq\la_0\in\C),\quad 
\frac{\<a^{(1)}_0,a^{(0)}_{0}\>_{\C^2}}{1-||a^{(0)}_0||^2_{\C^2}}a^{(0)}_{0}+a^{(1)}_0=\frac{\la_0}{1-||a^{(0)}_0||^2_{\C^2}}a^{(0)}_{0};
\ee
by (\ref{c20-2})--(\ref{c20-3}), (\ref{c3n0-3}) and (\ref{c20c3n0-1}),
\be\label{c20c3n0-2}
a^{(1)}_1=\la_1\,a^{(0)}_1,\quad (0\neq\la_1\in\C),\quad 
\<a^{(1)}_1,a^{(0)}_1\>_{\C^2}=-\<a^{(1)}_{0},a^{(0)}_0\>_{\C^2},
\quad a^{(1)}_p=0\text{ for }p>1;
\ee
\be\label{c20c3n0-3}
\la_1=-\frac{||a^{(0)}_0||^2_{\C^2}}{1-||a^{(0)}_0||^2_{\C^2}}\,\la_0,
\quad |\la_0|^2=\frac{1-||a^{(0)}_0||^2_{\C^2}}{||a^{(0)}_0||^2_{\C^2}},
\quad \text{[by (\ref{cm2}), (\ref{cm21}), (\ref{c20c3n0-1}) and (\ref{c20c3n0-2})]};
\ee
Here, being $a^{(0)}_p=a^{(1)}_p=0$ for $p>1$, by (\ref{c2}) and (\ref{c21}), $C_{2^k}=0$ for $k>0$ and $C_{2^k+1}=0$ for $k>1$. Thus, by  (\ref{clg}), $C_l\neq0$ only if $l=2^p-1$, ($p>0$), and
$$
C_{2^p-1} = \frac{C_3^{p-1}}{C_1^{p-2}}=
\frac{1}{C_1^{p-2}}\big(-\frac{C_1\overline{B_0}}{\overline{C_1}}\,\<a^{(1)}_{0},a^{(0)}_0\>_{\C^2}\big)^{p-1}=
C_1\big(-\frac{\overline{B_0}\,||a^{(0)}_0||^2_{\C^2}\,\la_0}{\overline{C_1}}\big)^{p-1}\,,\quad (p>1).
$$
This case leads to the families of {\it type 3} in propositions \ref{prop22} and \ref{prop23} below, with $\rho=||a^{(0)}_0||_{\C^2}$, $a^{(0)}_0=\rho u_0$, $a^{(0)}_1=(1-\rho^2)^{1/2} u_1$ and $\theta=\text{arg}(\la_0)$.

If $C_2\neq0$, 
\be\label{c2n0-1}
\<a^{(0)}_1,a^{(0)}_0\>_{\C^2}\neq0,
\quad \text{[by (\ref{c2}) with $k=1$]},
\ee
and equating the expressions for $C_{2^{p'}}\overline{C_{2}}$ obtained using (\ref{c2}) and (\ref{c2c2}) with $1=p<p'$, i.e., by (\ref{c2c2x}) with $1=p<p'$,
\be\label{c2n0-2}
\frac{\<a^{(0)}_{p'},a^{(0)}_0\>_{\C^2}\,\<a^{(0)}_0,a^{(0)}_{1}\>_{\C^2}}{1-||a^{(0)}_0||^2_{\C^2}}=-\<a^{(0)}_{p'},a^{(0)}_1\>_{\C^2}-\<a^{(0)}_{p'-1},a^{(0)}_0\>_{\C^2},
\quad (p'>1). 
\ee

Then, when $C_2\neq0$ and $C_3=0$, 
\be\label{c2n0c30-1}
a^{(0)}_p=\gamma_p a^{(0)}_0, \text{ where }\gamma_p\in\C,\quad (p>0),\quad 
\text{[by (\ref{c30-2})]};
\ee
\be\label{c2n0c30-2}
\gamma_1\neq0,\quad \gamma_p=\big(-\frac{1-||a^{(0)}_0||^2_{\C^2}}{\overline{\gamma_1}}\big)^{p-1}\,\gamma_1,\quad (p>0),\quad
\text{[by (\ref{c2n0-1}), (\ref{c2n0-2}) and (\ref{c2n0c30-1})]}.
\ee
In this case, relations (\ref{a0a1}) for $i=j=0$ and $m\in\Z$, together with (\ref{c2n0c30-2}), lead to $|\gamma_1|=\frac{1-||a^{(0)}_0||^2_{\C^2}}{||a^{(0)}_0||_{\C^2}}$. If $\theta=\text{arg}(\gamma_1)$,
\be\label{c2n0c30-3}
\gamma_p=\big(-||a^{(0)}_0||_{\C^2}\,e^{i\theta}\big)^{p-1}\,\frac{1-||a^{(0)}_0||^2_{\C^2}}{||a^{(0)}_0||_{\C^2}}e^{i\theta},\quad (p>0).
\ee
Since $C_3=0$, by (\ref{cm21}) and lemma \ref{lemacs0}, $C_l\neq0$ only if $l=2^p$, ($p\geq 0$). From (\ref{c2}) and (\ref{c2n0c30-3}),
$$
C_{2^p} = -\frac{B}{\overline{C_1}}\,\<a^{(0)}_p,a^{(0)}_0\>_{\C^2}=
\frac{B}{\overline{C_1}}\,\big(-||a^{(0)}_0||_{\C^2}\,e^{i\theta}\big)^{p}\,(1-||a^{(0)}_0||^2_{\C^2}),\quad (p\in\N).
$$
This case corresponds with the families of {\it type 4} in propositions \ref{prop22} and \ref{prop23} below, where $\rho=||a^{(0)}_0||_{\C^2}$, $a^{(0)}_0=\rho u_0$, $a^{(1)}_0=u_1$ and $\theta=\text{arg}(\gamma_1)$.

Finally, when $C_2\neq0$ and $C_3\neq0$, equating the expressions for $C_3\overline{C_{2^p}}$ obtained using (\ref{c2}), (\ref{c21}) and (\ref{c21c2}) with $1=p'\leq p$, i.e., by (\ref{c21c2x}) with $1=p'\leq p$,
\be\label{c2n0c3n0-0}
a^{(0)}_{p}\perp \frac{\<a^{(1)}_0,a^{(0)}_{0}\>_{\C^2}}{1-||a^{(0)}_0||^2_{\C^2}}a^{(0)}_{0}+a^{(1)}_0\neq0,
\quad (p>0), 
\ee
and equating the expressions for $C_{2^{p'}+1}\overline{C_{2}}$ obtained using (\ref{c2}), (\ref{c21}) and (\ref{c21c2}) with $1=p\leq p'$, i.e., by (\ref{c21c2x}) with $1=p\leq p'$,
\be\label{c2n0c3n0-1}
a^{(1)}_0\perp \frac{\<a^{(0)}_1,a^{(0)}_{0}\>_{\C^2}}{1-||a^{(0)}_0||^2_{\C^2}}a^{(0)}_{0}+a^{(0)}_1\neq0, 
\ee
\be\label{c2n0c3n0-2}
\frac{\<a^{(1)}_{p'},a^{(0)}_0\>_{\C^2}\,\<a^{(0)}_0,a^{(0)}_{1}\>_{\C^2}}{1-||a^{(0)}_0||^2_{\C^2}}=-\<a^{(1)}_{p'},a^{(0)}_1\>_{\C^2}-\<a^{(1)}_{p'-1},a^{(0)}_0\>_{\C^2},
\quad (p'>0). 
\ee
Recall that $\frac{\<a^{(0)}_1,a^{(0)}_{0}\>_{\C^2}}{1-||a^{(0)}_0||^2_{\C^2}}a^{(0)}_{0}+a^{(0)}_1\neq0$ due to (\ref{c2n0-1}). 
Then, by (\ref{c3n0-3}) and (\ref{c2n0c3n0-0}), there exist three unitary vectors $u_0,u_1,v\in\C^2$ and two sequences $\{\rho_p\}_{p\geq0}$ and $\{\tau_p\}_{p\geq0}$ of complex numbers such that  
\be\label{c2n0c3n0-3}
a^{(0)}_0=\rho_0 u_0,\quad a^{(1)}_0=\tau_0 u_1,\quad
a^{(0)}_p=\rho_p v,\quad a^{(1)}_p=\tau_p v,\quad (p>1).
\ee
Moreover, $0<|\rho_0|=||a^{(0)}_0||^2_{\C^2}<1$, $\rho_1\neq0$ (because $C_2\neq0$), $\tau_0\neq0$ (since $C_3\neq0$). Also, 
\be\label{c2n0c3n0-4}
\<v,u_0\>_{\C^2}\neq 0, \text{ [by (\ref{c2n0-1})]};\quad 
\<u_1,u_0\>_{\C^2}\neq 0, \text{ [by (\ref{c3n0-1})]};
\ee
\be\label{c2n0c3n0-5}
\frac{|\rho_0|^2}{1-|\rho_0|^2}\,\<u_0,u_1\>_{\C^2}\,\<v,u_0\>_{\C^2}+\<v,u_1\>_{\C^2}=0, \text{ [by (\ref{c3n0-2}) or (\ref{c2n0c3n0-1}) or (\ref{c2n0c3n0-2})]};
\ee
\be\label{c2n0c3n0-6}
\frac{|\rho_0|^2}{1-|\rho_0|^2}=-\frac{\<v,u_1\>_{\C^2}}{\<u_0,u_1\>_{\C^2}\,\<v,u_0\>_{\C^2}}, \text{ [by (\ref{c2n0c3n0-4}) and (\ref{c2n0c3n0-5})]}
\quad\Rightarrow\quad \<v,u_1\>_{\C^2}\neq 0.
\ee
Taking into account that the sequences $\{\rho_p\}_{p\geq1}$ and $\{\tau_p\}_{p\geq0}$ satisfy the respective recurrence relations (\ref{c2n0-2}) and (\ref{c2n0c3n0-2}), and that both relations coincide,
\be\label{c2n0c3n0-7}
\tau_1=-\frac{\overline{\rho_0}}{\overline{\rho_1}}\frac{\<u_1,u_0\>_{\C^2}}{1+\frac{|\rho_0|^2}{1-|\rho_0|^2}\,|\<v,u_0\>_{\C^2}|^2}\tau_0,\ee
\be\label{c2n0c3n0-8}
\rho_p=r^{p-1}\rho_1,\quad  
\tau_p=r^{p-1}\tau_1,\quad (p>1);\quad
r=-\frac{\overline{\rho_0}}{\overline{\rho_1}}\frac{\<v,u_0\>_{\C^2}}{1+\frac{|\rho_0|^2}{1-|\rho_0|^2}\,|\<v,u_0\>_{\C^2}|^2}.  
\ee
The conditions (\ref{a0a1}) lead to
\be\label{c2n0c3n0-9}
1=|\rho_0|^2+|\rho_1|^2\frac{1}{1-|r|^2},\quad
1=|\tau_0|^2+|\tau_1|^2\frac{1}{1-|r|^2},
\ee
\be\label{c2n0c3n0-10}
0=\rho_0\<u_0,v\>_{\C^2}+\rho_1\frac{\overline{r}}{1-|r|^2},\quad
0=\tau_0\<u_1,v\>_{\C^2}+\tau_1\frac{\overline{r}}{1-|r|^2},\quad
0=\rho_0\overline{\tau_0}\<u_0,u_1\>_{\C^2}+\rho_1\overline{\tau_1}\frac{1}{1-|r|^2}.
\ee
From (\ref{c2n0c3n0-7})--(\ref{c2n0c3n0-10}),
\be\label{c2n0c3n0-11}
|\rho_1|^2=\frac{1-|\rho_0|^2}{1+\frac{|\rho_0|^2}{1-|\rho_0|^2}\,|\<v,u_0\>_{\C^2}|^2},\quad
|\tau_0|^2=\frac{1}{1+\frac{|\rho_0|^2}{1-|\rho_0|^2}\,|\<u_1,u_0\>_{\C^2}|^2},\quad
\tau_1=\tau_0\frac{\rho_1 \<u_1,v\>_{\C^2}}{\rho_0 \<u_0,v\>_{\C^2}},  
\ee
and the interesting additional relations
$$
|r|^2=\frac{\frac{|\rho_0|^2}{1-|\rho_0|^2}\,|\<v,u_0\>_{\C^2}|^2}{1+\frac{|\rho_0|^2}{1-|\rho_0|^2}\,|\<v,u_0\>_{\C^2}|^2},\quad
\frac{|\rho_0|^2}{1-|\rho_0|^2}\,|\<v,u_0\>_{\C^2}|^2= \frac{|r|^2}{1-|r|^2},\quad
\frac{|\tau_1|^2}{|\rho_1|^2}=\frac{1-|\tau_0|^2}{1-|\rho_0|^2}.
$$
From (\ref{c2}), (\ref{c21}) and (\ref{c2n0c3n0-7})--(\ref{c2n0c3n0-10}),
\be\label{c2n0c3n0-12}
\begin{array}{ll}
\begin{array}{l}
\ds C_{2^p}=\frac{B}{\overline{C_1}}\, (1-|\rho_0|^2)\,r^p,
\\
\ds C_{2^p+1}=\frac{C_1\overline{B_0}}{\overline{C_1}}\,\frac{\tau_1}{\rho_1}\, (1-|\rho_0|^2)\,r^{p-1}=
\frac{C_1\overline{B_0}}{\overline{C_1}}\,\frac{\tau_0\<u_1,v\>_{\C^2}}{\rho_0\<u_0,v\>_{\C^2}}\, (1-|\rho_0|^2)\,r^{p-1},
\end{array}
& (p>0).
\end{array}
\ee
In this case, the families of $M^+$-inner ($2\times 2$)-matrix functions
$A^+(\om)$ and Hardy functions $\{{\mf h}_l\}_{l\in\N\cup\{0\}}$ of $H^+_{\C^2}$ are those of {\it type 5} in propositions \ref{prop22} and \ref{prop23} below.

\medskip

Let us remember that the condition $||a^{(0)}_0||_{\C^2}<1$ in proposition \ref{pclc2a} restricts the attention to $M^+$-inner matrices $A^+$ with initial subspace $C$ of dimension two, i.e., such that $A^+(\om)$ is unitary for a.e. $\om\in\partial\D$. On the other hand, as we have seen at the beginning of this section, when dimension of $C$ is one, the only feasible pair $\{{\mf h}_0,{\mf h}_1\}$ leading to a tight wavelet frame, with frame bound $B$, must satisfy ${\mf h}_0=\Psi_1=Ba^{(0)}_0$ and ${\mf h}_1=0$, with $||\Psi_1||_{\C^2}=B$ or $||a^{(0)}_0||_{\C^2}=1$, so that one gets a trivial Haar wavelet frame.
The corresponding $M^+$-inner matrix $A^+$ is of the form
$$
A^+(\om)= \left(\begin{array}{cc} {\mf a}^{(0)}(\om) & {\mf a}^{(1)}(\om)\end{array}\right)=
\left(\begin{array}{cc} a^{(0)}_0 & 0\end{array}\right)\,.
$$
Including this case as ``Type 1'', we have proved the following:

\begin{prop}\label{prop22}
There are five types of families of $M^+$-inner ($2\times 2$)-matrix functions
$$
A^+(\om)= \left(\begin{array}{cc} {\mf a}^{(0)}(\om) & {\mf a}^{(1)}(\om)\end{array}\right)\,,\quad\text{ where } {\mf a}^{(0)},{\mf a}^{(1)}\in H^+_{\C^2}\,,
$$
satisfying conditions (\ref{cm2x})--(\ref{c21c2x}). They are as follows: 
\begin{itemize}
\item[Type 1.]
Given $u_0\in\C^2$, such that $||u_0||_{\C^2}=1$,
$$
\left\{\begin{array}{l}
\ds {\mf a}^{(0)}(\om)=u_0\,,
\\
{\mf a}^{(1)}(\om)=0\,.
\end{array}\right.
$$
\item[Type 2.]
Given an orthonormal basis $\{u_0,u_1\}$ of $\C^2$, 
$$
\left\{\begin{array}{l}
\ds {\mf a}^{(0)}(\om)=\om\,u_0\,,
\\
{\mf a}^{(1)}(\om)=u_1\,.
\end{array}\right.
$$

\item[Type 3.]
Given an orthonormal basis $\{u_0,u_1\}$ of $\C^2$, $0<\rho<1$ and $\theta\in\R$,
$$
\left\{\begin{array}{l}
\ds {\mf a}^{(0)}(\om)=\rho u_0+\om(1-\rho^2)^{1/2}u_1\,,
\\
\ds {\mf a}^{(1)}(\om)=e^{i\theta}[(1-\rho^2)^{1/2}u_0-\om\,\rho u_1]\,.
\end{array}\right.
$$

\item[Type 4.]
Given an orthonormal basis $\{u_0,u_1\}$ of $\C^2$, $0<\rho<1$ and $\theta\in\R$, 
$$
\left\{\begin{array}{l}
\ds {\mf a}^{(0)}(\om)=\big(\rho+(1-\rho^2)e^{i\theta}\sum_{k=1}^\infty\om^k\,\big(-\rho\,e^{i\theta}\big)^{k-1}\big)\,u_0\,,
\\
\ds {\mf a}^{(1)}(\om)=u_1\,.
\end{array}\right.
$$

\item[Type 5.]
Given $0<|\rho_0|<1$, choose three unitary vectors $u_0$, $u_1$ and $v$ in $\C^2$ such that (\ref{c2n0c3n0-5}) is satisfied.\footnote{When the coordinates of  $u_0$, $u_1$ and $v$ are real, (\ref{c2n0c3n0-5}) is equivalent to 
$$
\frac{1}{1-|\rho_0|^2}=\tan(\widehat{vu_0})\tan(\widehat{u_0u_1}),
$$
where $\widehat{ab}$ denotes the angle from $a$ to $b$ as vectors in $\R^2$.
}
Then, $|\rho_1|$ and $|\tau_0|$ are given by (\ref{c2n0c3n0-11}). Once the free arguments for $\rho_0$, $\rho_1$ and $\tau_0$ have been selected, say $\theta_{\rho_0}$, $\theta_{\rho_1}$ and $\theta_{\tau_0}$, the value of $r$ is determined by (\ref{c2n0c3n0-8}) and the value of $\tau_1$ is given by (\ref{c2n0c3n0-11}). Then,
$$
\left\{\begin{array}{l}
\ds {\mf a}^{(0)}(\om)=\rho_0u_0+\rho_1v\sum_{k=1}^\infty\om^k\,r^{k-1}\,,
\\
\ds {\mf a}^{(1)}(\om)=\tau_0u_1+\tau_1v\sum_{k=1}^\infty\om^k\,r^{k-1}\,.
\end{array}\right.
$$
\end{itemize}
\end{prop}

In terms of the families $\{{\mf h}_l\}_{l\in\N\cup\{0\}}\subset H^+_{\C^2}$ the result reads as follows:

\begin{prop}\label{prop23}
There are five types of families $\{{\mf h}_l\}_{l\in\N\cup\{0\}}\subset H^+_{\C^2}$ satisfying the complete set of conditions (\ref{h1}).
They are as follows:
\begin{itemize}
\item[Type 1.]
Given $u_0\in\C^2$, such that $||u_0||_{\C^2}=1$, and $B_0\in\C$, with $|B_0|^2=B$,
$$
\left\{\begin{array}{l}
\ds {\mf h}_0(\om)=\Psi_{1}=B_0\,u_0\,,
\\
{\mf h}_l(\om)=0\,,\text{ for } l\geq1\,.
\end{array}\right.
$$

\item[Type 2.]
Given an orthonormal basis $\{u_0,u_1\}$ of $\C^2$ and a pair of constants $B_0,C_1\in\C$, with $|B_0|^2=|C_1|^2=B$,
$$
\left\{\begin{array}{l}
{\mf h}_0(\om)=\om\,\Psi_{2}=\om\,{B_0}\,u_0\,,
\\
{\mf h}_1(\om)=\Psi_{3}={C_1}\,u_1\,,
\\
{\mf h}_l(\om)=0\,,\text{ for } l\geq2.
\end{array}\right.
$$

\item[Type 3.]
Given an orthonormal basis $\{u_0,u_1\}$ of $\C^2$, $0<\rho<1$, constants $B_0,C_1\in\C$ such that $|B_0|^2=B$, $|C_1|^2=B(1-\rho^2)$, and $\theta\in\R$,
$$
\left\{\begin{array}{l}
\ds {\mf h}_0(\om)=\Psi_1+\om\,\Psi_{2}=
{B_0}(\rho u_0+\om(1-\rho^2)^{1/2}u_1),
\\
\ds {\mf h}_1(\om)=\Psi_{3}+\om\,\Psi_5=
C_1e^{i\theta}[(1-\rho^2)^{1/2}u_0-\om\,\rho u_1],
\\
\begin{array}{ll}
\ds {\mf h}_{2^p-1}(\om)&\ds =\Psi_{2^{p}+2^p-1}+\om\,\Psi_{2^{p+1}+2^p-1}=
\\ &\ds =\big(-\frac{\overline{B_0}\rho(1-\rho^2)^{1/2}e^{i\theta}}{\overline{C_1}}\big)^{p-1}C_{1}e^{i\theta}[(1-\rho^2)^{1/2}u_0-\om\,\rho u_1],\quad (p>1),
\end{array}
\\
\ds {\mf h}_{l}(\om)=0,\quad \text{else}.
\end{array}\right.
$$

\item[Type 4.]
Given an orthonormal basis $\{u_0,u_1\}$ of $\C^2$, $0<\rho<1$, constants $B_0,C_1\in\C$ such that $|B_0|^2=B$, $|C_1|^2=B(1-\rho^2)$, and $\theta\in\R$,
$$
\left\{\begin{array}{l}
\ds {\mf h}_0(\om)=\sum_{k=0}^\infty \om^k\,\Psi_{2^k}=
{B_0}\big(1+\frac{1-\rho^2}{\rho}e^{i\theta}\sum_{k=1}^\infty\om^k\,\big(-\rho\,e^{i\theta}\big)^{k-1}\big)\,\rho u_0,
\\
\ds {\mf h}_1(\om)=\Psi_{3}=
C_1u_1,
\\
\ds {\mf h}_{2^p}(\om)=\Psi_{2^{p+1}+2^p}=\frac{B}{\overline{C_1}}\,\big(-\rho\,e^{i\theta}\big)^{p}\,(1-\rho^2)\,u_1,\quad (p>0),
\\
\ds {\mf h}_{l}(\om)=0,\quad \text{else}.
\end{array}\right.
$$

\item[Type 5.]
Given $0<|\rho_0|<1$, choose three unitary vectors $u_0$, $u_1$ and $v$ in $\C^2$ such that (\ref{c2n0c3n0-5}) is satisfied.
Then, $|\rho_1|$ and $|\tau_0|$ are given by (\ref{c2n0c3n0-11}). Once the free arguments for $\rho_0$, $\rho_1$ and $\tau_0$ have been selected, say $\theta_{\rho_0}$, $\theta_{\rho_1}$ and $\theta_{\tau_0}$, the value of $r$ is determined by (\ref{c2n0c3n0-8}) and the value of $\tau_1$ is given by (\ref{c2n0c3n0-11}). Choose constants $B_0,C_1\in\C$ such that $|B_0|^2=B$, $|C_1|^2=B(1-\rho_0^2)$, and select their free arguments $\theta_{B_0}$ and $\theta_{C_1}$. Then,
$$
\left\{\begin{array}{l}
\ds {\mf h}_0(\om)=\sum_{k=0}^\infty \om^k\,\Psi_{2^k}=
{B_0}\big(\rho_0u_0+\rho_1v\sum_{k=1}^\infty\om^k\,r^{k-1}\big),
\\
\ds {\mf h}_1(\om)=\sum_{k=0}^\infty \om^k\,\Psi_{2^{k+1}+1}=
{C_1}\big(\tau_0u_1+\tau_1v\sum_{k=1}^\infty\om^k\,r^{k-1}\big),
\\
\ds {\mf h}_{l}(\om)=\sum_{k=0}^\infty \om^{k}\,\Psi_{2^{p+k+1}+l}=\frac{C_l}{C_1}\,{\mf h}_1(\om),\quad (l\geq1,\,l=2^p+\sum_{t=0}^{p-1}l_t2^t),
\end{array}\right.
$$
where $C_{2^p}$ and $C_{2^p+1}$, ($p>0$), are given by (\ref{c2n0c3n0-12}), and the other $C_l$'s are calculated using (\ref{clg}). 
\end{itemize}
\end{prop}

As commented in the Introduction, for type 1 functions $A^+(\om)$, both functions $\psi_1$ and $\psi_2$ are proportional to the Haar wavelet, and type 3 functions $A^+(\om)$ lead to the reflected version of functions $\psi_1$ and $\psi_2$ obtained from  type 4 functions $A^+(\om)$ with the same parameters. For types 2--5 functions $A^+(\om)$, some examples of real functions $\psi_1$ and $\psi_2$ are shown in figures \ref{fig1}--\ref{fig4}.

Recall that the $M^+$-inner matrix function $A^+(\om)$ and the $M^+$-wandering subspace generated by $\{{\mf h}_0,{\mf h}_1\}$ in $H^+_{\C^2}$ are connected by  Halmos's lemma \ref{lwr}.  
There, the  subspace generated by $\{{\mf h}_0,{\mf h}_1\}$ uniquely determines $A^+(\om)$ to within a constant partially isometric factor on the right. In particular, in the non-trivial types 2--5 of propositions \ref{prop22} and \ref{prop23}, to within a constant unitary factor on the right.

\begin{prop}\label{prop24}
Let $A^+(\om)= \left(\begin{array}{cc} {\mf a}^{(0)}(\om) & {\mf a}^{(1)}(\om)\end{array}\right)$
be an $M^+$-inner ($2\times 2$)-matrix function, unitary for a.e. $\om\in\partial\D$, with ${\mf a}^{(0)},{\mf a}^{(1)}\in H^+_{\C^2}$ satisfying conditions (\ref{cm2x})--(\ref{c21c2x}). Given an arbitrary constant unitary ($2\times 2$)-matrix $U$, consider the  ($2\times 2$)-matrix function $B^+(\om)$ defined by 
$$
B^{+}(\om) = \left( \begin{array}{cc} {\mf b}^{(0)}(\om) & {\mf b}^{(1)}(\om) \end{array} \right) := A^{+}(\om) \cdot U\,.
$$
Then, $B^+(\om)$ is also an $M^+$-inner ($2\times 2$)-matrix function, unitary for a.e. $\om\in\partial\D$, and such that ${\mf b}^{(0)}(\om), \, {\mf b}^{(1)}(\om)$ verify conditions (\ref{cm2x})--(\ref{c21c2x}).
\end{prop}

\begin{proof}
Conditions (\ref{cm2x})--(\ref{c21c2x}) for ${\mf a}^{(0)},{\mf a}^{(1)}$ are given in terms of their Fourier-Taylor coefficients: 
$$
{\mf a}^{(0)}(\om)=\sum_{k=0}^\infty \om^k\,a^{(0)}_k\,,\quad
{\mf a}^{(1)}(\om)=\sum_{k=0}^\infty \om^k\,a^{(1)}_k\,.
$$
On the other hand, the constant matrix $U= \left( \begin{array}{cc} u_{11} & u_{12}\\ u_{21}& u_{22}\end{array} \right)$ is unitary if and only if 
\begin{equation}\label{CondMatUnit}
\begin{array}{c}
|u_{12}|^2 = |u_{21}|^2 = 1 - |u_{11}|^2 = 1 - |u_{22}|^2\\[2ex]
\theta_{11} - \theta_{12} = \theta_{21} - \theta_{22} + \pi \text{  mod} (2\pi)
\end{array}
\end{equation}
with $u_{jk} = |u_{jk}| e^{i\theta_{jk}}, \, j,k\in \{1,2\}$.
Using (\ref{CondMatUnit}), direct calculations show that ${\mf b}^{(0)}(\om), \, {\mf b}^{(1)}(\om)$ verify conditions (\ref{cm2x})--(\ref{c21c2x}) if and only if 
\begin{equation}\label{ConditionChangeCase}
\big( 1 - \| a^{(0)}_{0} \|_{\C^2}^{2} \big)\, \Big| u_{11} - u_{21} \frac{\< a^{(1)}_{0}, a^{(0)}_{0} \>}{1- \| a^{(0)}_{0} \|_{\C^2}^{2}} \Big|^{2} = 1 - \| a^{(0)}_{0} u_{11} + a^{(1)}_{0} u_{21} \|_{\C^2}^{2}\,.
\end{equation}
Finally, it is easy to see that, given $A^{+}(\om) = \big( a^{(0)}(\om) \quad a^{(1)}(\om) \big)$ in any of the types 2--5 of proposition \ref{prop22} and an arbitrary constant unitary matrix $U$, condition (\ref{ConditionChangeCase}) is always satisfied.
\end{proof}

In other words, proposition \ref{prop24} asserts that, given an $M^+$-inner ($2\times 2$)-matrix function $A^+(\om)$ in the non-trivial types 2--5 of proposition \ref{prop22}, $B^{+}(\om) =  A^{+}(\om) \cdot U$ is also an $M^+$-inner ($2\times 2$)-matrix function in the non-trivial types 2--5 of proposition \ref{prop22}, for every constant unitary ($2\times 2$)-matrix $U$.
These transformations connect $M^+$-inner matrix functions in types 2 and 3 on the one hand (those with a finite number of non-zero Fourier-Taylor coefficients) and $M^+$-inner matrix functions in types 4 and 5 on the other hand (those with an infinite number of non-zero Fourier-Taylor coefficients).
To be precise:
\begin{itemize}
\item[(i)] 
Starting from a type 2 matrix function $A^{+}(\om)$, where  
$$
\left\{\begin{array}{l}
\ds {\mf a}^{(0)}(\om)=\om\,u_0\,,
\\
{\mf a}^{(1)}(\om)=u_1\,,
\end{array}\right.
$$
for $B^{+}(\om)=A^{+}(\om)\cdot U$ one has:
\begin{enumerate}
\item{
If $U$ is a diagonal unitary matrix, i.e., $u_{12}=u_{21}=0$ and $|u_{11}|=| u_{22} |= 1$, then $B^{+}(\om)$ is also a type 2 matrix function:
$$
\left\{\begin{array}{l}
\ds {\mf b}^{(0)}(\om)=\om\,u_0'\,,
\\
{\mf b}^{(1)}(\om)=u_1'\,,
\end{array}\right.
$$
where $u_0' = e^{i \theta_{11}} \, u_0, \, u_1' = e^{i \theta_{22}} \, u_1$.

\item
If $U$ is not diagonal, that is, $|u_{12}| = |u_{21}| > 0$, then $B^{+}(\om)$ is a type 3 matrix function:
$$
\left\{\begin{array}{l}
{\mf b}^{(0)}(\om) = \rho u_0' + \om (1 - \rho^{2})^{1/2} u_1'\,,
\\
{\mf b}^{(1)}(\om) = e^{i \theta} (1 - \rho^{2})^{1/2} u_0' - \om e^{i \theta} \rho u_1'\,,
\end{array}\right.
$$
where $u_0' = e^{i \theta_{21}} \, u_1, \, u_1' = e^{i \theta_{11}} \, u_0$, $\rho = |u_{12}|$ and $\theta = \theta_{22} - \theta_{21} = \theta_{12} - \theta_{11} + \pi$.}
\end{enumerate}

\item[(ii)]
Starting from a type 4 matrix function $A^{+}(\om)$, where  
$$
\left\{\begin{array}{l}
\ds{\mf a}^{(0)}(\om) = \rho \, u_0 + \sum_{k = 1}^{\infty} \, \om^{k} \, e^{i \theta} \, (1 - \rho^{2})^{1/2} \, (-\rho \, e^{i \theta})^{k-1} \, u_0\,,
\cr
\ds{\mf a}^{(1)}(\om) = u_1\,,
\end{array}\right.
$$
for $B^{+}(\om)=A^{+}(\om)\cdot U$ one has:
\begin{enumerate}
\item
If $U$ is a diagonal unitary matrix, then $B^{+}(\om)$ is also a type 4 matrix function:
$$
\left\{\begin{array}{l}
\ds{\mf b}^{(0)}(\om) = \rho \, u_0' + \sum_{k = 1}^{\infty} \, \om^{k} \, e^{i \theta} \, (1 - \rho^{2})^{1/2} \, (-\rho \, e^{i \theta})^{k-1} \, u_0'\,,
\cr
\ds{\mf b}^{(1)}(\om) = u_1'\,,
\end{array}\right.
$$
where $u_0' = e^{i \theta_{11}} \, u_0, \, u_1' = e^{i \theta_{22}} \, u_1$.

\item
If $U$ is not diagonal, then $B^{+}(\om)$ is a type 5 matrix function:
$$
\left\{\begin{array}{l}
\ds{\mf b}^{(0)}(\om) = \rho_{0} \, u_0' + \rho_{1} \, v' \, \sum_{k = 1}^{\infty} \, r^{k-1} \, \om^{k}\,,
\cr
\ds{\mf b}^{(1)}(\om) = \tau_{0} \, u_1' + \tau_{1} \, v' \, \sum_{k = 1}^{\infty} \, r^{k-1} \, \om^{k}\,,
\end{array}\right.
$$
where 
$$
\left\{\begin{array}{l}
\ds u_0' = \frac{e^{i \theta_{11}} \, |u_{11}| \, \rho}{\rho_{0}} \, u_0 + \frac{e^{i \theta_{21}} \, |u_{12}|}{\rho_{0}} \, u_1,
\quad
u_1' = \frac{e^{i \theta_{12}} \, |u_{12}| \, \rho}{\tau_{0}} \, u_0 + \frac{e^{i \theta_{22}} \, |u_{11}|}{\tau_{0}} \, u_1,
\cr
\ds v' = \frac{e^{i(\theta_{11} + \theta)} \, |u_{11}| \, (1 - \rho^{2})}{\rho_{1}} \, u_0,\cr
\ds\rho_{0} = (|u_{11}|^{2} \rho^{2} + |u_{12}|^{2})^{1/2} \, e^{i\theta_{\rho_0}},
\quad
\rho_{1} =  |u_{11}| \, \, (1 - \rho^{2}) \, e^{i\theta_{\rho_1}},
\cr
\ds\tau_{0} = (|u_{12}|^{2} \rho^{2} + |u_{11}|^{2})^{1/2} \, e^{i\theta_{\tau_0}},
\quad 
\tau_{1} = |u_{12}| \, (1 - \rho^{2}) \, e^{i(\theta_{\rho_1} + \theta_{12} - \theta_{11})},
\cr
r = - \rho e^{i \theta},
\cr
\theta_{\rho_0},\, \theta_{\rho_1},\, \theta_{\tau_0} \in \R.
\end{array}\right.
$$
\end{enumerate}

\end{itemize}
These relationships exhaust the four family types 2--5 of proposition \ref{prop22}.

%------------------------

\section*{Acknowledgements}
This work was partially supported by research projects MTM2012-31439 and MTM2014-57129-C2-1-P (Secretar\'{\i}a General de Ciencia, Tecnolog\'{\i}a e Innovaci\'on, Ministerio de Econom\'{\i}a y Competitividad, Spain).

%-----------------------------

%\appendix

\section*{Appendix: Haar bases}\label{ap2}

Let $\varphi^H$ and $\psi^H$ be the {\it Haar scaling function and wavelet} given by
$$
\varphi^H:=\chi_{[0,1)}\,,\quad \psi^H:=\chi_{[0,1/2)}-\chi_{[1/2,1)}\,, 
$$
where $\chi$ denotes the characteristic function:
$$
\chi_{[a,b)}(x):=\left\{\begin{array}{ll}1,& \text{ if }x\in[a,b),\\ 0,& \text{ otherwise,}\end{array}\right.
$$
and consider the {\it ``Haar ONBs"} $\{L_{i}^{(0)}(x)\}_{i\in\N\cup\{0\}}$ of $L^2[0,1)$ and $\{K_{\pm,j}^{(0)}(x)\}_{j\in\N\cup\{0\}}$ of $L^2[\pm 1,\pm 2)$ defined by
\be\label{hbb0}
\begin{array}{l}
L_{0}^{(0)}:=\varphi^H_{0,0}\,;\\
L_{2^p+q}^{(0)}:=\psi^H_{p,q}\,, \quad \text {for }p=0,1,\,\dots;\, q=0,1,\ldots,2^{p}-1\,;
\\[1ex]
K_{+,0}^{(0)}:=\varphi^H_{0,1}\,;\\
K_{+,2^p+q}^{(0)}:=\psi^H_{p,2^p+q}\,, \quad \text {for }p=0,1,\,\dots;\,q=0,1,\ldots,2^{p}-1\,;
\\[1ex]
K_{-,0}^{(0)}:=\varphi^H_{0,-2}\,;\\
K_{-,2^p+q}^{(0)}:=\psi^H_{p,-2^{p+1}+q}\,, \quad \text {for }p=0,1,\,\dots;\,q=0,1,\ldots,2^{p}-1\,.
\end{array}
\ee
In accordance with (\ref{lt}) and (\ref{kd}), for each $n,m\in\Z$,
\be\label{hbb}
\begin{array}{l}
L_{0}^{(n)}=\varphi^H_{0,n}\,,\\
L_{2^p+q}^{(n)}=\psi^H_{p,q+2^p n}\,, \quad \text {for }p=0,1,\,\dots;\, q=0,1,\ldots,2^{p}-1\,;
\\[1ex]
K_{+,0}^{(m)}:=\varphi^H_{m,1}\,;\\
K_{+,2^p+q}^{(m)}:=\psi^H_{p+m,2^p+q}\,, \quad \text {for }p=0,1,\,\dots;\,q=0,1,\ldots,2^{p}-1\,;
\\[1ex]
K_{-,0}^{(m)}:=\varphi^H_{m,-2}\,;\\
K_{-,2^p+q}^{(m)}:=\psi^H_{p+m,-2^{p+1}+q}\,, \quad \text {for }p=0,1,\,\dots;\,q=0,1,\ldots,2^{p}-1\,.
\end{array}
\ee
For these bases the change of representation matrix $\big(\alpha_{i,n}^{s,j,m}\big)$, defined by (\ref{chm}), is as follows:

\smallskip\noindent
For $n=0$, $i=0$, 
$$
\alpha_{0,0}^{s,j,m}=\left\{
\begin{array}{ll}
\ds 2^{-m/2},& \text{ if } s=+,\,\,j=0,\,\, m>0,\\
\ds 0, &  \text{ otherwise.}
\end{array}
\right.
$$
For $n=0$, $r=0,1,2,\ldots$,
$$
\alpha_{2^r,0}^{s,j,m}=\left\{
\begin{array}{ll}
\ds -2^{-1/2},& \text{ if } s=+,\,\, j=0,\,\, m=r+1,\\
\ds 2^{(r-m)/2},& \text{ if } s=+,\,\, j=0,\,\, m>r+1,\\
\ds 0, &  \text{ otherwise.}
\end{array}
\right.
$$
For $n=0$, $r=0,1,2,\ldots$ and $t=2^p+q$ (with $0\leq p<r$ and $q=0,1,\ldots,2^p-1$),
$$
\alpha_{2^r+t,0}^{s,j,m}=\left\{
\begin{array}{ll}
\ds 1,& \text{ if } s=+,\,\, j=t,\,\, m=r-p,\\
\ds 0, &  \text{ otherwise.}
\end{array}
\right.
$$
For $n=1$ and every $i\geq0$,
$$
\alpha_{i,1}^{s,j,m}=\left\{
\begin{array}{ll}
\ds 1,& \text{ if } s=+,\,\, j=i,\,\, m=0,\\
\ds 0, &  \text{ otherwise.}
\end{array}
\right.
$$
For $n>1$ and $i>0$, with $n=2^u+v$ ($u=1,2,\ldots$, $v=0,1,\ldots,2^u-1$) and $i=2^r+t$ ($r=0,1,\ldots$, $t=0,1,\ldots,2^r-1$), 
$$
\alpha_{2^r+t,2^u+v}^{s,j,m}=\left\{
\begin{array}{ll}
\ds 1,& \text{ if } s=+,\,\, j=2^r(2^u+v)+t,\,\, m=-u,\\
\ds 0, &  \text{ otherwise.}
\end{array}
\right.
$$
For $n>1$ and $i=0$, with $n=2^u+v$ ($u=1,2,\ldots$, $v=0,1,\ldots,2^u-1$), 
$$
\alpha_{0,2^u+v}^{s,j,m}=\left\{
\begin{array}{ll}
\ds 2^{-u/2},& \text{ if } s=+,\,\, j=0,\,\, m=-u,\\
\ds (-1)^{w(u,v,p)}\, 2^{(p-u)/2},& \text{ if } s=+,\,\, j=2^p+\left[{v}/{2^{u-p}}\right]\\
&\text{ for }0\leq p<u,\,\, m=-u,\\
\ds 0, &  \text{ otherwise,}
\end{array}
\right.
$$
where $[\cdot]$ denotes ``entire part of" and\footnote{Taking into account the binary expression $v=\sum_{k=0}^{u-1} t_k\,2^k$, with $t_k=0$ or $1$, one has $w(u,v,p)=t_{u-p-1}$ and $\left[{v}/{2^{u-p}}\right]=\sum_{k=u-p}^{u-1} t_k\,2^{k-(u-p)}$.}
\be\label{apph1}
w(u,v,p)=\left[\frac{v-2^{u-p}[v/2^{u-p}]}{2^{u-p-1}}\right].
\ee
For $n=-1$, $i=0$, 
$$
\alpha_{0,-1}^{s,j,m}=\left\{
\begin{array}{ll}
\ds 2^{-m/2},& \text{ if } s=-,\,\,j=0,\,\, m>0,\\
\ds 0, &  \text{ otherwise.}
\end{array}
\right.
$$
For $n=-1$, $r=0,1,2,\ldots$,
$$
\alpha_{2^{r+1}-1,-1}^{s,j,m}=\left\{
\begin{array}{ll}
\ds 2^{-1/2},& \text{ if } s=-,\,\, j=0,\,\, m=r+1,\\
\ds -2^{(r-m)/2},& \text{ if } s=-,\,\, j=0,\,\, m>r+1,\\
\ds 0, &  \text{ otherwise.}
\end{array}
\right.
$$
For $n=-1$, $r=1,2,\ldots$, $0\leq p<r$ and $q=0,1,\ldots,2^p-1$,
$$
\alpha_{2^{r+1}-2^{p+1}+q,-1}^{s,j,m}=\left\{
\begin{array}{ll}
\ds 1,& \text{ if } s=-,\,\, j=2^p+q,\,\, m=r-p,\\
\ds 0, &  \text{ otherwise.}
\end{array}
\right.
$$
For $n=-2$ and every $i\geq0$,
$$
\alpha_{i,-2}^{s,j,m}=\left\{
\begin{array}{ll}
\ds 1,& \text{ if } s=-,\,\, j=i,\,\, m=0,\\
\ds 0, &  \text{ otherwise.}
\end{array}
\right.
$$
For $n<-2$ and $i>0$, with $n=-2^{u+1}+v$ ($u=1,2,\ldots$, $v=0,1,\ldots,2^u-1$) and $i=2^r+t$ ($r=0,1,\ldots$, $t=0,1,\ldots,2^r-1$), 
$$
\alpha_{2^r+t,-2^{u+1}+v}^{s,j,m}=\left\{
\begin{array}{ll}
\ds 1,& \text{ if } s=-,\,\, j=2^r(2^u+v)+t,\,\, m=-u,\\
\ds 0, &  \text{ otherwise.}
\end{array}
\right.
$$
For $n<-2$ and $i=0$, with $n=-2^{u+1}+v$ ($u=1,2,\ldots$, $v=0,1,\ldots,2^u-1$),
$$
\alpha_{0,-2^{u+1}+v}^{s,j,m}=\left\{
\begin{array}{ll}
\ds 2^{-u/2},& \text{ if } s=-,\,\, j=0,\,\, m=-u,\\
\ds (-1)^{w(u,v,p)}\, 2^{(p-u)/2},& \text{ if } s=-,\,\, j=2^p+\left[{v}/{2^{u-p}}\right]\\
&\text{ for }0\leq p<u,\,\, m=-u,\\
\ds 0, &  \text{ otherwise,}
\end{array}
\right.
$$
where $w(u,v,p)$ is given by (\ref{apph1}).

%%---------------------------------
%\bibliographystyle{acm}%{abbrv}{plain}{alpha}
%\bibliography{F:/Nani/FGC_wavelet}
%%---------------------------------

\end{document}